\documentclass[oneside,11pt,reqno]{amsart}
\usepackage{etex}
\usepackage{stmaryrd}
\usepackage{bbm}
\usepackage{mathrsfs}
\usepackage{enumerate}
\usepackage{latexsym,amsxtra}
\usepackage[dvips]{graphicx}
\usepackage{dsfont}
\usepackage{slashed}
\usepackage[all]{xy}
\usepackage{amscd,graphics}
\usepackage{amsmath,amsfonts,amsthm,amssymb}
\usepackage{latexsym,amsmath}
\usepackage{graphicx,psfrag}
\usepackage[usenames, dvipsnames]{color}
\usepackage{soul}
\usepackage{hyperref}
\usepackage[normalem]{ulem}
\usepackage{dcpic,pictex}

\newtheorem{thm}{Theorem}[section]
\newtheorem{lem}[thm]{Lemma}
\newtheorem{cor}[thm]{Corollary}
\newtheorem{pro}[thm]{Proposition}

\theoremstyle{definition}
\newtheorem{defi}[thm]{Definition}
\newtheorem{ex}[thm]{Example}

\theoremstyle{remark}
\newtheorem*{rmk}{Remark}

\DeclareMathOperator{\inv}{inv}
\DeclareMathOperator{\inti}{int}
\DeclareMathOperator{\im}{im}

\DeclareMathOperator{\morp}{mor}
\DeclareMathOperator{\ob}{ob}
\DeclareMathOperator{\grad}{grad}
\DeclareMathOperator{\spf}{sf}
\DeclareMathOperator{\gradtil}{\widetilde{grad}}

\DeclareMathOperator{\edim}{dim}
\DeclareMathOperator{\grade}{gr}
\newcommand{\stL}{\text{\sout{\ensuremath{\mathcal{L}}}} }

\DeclareFontFamily{U}{mathx}{\hyphenchar\font45}
\DeclareFontShape{U}{mathx}{m}{n}{<-> mathx10}{}
\DeclareSymbolFont{mathx}{U}{mathx}{m}{n}
\DeclareMathAccent{\widebar}{0}{mathx}{"73}

\title
{Unfolded Seiberg-Witten Floer spectra, I: Definition and invariance }

\author{Tirasan Khandhawit}              
\date{} 
\address{Kavli Institute for the Physics and Mathematics of the Universe (WPI),The University of Tokyo Institutes for Advanced Study, The University of Tokyo, Kashiwa, Chiba 277-8583, Japan}
\email{tirasan.khandhawit@ipmu.jp}

\author{Jianfeng Lin}
\address{Department of Mathematics 
University of California Los Angeles, 
Box 951555, Los Angeles, CA 90095-1555, USA }
\email{juliuslin@math.ucla.edu}

\author{Hirofumi Sasahira}
\address{Graduate School of Mathematics, Nagoya University,
Furocho, Chikusa-ku, Nagoya 464-8602,
Japan}
\email{hsasahira@math.nagoya-u.ac.jp}

\begin{document}

\vskip 0.3 truecm

\maketitle 

\begin{abstract}
Let $Y$ be a closed and oriented $3$-manifold. We define different versions of unfolded Seiberg-Witten Floer spectra for
$Y$. These invariants generalize Manolescu's Seiberg-Witten Floer spectrum for rational homology $3$-spheres. We also compute
some examples when $Y$ is a Seifert space.
  
\end{abstract}

\tableofcontents

\section {Introduction}   \label{section introduction}
 The Seiberg-Witten equations, introduced in \cite{Witten}, and related theories have been playing a central role in the study of smooth 4-dimensional manifolds since 1990s. Following the seminal work of Floer \cite{Floer}, Kronheimer and Mrowka \cite{Kronheimer-Mrowka} used Seiberg-Witten equations on 3-manifolds to construct monopole Floer homology. The monopole Floer homology and its counterparts are powerful invariants of 3-manifolds and became an important tool in the study of low-dimensional topology with many remarkable applications.

In contexts of symplectic Floer theory  and instanton Floer theory, Cohen, Jones and Segal \cite{Cohen-Jones-Segal} posed a question of constructing a ``Floer spectrum,'' an object whose homology recovers the Floer homology .  In 2003, Manolescu \cite{Manolescu1} first constructed the Seiberg-Witten Floer spectrum for rational homology 3-spheres by incorporating Furuta's technique of finite dimensional approximation in Seiberg-Witten theory \cite{Furuta108} and Conley index theory \cite{Conley}.    It has been just recently shown by Lidman and Manolescu \cite{Lidman-Manolescu} that the homology of this spectrum is isomorphic to the monopole Floer homology. In principle, the Seiberg-Witten Floer spectrum can be thought as a stable homotopy refinement
of Floer homology. For example, one can apply the K-theory functor to this spectrum and define ``Seiberg-Witten Floer K-theory''
as well as other generalized homology theories (see \cite{Manolescu2}, \cite{Furuta-Li} and \cite{Jianfeng} for applications
in this direction).

As monopole Floer homology is defined for general 3-manifolds, it is a natural question to extend Manolescu's construction to any 3-manifold $Y$ with $b_1(Y) >0$.  In the case $b_1(Y) =1 $, Kronheimer and Manolescu \cite{Kronheimer-Manolescu} constructed a periodic pro-spectrum for such $Y$ with nontorsion spin$^c$ structure. The first author \cite{Khandhawit2} gave an approach to construct Seiberg-Witten Floer spectrum for a general case.

The main goal of the current paper is to rigorously construct the ``unfolded'' version of Seiberg-Witten Floer spectrum for general 3-manifolds. Our invariants come with two variations: type-A invariant and type-R invariant. The letters ``A'' and ``R'' stand for attractor and repeller, which are notions in dynamical system and play a role in our construction.

\begin{thm}\label{main theorem 1}
Let $Y$ be a closed, oriented 3-manifold and let $\mathfrak{s}$ be a spin$^{c}$ structure on $Y$. Given a Riemannian metric $g$ on $Y$ and a spin$^{c}$ connection $A_{0}$ which induces a connection on the determinant bundle of the spinor bundle with harmonic curvature, we can define
$$ \underline{\textnormal{swf}}_{}^{A}(Y,\mathfrak{s},A_{0}, g ; S^1) \text{ and } \  \underline{\textnormal{swf}}^{R}_{}(Y,\mathfrak{s},A_{0}, g ; S^1) $$
as a direct system and an inverse system in the $S^1$-equivariant stable category. These objects are well-defined up to canonical isomorphisms in the corresponding categories.

{In the case that} $c_{1}(\mathfrak{s})$ is nontorsion and $l=\gcd\{(h\cup [c_{1}(\mathfrak{s})])[Y] \mid h \in H^1(Y;\mathbb{Z})\} $, the objects $\underline{\textnormal{swf}}_{}^{A}(Y,\mathfrak{s},A_{0}, g ; S^1)$ and
$\underline{\textnormal{swf}}_{}^{R}(Y,\mathfrak{s},A_{0}, g ; S^1)$ are $l$-periodic in the sense that
\begin{align*}
\Sigma^{\frac{l}{2}\mathbb{C}} \underline{\textnormal{swf}}_{}^{A}(Y,\mathfrak{s},A_{0}, g ; S^1)
&\cong \underline{\textnormal{swf}}_{}^{A}(Y,\mathfrak{s},A_{0}, g ; S^1), \\
\Sigma^{\frac{l}{2}\mathbb{C}} \underline{\textnormal{swf}}_{}^{R}(Y,\mathfrak{s},A_{0}, g ; S^1)
&\cong \underline{\textnormal{swf}}_{}^{R}(Y,\mathfrak{s},A_{0}, g ; S^1).
\end{align*}
{ When  the metric $g$ or the connection $A_{0}$ changes, the objects $\underline{\textnormal{swf}}_{}^{A}(Y,\mathfrak{s},A_{0},
g ; S^1)$ and $\underline{\textnormal{swf}}_{}^{R}(Y,\mathfrak{s},A_{0}, g ; S^1) $ can change only by suspending or desuspending
by copies of the complex representation $\mathbb{C}$ of $S^{1}$. }

{In the case that} $c_{1}(\mathfrak{s})$ is torsion, we can normalize the above objects to obtain invariants
$$
        \underline{\textnormal{SWF}}_{}^{A}(Y,\mathfrak{s} ; S^1)        \text{ and }
        \underline{\textnormal{SWF}}_{}^{R}(Y,\mathfrak{s};S^1)
$$
of the spin-c manifold $(Y, \mathfrak{s})$.

\end{thm}

A portion of this paper is devoted to proving that our construction is well-defined, i.e. it does not depend on choices involved in the construction up to canonical isomorphisms. Note that, for rational homology 3-spheres, the invariants $\underline{\textnormal{SWF}}^{A}$  and $\underline{\textnormal{SWF}}^{R}$ are the same and they  agree with Manolescu's spectrum.  In the case $b_1 (Y) = 1 $ and $\mathfrak{s} $ is nontorsion, $\underline{\textnormal{swf}}_{}^{A}(Y,\mathfrak{s},A_{0}, g ; S^1)$ is equivalent to $\textnormal{SWF}_{0}(Y,\mathfrak{s},g,A_{0}) $ constructed by Kronheimer and Manolescu.

\begin{rmk} According to Furuta \cite{FurutaPrivate}, one could set up a periodically graded category so that it is possible to define $\underline{\textnormal{SWF}}_{}^{A}(Y,\mathfrak{s} ; S^1) $ and
$\underline{\textnormal{SWF}}_{}^{R}(Y,\mathfrak{s} ; S^1)$ as invariants of the manifold in the nontorsion case.
 \end{rmk}
When $\mathfrak{s}$ is a spin structure, there is an additional $Pin(2)$-symmetry on the Seiberg-Witten equations.
The $Pin(2)$-equivariant Seiberg-Witten Floer spectrum for a rational homology sphere is instrumental in Manolescu's solution \cite{Manolescu3} of the Triangulation Conjecture. For a general spin 3-manifold, we have the following generalization:
 \begin{thm}\label{main theorem 2}
Let $Y$ be a closed, oriented 3-manifold and let $\mathfrak{s}$ be a spin {structure} on $Y$. We can obtain
$$\underline{\textnormal{SWF}}_{}^{A}(Y,\mathfrak{s} ; Pin(2)) \text{ and } \underline{\textnormal{SWF}}_{}^{R}(Y,\mathfrak{s};Pin(2))$$
as $Pin(2)$-equivariant analogs of $\underline{\textnormal{SWF}}_{}^{A}(Y,\mathfrak{s} ; S^1)$ and  $\underline{\textnormal{SWF}}_{}^{R}(Y,\mathfrak{s};S^1)$.
\end{thm}

Let us try to explain the motivation of our ``unfolded'' construction. Intuitively, the monopole Floer homology is a Morse-Floer homology of a quotient configuration space $Coul(Y)/H^{1}(Y;\mathbb{Z})$, where $Coul(Y)$ is a Hilbert space of configurations with gauge fixing. We see that this is a Hilbert bundle when $b_1 (Y) > 0 $ and we cannot simply use vector spaces for finite dimensional approximation. There is also a topological obstruction to find a good sequence of subbundles for finite dimensional approximation (cf. \cite[Proposition~6]{Kronheimer-Manolescu}). Thus, we instead do finite dimensional approximation on $Coul(Y)$. Since the Seiberg-Witten solutions and trajectories are no longer compact on $Coul(Y)$, we will require to consider spectra obtained from an increasing sequence of bounded sets with nice properties on $Coul(Y)$. Our unfolded spectrum is then obtained as a direct (or inverse) system from these spectra.

From the construction, we expect the homology of our unfolded invariants to agree with monopole
Floer homology with fully twisted coefficients, i.e. homology with
a local system on the blown up configuration space whose fiber is the group ring $\mathbb{Z}[H^{1}(Y;\mathbb{Z})]$. By equivalence of monopole Floer homology and Heegaard Floer homology, the corresponding Heegaard Floer group with  totally twisted coefficient $\underline{HF} (Y,\mathfrak{s}) $ is constructed by Ozsv\'ath and Szab\'o \cite[Section~8]{OStwisted}.
This inspires us to use underline notation $\underline{\textnormal{SWF}} $ for the unfolded spectrum.
Moreover, it should be possible to give a rigorous proof of this speculation with techniques developed
by Lidman and Manolescu \cite{Lidman-Manolescu}. However, this is not the aim of the present paper.

In another direction, the third author \cite{Sasahira} defined a folded version of Seiberg-Witten Floer spectra in the  case that the topological obstruction, as mentioned above, vanishes. The first author \cite[Chapter~6]{Khandhawit2} also gave an approach to define a folded invariant, called twisted Floer spectrum, for general 3-manifolds as  a twisted
parametrized spectrum. These theories will not be discussed here either.

One of the main complication to show well-definedness of our invariants is that we need to perturb the Chern-Simons-Dirac functional in the construction. First, we perturb the functional by a nonexact 2-form so that the functional is balanced (see Section \ref{sec setup} ). Second, we require that the set of critical points is discrete modulo gauge otherwise we cannot construct a good sequence of bounded subsets to apply finite dimensional approximation. As a result, the space of such perturbations may not be path connected and we cannot use standard homotopy argument here. Note that this difficulty was avoided in Manolescu's original construction because perturbations are not necessary in the case of homology spheres.

In general, our invariants are quite difficult to compute. However, by using Mrowka-Ozv$\acute{\text{a}}$th-Yu's explicit description of the Seiberg-Witten moduli space for Seifert manifolds \cite{MOY} and a refinement of the rescaling technique developed by the first author \cite{Khandhawit2}, we are able to give explicit computation of the invariants in torsion cases of the following manifolds:
\begin{enumerate}
\item  The manifold  $S^{2}\times S^{1}$;
\item  Large degree circle bundles over surfaces;
\item  All nil manifolds;
\item  All flat manifolds except $T^{3}$.
\end{enumerate}

 At the end of this introductory section, we briefly mention further developments that we hope to cover in our subsequent papers.
\begin{itemize}
\item  As an extension of Manolescu's construction \cite{Manolescu1}, we will define relative Bauer-Furuta invariants for a 4-manifold whose boundary can be an arbitrary 3-manifold.
\item The relative invariants will give new inequalities regarding intersection forms of spin 4-manifolds with boundary as in \cite{Manolescu2}.
\item We will establish Spanier-Whitehead duality between $\underline{\textnormal{SWF}}_{}^{A}$ and $\underline{\textnormal{SWF}}_{}^{R}$.
\item We will prove generalized gluing theorem for the relative Bauer-{Furuta} invariants.
\item Various applications of the generalized gluing theorem: behavior of the fiberwise Bauer-Furuta invariant under
surgery along loops, generalization of Bauer's connected sum theorem \cite{Bauer}; nonexistence
of essential spheres with trivial normal bundle in a 4-manifold with nontrivial Bauer-Furuta invariant, K$\ddot{\text{u}}$nneth
formula for Manolescu's spectrum.
\end{itemize}

The paper is organized as follows: Section \ref{sec setup} covers some of the basics of the Seiberg-Witten equations. Section \ref{Section approx SW traj} gives the analytical results which are needed in our constructions. Section \ref{Section Cat Top} reviews some elementary facts about the Conley index theory. Section \ref{construction of spectrum} constructs the spectrum invariants. Section \ref{invariance} proves the invariance. Section \ref{section linearized flow}  and  \ref{section examples} are devoted to the calculation of the examples.

\bigskip\noindent\textbf{Acknowledgement:} Some arguments in this paper first appeared in the first author's Ph.D. thesis at Massachusetts Institute of Technology \cite{Khandhawit2}. The first author is supported by World Premier International Research Center Initiative (WPI), MEXT, Japan. The third author is supported by JSPS KAKENHI Grant Number 25800040. We would like to thank Ciprian Manolescu for helpful discussions during the preparation of this paper.

\section{The Chern-Simons-Dirac functional and Seiberg-Witten trajectories} \label{sec setup}

Let $Y$ be a closed, oriented (but not necessarily connected) 3-manifold endowed with a $\text{spin}^{c}$ structure $\mathfrak{s}$ and a Riemannian metric $g$. We denote its connected components by $Y_{1}, \ldots, Y_{b_0}$ and denote by $b_1 = b_{1}(Y)$ its first Betti number. Let $S_{Y}$ be the associated spinor bundle and $\rho \colon TY\rightarrow \text{End}(S_{Y})$ be the Clifford multiplication. After fixing a base $\text{spin}^{c}$ connection $A_{0}$, the space of $\text{spin}^{c}$ connections on $S_{Y}$ can be identified with $i\Omega^{1}(Y)$ via the correspondence $A \mapsto A-A_{0}$.

Let $A_{0}^{t}$ be the connection on $\text{det}(S_{Y})$ induced by $A_{0}$.
We choose $A_{0}$ such that the curvature $F_{A_{0}^{t}}$ equals $2\pi i\nu_{0}$,
where $\nu_{0}$ is the harmonic $2$-form representing $-c_{1}(\mathfrak{s})$.
For a \mbox{1-form} $a \in i\Omega^{1}(Y) $, we let $\slashed{D}_{A_0 +a}$ be the Dirac operator associated to the connection $A_{0}+a$.
We also denote by $\slashed{D} := \slashed{D}_{A_{0}}$ the Dirac operator corresponding to the base connection, so we have $\slashed{D}_{A_{0}+a} = \slashed{D}_{} + \rho(a) $.

 The gauge group $\text{Map}(Y,S^{1})$ acts on the space $i\Omega^{1}(Y)\oplus\Gamma(S_{Y})$ by
 \begin{equation*}
     u \cdot (a,\phi)=(a-u^{-1}du, u  \phi)  ,
\end{equation*}
where $ u\in \text{Map}(Y,S^{1}) $ and $ (a,\phi) \in i\Omega^{1}(Y)\oplus\Gamma(S_{Y})$.
In practice, we will work with the Sobolev completion of the spaces $i\Omega^{1}(Y)\oplus\Gamma(S_{Y})$ and $\text{Map}(Y,S^{1})$ by the $L^{2}_{k}$ and $L^{2}_{k+1}$ norms respectively. We fix an integer $k > 4$ throughout the paper and denote the completed spaces by $\mathcal{C}_{Y}$ and $\mathcal{G}_{Y}$ respectively. We will also consider the following subgroups of $\mathcal{G}_{Y}$:
\begin{itemize}
\item $\mathcal{G}^{e}_{Y}:=\{u \in \mathcal{G}_{Y} \mid u=e^{\xi} \text{ for some }\xi \colon Y\rightarrow i\mathbb{R}\}$;
\item $\mathcal{G}^{e,0}_{Y}:=\{u \in \mathcal{G}^{e}_{Y} \mid u=e^{\xi} \text{ with } \int_{Y_{j}}\xi d\text{vol}=0 \text{ for } j = 1, \dots,  b_0 \}$;
\item $\mathcal{G}^{h}_{Y}:=\{u \in \mathcal{G}_{Y}\mid \Delta (\log u) = 0\} $ the harmonic gauge group, where $\Delta = d^* d$;
\item $\mathcal{G}^{h,o}_{Y}:=\{u \in \mathcal{G}_{Y}^{h} \mid u(o_{j})=1 \text{ for } j = 1, \ldots, b_0 \}$ the based harmonic gauge group,
where $o_{j}$ is a chosen base point on $Y_{j}$.

\end{itemize}
Note that $\mathcal{G}^{e}_{Y} \cong \mathcal{G}^{e,0}_{Y} \times (S^1)^{b_0}$ and $\mathcal{G}^{h}_{Y} \cong \mathcal{G}^{h,0}_{Y} \times (S^1)^{b_0} $.

The balanced Chern-Simons-Dirac functional $CSD_{\nu_{0}} \colon \mathcal{C}_{Y}\rightarrow \mathbb{R}$ is defined as
\begin{equation*}\label{unperturbed Chern-Simons functional}
      CSD_{\nu_0}(a,\phi)
      :=-\frac{1}{2} \left( \int_{Y} a\wedge da- \int_{Y} \langle\phi,\slashed{D}_{A_{0}+a}(\phi)\rangle d\text{vol} \right).
\end{equation*}
Note that this is a perturbation of the standard Chern-Simons-Dirac functional by the closed but nonexact 2-form $\nu_0$ so that $CSD_{\nu_{0}} $ becomes invariant under the full gauge group (cf.\cite[Definition~29.1.1]{Kronheimer-Mrowka}). The formal $L^2$-gradient is given by
\begin{equation} \label{eq gradCSD}
\grad  CSD_{\nu_{0}}(a,\phi)= (*da +\rho^{-1}(\phi\phi^{*})_{0}, \slashed{D}_{A_0 + a} \phi),
\end{equation}
where $(\phi\phi^{*})_{0}$ is the traceless part the endomorphism $\phi\phi^{*} $ on $S_{Y}$.

If we slightly perturb $CSD_{\nu_0}$, the critical points of $CSD_{\nu_{0}}$ are discrete modulo gauge transformations. To ensure this property, we will need to pick a function $f \colon \mathcal{C}_{Y}\rightarrow \mathbb{R}$ which is invariant under $\mathcal{G}_{Y}$ and consider a twice perturbed functional $CSD_{\nu_{0},f}:=CSD_{\nu_{0}}+f$. We will make use of a large Banach space of perturbations constructed by Kronheimer and Mrowka \cite[Section~11]{Kronheimer-Mrowka}.

\begin{defi} \label{def space perturbation}
 Let $\{ \hat{f}_j \}_{j=1}^{\infty} $ be a countable collection of cylinder functions as
in \cite[Page~193]{Kronheimer-Mrowka}. Given a sequence $\{ C_j \}_{j=1}^{\infty}$ of positive real numbers, we consider a separable Banach space
\begin{equation} \label{banach space of cylinder functions}
     \mathcal{P} =
     \left\{ \sum_{j=1}^{\infty} \eta_j \hat{f}_j    \left|
       \eta_j \in \mathbb{R},  \  \sum_{j=1}^{\infty} C_j | \eta_j | < \infty   \right. \right\},
\end{equation}
where the norm is defined by $\left\Vert\sum_{j=1}^{\infty} \eta_j \hat{f}_j\right\Vert = \sum_{j=1}^{\infty} | \eta_j | C_j$. An element of $\mathcal{P}$ will be called an \emph{extended cylinder function}.
\end{defi}

The Banach space $\mathcal{P}$ will be fixed throughout the paper. In particular, we will choose a real sequence $\{ C_j \}_j$ satisfying our requirements as in the following result.

\begin{pro}   \label{prop perturbproperty}
The sequence $\{ C_j \}_j$ can be chosen so that any extended cylinder function $\bar{f}$ in $\mathcal{P}$ has the following properties:
\begin{enumerate}[(i)]
\item \label{item bounded function} $\bar{f}$ is a bounded function;
\item \label{item propertyperturb1} The formal $L^2$-gradient $\grad \bar{f}$ is a tame perturbation (see \cite[Definition 10.5.1]{Kronheimer-Mrowka});

\item \label{item propertyperturb2} For any positive integer $m$, the gradient  $\grad \bar{f}$ defines
a smooth vector field on the Hilbert space $L^{2}_{m}(i\Omega^{1}(Y)\oplus
\Gamma(S_{Y}))$. Moreover, for each nonnegative integer $n$, we have
\[
       \| \mathcal{D}^{n}_{(a,\phi)}\grad \bar{f}\|   \leq   C \, p_{m,n}(\|(a,\phi)\|_{L^{2}_{m}}),
\]
where $p_{m,n}$ is a polynomial depending only on $m,n$ and $C$ is a constant depending on $m, n$ and $\bar{f}$. The norm of $\mathcal{D}^{n}_{(a,\phi)}\grad \bar{f}$ is taken considering $\mathcal{D}^{n}_{(a,\phi)}\grad \bar{f}$ as an element of
\[
       \operatorname{Mult}^{n}(\times_{n}L^{2}_{m}(i\Omega^{1}(Y)   \oplus
       \Gamma(S_{Y})),L^{2}_{m}(i\Omega^{1}(Y)\oplus \Gamma(S_{Y}))).
\]

\item \label{item propertyperturb3}
$\{ C_j \}_j$ is taken so that the statement of Lemma \ref{lem C_j W_M} and \ref{double mixed tame functions} holds.


\end{enumerate}
\end{pro}

\begin{proof}
By the definition of cylinder functions, each $\hat{f}_{j}$ is bounded. Therefore, property~(\ref{item bounded function}) can be ensured by taking $\{C_{j}\}_{j}$ increasing fast enough.  Property~(\ref{item propertyperturb1}) is a consequence of \cite[Theorem~11.6.1]{Kronheimer-Mrowka}. For property~(\ref{item propertyperturb2}), let $\hat{f}_j$ be a cylinder function from the collection. By \cite[Proposition 11.3.3]{Kronheimer-Mrowka}, the gradient $\grad \hat{f}_j$ defines a smooth vector field over  $L^{2}_{m}(i\Omega^{1}(Y)\oplus \Gamma(S_{Y}))$ with the property that
\[
        \|\mathcal{D}^{n}_{(a,\phi)}\grad \hat{f}_j\|
           \leq
       C'_{j, m, n} (1+\|\phi\|_{L^{2}})^{n}  (1+\|a\|_{L^{2}_{m-1}})^{m}  (1+\|\phi\|_{L^{2}_{m,A_{0}+a}}) ,
\]
where $C'_{j, m, n}$ is a constant and $\|\cdot\|_{L^{2}_{m,A_{0}+a}}$ denotes the $L^{2}_{m}$-norm defined using the connection $A_{0}+a$. Therefore, we only need to estimate $\|\phi\|_{L^{2}_{m,A_{0}+a}}$ by a polynomial of $\|(a,\phi)\|_{L^{2}_{m}}$.

Notice that the expansion of $\nabla^{(m)}_{A_{0}+a}\phi$ consists of terms  of the form $\nabla^{(n_{1})}a \cdot \nabla^{(n_{2})}a \cdots \nabla^{(n_{i})}a\cdot \nabla^{(n_{i+1})}_{A_{0}}\phi$
 where $\nabla$ denotes the Levi-Civita connection and $i,n_{1}, \ldots ,n_{i+1}$ are nonnegative integers satisfying $n_{1}+n_{2}+ \cdots +n_{i+1}+i=m.$
As we want to control the $L^{2}$-norm of this term using $\|(a,\phi)\|_{L^{2}_{m}}$, there are three cases:

\begin{itemize}
\item $i=0$: This is trivial since $\|\phi\|_{L^{2}_{m}}\leq \|(a,\phi)\|_{L^{2}_{m}}$;

\item $i=1 \text{ and }n_{1}=m-1$: We apply Sobolev multiplication $L^{2}_{1}\times L^{2}_{m}\rightarrow L^{2}$ and obtain $\left\Vert \nabla^{(m-1)}a \cdot \phi \right\Vert_{L^2} \le C \left\Vert \nabla^{(m-1)}a \right\Vert_{L^2_1} \left\Vert \phi \right\Vert_{L^2_m} \le C \|(a,\phi)\|^2_{L^{2}_{m}}$. The case $i=1$ and $n_{2}=m-1$ can be done in the same manner;

\item Otherwise, we will have $i \ge 1$ and $n_{1}, \ldots, n_{i+1}<m-1$. Similarly, we consider $n_{\text{max}} = \max{\{n_1 , \ldots , n_{i+1}\}}$ and apply Sobolev multiplication
\[
       L^{2}_{m-n_{1}}\times \cdots \times L^{2}_{m-n_{i}}\times L^{2}_{m-n_{i+1}}
        \rightarrow
     L^{2}_{m- n_{\text{max}}} \hookrightarrow\ L^{2}.
\]
\end{itemize}

Putting these together, we can find a polynomial $p_{m,n} $ (independent of $j$) such that
\[
      \|\mathcal{D}^{n}_{(a,\phi)}\grad \hat{f}_j\|\leq C'_{j, m, n} \, p_{m,n}(\|(a,\phi)\|_{L^{2}_{m}}).
\]
For each $j$, take a constant $C_j$ with
\[
           C_{j} \geq \max \{ \ C_{l_1, l_2, l_3}' \ | \ 0 \leq l_1, l_2, l_3  \leq j \ \}.
\]
We will prove that condition (\ref{item propertyperturb2}) is satisfied.  Take any element $\bar{f} = \sum_{j} \eta_j \hat{f}_j$ of $\mathcal{P}$.  Then we have
\[
   \begin{split}
      \| \mathcal{D}^n_{ (a, \phi) } \grad \bar{f} \|
    &  \leq   \sum_{j} | \eta_j | \| \mathcal{D}^n_{ (a, \phi) } \grad \hat{f}_j \|      \\
    &  \leq   \sum_{j} | \eta_j | C_{j, m, n}' p_{m, n}( \| (a, \phi) \|_{L^2_m} )      \\
    &   \leq  \left( \sum_{1 \leq j \leq N }  | \eta_j | C_{j, m, n}' +  \sum_{ j \geq N} | \eta_j | C_j \right)
                    p_{m, n}( \| (a, \phi ) \|_{L^2_m} ).
      \end{split}
\]
Here $N = \max \{ m, n \}$.  Putting $C := \left( \sum_{1 \leq j \leq N }  | \eta_j | C_{j, m, n}' +  \sum_{ j \geq N} | \eta_j | C_j \right)$, we obtain
\[
          \| \mathcal{D}^n_{ (a, \phi) } \grad \bar{f} \| \leq C p_{m, n}( \| (a, \phi) \|_{L^2_m}).
\]
Thus $\mathcal{P}$ satisfies  (\ref{item propertyperturb2}).

By further shrinking $C_{j}$, we may suppose that $C_j$ satisfies Lemma \ref{lem C_j W_M} and Lemma \ref{double mixed tame functions} (2).  That is, condition (\ref{item propertyperturb3}) is satisfied.


\end{proof}

The perturbation we consider in the current paper will be of the form
\begin{equation*}\label{general form of perturbation}
f(a,\phi)=\bar{f}(a,\phi)+\frac{\delta}{2}\|\phi\|_{L^{2}}^{2},
\end{equation*}
where $\bar{f}$ is an extended cylinder function and $\delta$ is a real number.
We sometimes write the above perturbation as a pair $(\bar{f},\delta)$.

\begin{defi} \label{good perturbation}
A perturbation $f = (\bar{f} , \delta)$ is called
\emph{good} if the critical points of $CSD_{\nu_{0},f}$ are discrete modulo
gauge transformations.
\end{defi}

When $\delta =0$, we know that good perturbations are generic in $\mathcal{P}$ by virtue of \cite[Theorem~12.1.2]{Kronheimer-Mrowka}. It is immediate to extend the result to a general case and we only give a statement here.

\begin{lem}
For any real $\delta$, a subset of extended cylinder functions $\bar{f}$ in $ \mathcal{P}$ such that $(\delta,\bar{f})$ is a good perturbation is residual.
\end{lem}



\begin{rmk}
To define our invariants, it is sufficient to take $\delta=0$. We include the term $\frac{\delta}{2}\|\phi\|^{2} $ as it will facilitate computations of many examples in Section \ref{section examples}.
\end{rmk}


Our main object of interest is the negative gradient flow of the functional $CSD_{\nu_{0}, f}$ on the space $\mathcal{C}_Y $ modulo the gauge group. Let $I \subset \mathbb{R}$ be an interval.  A trajectory $\gamma \colon I \rightarrow \mathcal{C}_{Y}$ of the negative gradient flow is described by the equation
\begin{equation*}\label{Seiberg-Witten trajectory}
 -\frac{\partial}{\partial t}\gamma(t)=\grad CSD_{\nu_{0},f}(\gamma(t)).
\end{equation*}

As in \cite{Manolescu1} and \cite{Khandhawit2}, it is more convenient to study the flow on the subspace called the Coulomb slice
\begin{equation*}
               Coul(Y)  =  \{(a,\phi) \mid d^{*}a=0\}\subset \mathcal{C}_{Y}.
\end{equation*}


Since any configuration $(a,\phi)\in \mathcal{C}_{Y}$ can be gauge transformed into $Coul(Y)$ by a unique element of $\mathcal{G}_{Y}^{e,0}$, the Coulomb slice is isomorphic to the quotient $\mathcal{C}_{Y}/\mathcal{G}^{e,o}_{Y}$ with residual action by the harmonic gauge group $\mathcal{G}^{h}_{Y}$.

%

Let $\Pi:\mathcal{C}_{Y}\rightarrow\mathcal{C}_{Y}/\mathcal{G}^{e,o}_{Y}\cong Coul(Y)$ be the nonlinear Coulomb projection. The formula for $\Pi$ is given by
\begin{align} \label{eq nonlinearCoulproj}
\Pi (a,\phi) =  \left(a - d\bar{\xi}(a) , e^{\bar{\xi}(a)} \phi\right) ,
\end{align}
where $\bar{\xi}(a) \colon Y \rightarrow i \mathbb{R} $ is a unique function which solves \begin{equation}\label{solving laplace}\Delta \bar{\xi}(a) = d^* a \text{ and } \int_{Y_j}\bar{\xi}(a) = 0 \text{ for each } j = 1, \dots, b_0.\end{equation}

To describe the Seiberg-Witten vector field on $Coul(Y)$, we first consider a trivial bundle $\mathcal{T}_{k-1} $ over $\mathcal{C}_Y $ with fiber
$L^{2}_{k-1}(i\Omega^{1}(Y)\oplus
\Gamma(S_{Y}))$. Note that the vector field $\grad CSD_{\nu_{0},f}$ is a section of $\mathcal{T}_{k-1} $. Similarly, we have a trivial bundle $\mathit{Coul}_{k-1} $ over $Coul(Y) $ whose fiber is the $L^2_{k-1}$-completion of $\ker{d^*} \oplus \Gamma(S_Y) $. At a point $(a,\phi) \in Coul(Y) $, the pushforward ${\Pi}_* \colon \mathcal{T}_{k-1} \rightarrow \mathit{Coul}_{k-1}$ of the Coulomb projection $\Pi$ is given by

\begin{align} \label{eq pipushforward}
{\Pi}_{*  (a,\phi)} (b,\psi) =\left( b - d \bar{\xi}(b) , \psi + \bar{\xi}(b) \phi \right).
\end{align}

We now project the negative gradient flow lines from $\mathcal{C}_{Y}$ to $Coul(Y)$ using $\Pi$. Such projected trajectories $\gamma \colon I\rightarrow Coul(Y)$ are described by an equation
\begin{equation}\label{projected trajectory}
-\frac{\partial}{\partial t}\gamma(t)= {\Pi}_*\grad CSD_{\nu_{0},f} (\gamma(t)).
\end{equation}
From (\ref{eq gradCSD}) and (\ref{eq pipushforward}), we can write down an explicit formula for the induced vector field on $Coul(Y) $ as a section of $\mathit{Coul}_{k-1}$
\begin{align} \label{eq grad=l+c}
{\Pi}_*\grad CSD_{\nu_{0},f}(a, \phi) = l(a,\phi) + c(a,\phi),
\end{align}
where $l=(*d,\slashed{D})$ is a first order elliptic operator and $c = (c^1 , c^2) $ is given by
\begin{align}
c^1 (a, \phi) &= \rho^{-1}(\phi\phi^{*})_{0}+\grad
^{1}f(a,\phi) - d \bar{\xi}(\rho^{-1}(\phi\phi^{*})_{0}+\grad
^{1}f(a,\phi)),
                          \label{def of c^1} \\
c^2(a,\phi) &= \rho(a) \phi + \grad^2 f (a,\phi)+ \bar{\xi}(\rho^{-1}(\phi\phi^{*})_{0}+\grad^{1}f(a,\phi)) \phi .
                          \label{def of c^2}
\end{align}
Note that $l$ is linear and the nonlinear term $c$ has nice compactness properties which will be explored in Section \ref{Section approx SW traj}. We will call those trajectories $\gamma$ satisfying (\ref{projected trajectory}) the \emph{Seiberg-Witten trajectories}. By the standard elliptic bootstrapping argument,  $\gamma$ is actually a smooth path in $Coul(Y)$ when restricted to interior of $I$.


We would also like to interpret the vector field ${\Pi}_* \grad CSD_{\nu_{0},f} $ from (\ref{projected trajectory}) as a gradient vector field on $Coul(Y)$. However, ${\Pi}_* \grad CSD_{\nu_{0},f}$ is not the gradient of the restriction $CSD_{\nu_{0}, f}|_{Coul(Y)}$ with respect to the standard $L^2$-metric and we need to introduce another metric on $Coul(Y)$. Roughly speaking, we have to measure only the component of a vector on $Coul(Y)$ which is orthogonal to the linearized gauge group action.
More specifically, consider a bundle decomposition over $\mathcal{C}_Y$
\begin{align*}
\mathcal{T}_{k-1} = \mathcal{J}_{k-1} \oplus \mathcal{K}_{k-1},
\end{align*}
where the fiber of $\mathcal{J}_{k-1}$ at $(a,\phi)$ consists of a vector of the form $(-d \xi , \xi \phi)$ where $\xi \in L^2_{k}(Y;i \mathbb{R})$ with $\int_{Y_j} \xi=0$ and the fiber of $\mathcal{K}_{k-1}$ is the $L^2$-orthogonal complement. Note that this decomposition is slightly different from the decomposition which appeared in \cite[Section~9.3]{Kronheimer-Mrowka} as we use the derivative of the action of $\mathcal{G}^{e,0}_Y $ rather than $\mathcal{G}^e_Y$. Let $\widetilde{\Pi}$ be the $L^2$-orthogonal projection onto $\mathcal{K}_{k-1} $. Explicitly, the projection $\widetilde{\Pi}$ at $(a,\phi) $ is given by
\begin{align*}
\widetilde{\Pi}_{(a,\phi)} (b,\psi) = \left( b - d \tilde{\xi}(b, \psi, \phi) , \psi + \tilde{\xi}(b, \psi, \phi) \phi \right),
\end{align*}
where $\tilde{\xi}(b,\psi,\phi) : Y \rightarrow i \mathbb{R}$ is a unique function such that
$-d^{*}(b_{}-d\tilde{\xi}(b,\psi,\phi))+i\text{Re}\langle i\phi,\psi_{}+ \tilde{\xi}(b,\psi,\phi) \phi \rangle$ is a locally constant function and $\int_{Y_{j} }\tilde{\xi}(b,\psi,\phi)=0$. It is not hard to see that we have a bundle isomorphism
\begin{equation*}
\xymatrix{
   \mathit{Coul}_{k-1} \ar[dr]^{} \ar@<2pt>[rr]^{\widetilde{\Pi}} & & \mathcal{K}_{k-1} \ar@<2pt>[ll]^{\Pi_*} \ar[dl]  \\
   & Coul(Y)  &
}
\end{equation*}
since both are complementary to the derivative of the action of $\mathcal{G}^{e,0}_Y$.

We now define a metric $\tilde{g}$ for the bundle $\mathit{Coul}_{k-1} $ by setting$$\langle (b_{1},\psi_{1}),(b_{2},\psi_{2})\rangle _{\tilde{g}}:=\langle\widetilde{\Pi} (b_{1},\psi_{1}),\widetilde{\Pi}(b_{2},\psi_{2})\rangle _{L^{2}}.$$
Since $\widetilde{\Pi} $ and $\Pi_* $ are inverse of each other and $\widetilde{\Pi} $ is an orthogonal projection, we have the following identity
\begin{equation*}
            \left\langle \Pi_* v, w\right\rangle_{\tilde{g}}
            = \left\langle v , w \right\rangle_{L^2} \quad \text{whenever } v \in \mathcal{K}_{k-1}.
\end{equation*}
Since $CSD_{\nu_{0},f}$ is gauge invariant, $\grad CSD_{\nu_{0},f}$ lies in $\mathcal{K}_{k-1} $. From this point on, we will denote by $\gradtil $ the gradient on $Coul(Y)$ with respect to the metric $\tilde{g} $ and put
\[
           \mathcal{L} := CSD_{\nu_0, f} |_{ Coul(Y) }.
\]
We then have
\begin{equation}\label{two versions of gradient}
         \gradtil \mathcal{L}  = {\Pi}_*\grad CSD_{\nu_{0},f}  =  l+c
         \text{ and }
         \|\gradtil  \mathcal{L}\|_{\tilde{g}} = \|\grad  CSD_{\nu_{0},f}\|_{L^{2}}.
\end{equation}
Note that analogous results hold for any functional on $\mathcal{C}_Y$ which is $\mathcal{G}^{e,0}_Y$-invariant.

\section{Analysis of approximated Seiberg-Witten trajectories}  \label{Section approx SW traj}

In this section, we review some boundedness and convergence results relevant to finite dimensional approximation which will be used in the main construction.
\begin{defi}
A smooth path in $Coul(Y)$ is called \emph{finite type} if it
is contained in a fixed bounded set (in the $L^{2}_{k}$-norm).
\end{defi}
It can be proved that
a Seiberg-Witten trajectory $\gamma(t)=(\alpha(t),\phi(t))$ is of finite type if and only
if both $CSD_{\nu_{0},f}(\gamma(t))$ and $\|\phi(t)\|_{C^{0}}$ are bounded (cf. \cite[Definition~1]{Manolescu1}).

Recall that the set of the Seiberg-Witten solutions is compact modulo the full gauge group. However, there is a residual action by the group $\mathcal{G}^{h,o}_{Y}\cong H^{1}(Y; \mathbb{Z}) $ on $Coul(Y)$. This motivates us to consider a strip of balls
$$Str(R)=\{x\in Coul(Y) \mid \exists h\in \mathcal{G}^{h,o}_{Y} \text{ s.t. } \|h\cdot x\|_{L^{2}_{k}}\leq
R\},$$
where $R$ is a positive real number.

Since $CSD_{\nu_{0},f}$ is invariant under the full gauge group $\mathcal{G}_{Y}$, we  have a uniform  bound for the topological energy of all finite type trajectories (see \cite[Proposition
10]{Khandhawit2}). As a result, we have the following boundedness result.

\begin{thm}[\cite{Khandhawit2}]\label{Boundedness of trajectories} There exists a constant $R_{0}$ such that all finite type Seiberg-Witten trajectories are contained in the interior of $Str(R_{0})$. In particular, the set $Str(R_{0})$ contains all the critical points of $\mathcal{L}$ and trajectories between them.
\end{thm}

We now discuss finite dimensional approximation of Seiberg-Witten trajectories  following \cite{Manolescu1} and \cite{Khandhawit2}. To describe various projections, we first specify the $L^{2}_{m}$-inner product $(m\geq 1)$ on
$i\Omega^{1}(Y)\oplus \Gamma(S_{Y})$.
From the Hodge decomposition $\Omega^1 (Y) = \ker{d^*} \oplus \im{d} $, we will just define  an inner
product on each summand. On $i\ker d^{*}\oplus \Gamma(S_{Y})$, we use the elliptic operator $l=(*d,\slashed{D})$
$$
\langle (a_{1},\phi_{1}),(a_{2},\phi_{2})\rangle_{L^{2}_{m}}:=\langle (a_{1},\phi_{1}),(a_{2},\phi_{2})\rangle_{L^{2}}+\langle
l^{m}(a_{1},\phi_{1}),l^{m}(a_{2},\phi_{2})\rangle_{L^{2}}.
$$
 For $\beta_{1},\beta_{2}\in i\im  d$, we define
 $$
        \langle \beta_{1},\beta_{2}\rangle_{L^{2}_{m}}
       := \langle \beta_{1}, \beta_{2} \rangle_{L^{2}}+\langle \Delta^m \beta_{1},  \beta_{2}\rangle_{L^{2}}.
$$
\begin{defi}\label{nice projections}
With the Sobolev inner product defined above, a projection $\pi$ will be called a \emph{nice} projection if it satisfies the following
properties:
\begin{enumerate}[(i)]
\item $\pi$ is an $L^{2}_{m}$-orthogonal projection for any $m\geq 0$;
\item $\pi$ extends to a map on a cylinder $I \times Y $ with $ \left\Vert\pi\right\Vert_{L^2_{m}(I \times Y)} \le 1 $ for any $m \ge 0$.
\end{enumerate}
\end{defi}

Consider the spectral decomposition of $Coul(Y)$ with respect to the eigenspaces of $l=(*d,\slashed{D})$.
For any real numbers $\lambda < 0 \leq \mu$, let $V^{\mu}_{\lambda}$ be the span of the eigenspaces of $l$
with eigenvalues in the interval $(\lambda,\mu]$ and let $p^{\mu}_{\lambda}$ be the $L^2$-orthogonal
projection onto $V^{\mu}_{\lambda}$. It is not hard to see that $p^{\mu}_{\lambda}$ is a nice projection.

Recall that a Seiberg-Witten trajectory is an integral curve of the vector field  $l +c$ on $ Coul(Y)$.
This leads us to consider a trajectory on a finite-dimensional subspace $\gamma \colon I \rightarrow V^{\mu}_{\lambda}$ satisfying an equation
\begin{equation*}
-\frac{d\gamma(t)}{dt}= (l+p^{\mu}_{\lambda}\circ c)(\gamma(t)).
\end{equation*}
Such a trajectory will be loosely called an approximated Seiberg-Witten trajectory. We will also call a sequence of approximated Seiberg-Witten trajectories $\left\{\gamma_n \colon I \rightarrow V^{\mu_{n}}_{\lambda_{n}}\right\}_{n \in\mathbb{N}}$ an exhausting sequence when $-\lambda_n , \mu_n \rightarrow \infty $.
The next proposition is the main convergence result of this section.
\begin{pro}\label{convegence of approximated trajectories}
Let  $\{\gamma_{n} \colon [a,b]\rightarrow V^{\mu_{n}}_{\lambda_{n}}\}$ be an exhausting sequence of approximated
Seiberg-Witten trajectories whose $L^{2}_{k}$-norms are uniformly bounded. Then
there exists a Seiberg-Witten trajectory $\gamma_{\infty}\colon (a,b)\rightarrow Coul(Y)$, such that,
after passing to a subsequence, $\gamma_{n}(t)\rightarrow \gamma_{\infty}(t)$ uniformly in any Sobolev
norm on any compact subset of $(a,b)$.
\end{pro}

The proof of this proposition will be at end of this section. We basically follow the same strategy as in the proof of \cite[Proposition~3]{Manolescu1} and \cite[Proposition~11]{Khandhawit2}. Since our vector field $l+c$ has an extra term coming from $\grad f$, we need to assure that the nonlinear part $c$ still has nice compactness properties similar to those of the quadratic term in the Seiberg-Witten equation. For this purpose, we recall the notion of
 ``quadratic-like'' map and related results in \cite[Section~4.2]{Khandhawit2}. Since our setting here is slightly different,
we give out some details for completeness.

\begin{defi}\label{quadratic-like map}
Let $E$ be a vector bundle over $Y$. A smooth map $Q \colon Coul(Y)\rightarrow L^{2}_{k}(\Gamma(E))$
is called \emph{quadratic-like} if it has the following properties:\begin{enumerate}[(i)]
\item \label{item defquadlike bound} The map \(Q\) sends a bounded subset in $L^2_{k}$ to a bounded subset in  $L^2_{k}$;
\item \label{item defquadlike converge}  Let $m$ be a nonnegative integer not greater than  $k-1$. If there is a convergence of paths over a compact interval $(\frac{d}{dt})^{s}\gamma_{n}(t)\rightarrow (\frac{d}{dt})^{s}\gamma_{\infty}(t)$
uniformly in $L^{2}_{k-1-s}$ for each $s=0,1,...,m$, then we have  $(\frac{d}{dt})^{m}Q(\gamma_{n}(t))\rightarrow
(\frac{d}{dt})^{m}Q(\gamma_{\infty}(t))$ uniformly in $L^{2}_{k-2-m}$;
\item \label{item defquadlike 4d} The map $Q$ extends to a continuous map from $L^2_m(I \times Y)$ to $L^2_m(I \times Y)$ (with suitable bundles understood) for each integer $m \ge k-1 $. Here $I$ is a compact interval.
\end{enumerate}
\end{defi}

The sum of two quadratic-like maps is obviously quadratic-like. Furthermore, it can be shown that the pointwise tensor product of two quadratic-like maps is also quadratic-like (cf. \cite[Lemma~10]{Khandhawit2}).
%

\begin{lem}[cf. Lemma~9 of \cite{Khandhawit2}]\label{gradf is quadratic-like}
Let $f$ be a perturbation given by a pair $(\delta,\bar{f})$ with $\delta\in \mathbb{R}$ and $\bar{f}\in \mathcal{P}$.
Then the map $ \grad{f} \colon Coul(Y)\rightarrow L^{2}_{k}(i\Omega^{1}(Y)\oplus \Gamma(S_{Y}))$ is quadratic-like.
\end{lem}
\begin{proof}
We see that $\grad f(a,\phi)= (0, \delta \phi) + \grad \bar{f}(a,\phi)$ and the first term is obviously quadratic-like. We just need to show that $\grad{\bar{f}} $ is quadratic-like. First, we will check properties (\ref{item defquadlike bound}) and (\ref{item defquadlike converge}) when $m=0$ of Definition~\ref{quadratic-like map}  .

For two configurations $(a_0 , \phi_0 ) $ and $(a_1 , \phi_1 ) $, we consider a straight segment $(a_t , \phi_t ) = (1-t)(a_0 , \phi_0 ) + t (a_1 , \phi_1 ) $ joining them and apply the fundamental theorem of calculus
\begin{equation*}
   \begin{split}
    \left\Vert\grad \bar{f}(a_1,\phi_1) - \grad \bar{f}(a_0,\phi_0)\right\Vert_{L^2_j}
       &  = \left\Vert\int_{[0,1]} \mathcal{D}_{(a_t , \phi_t )} \grad \bar{f}(a_1 - a_0,\phi_1 - \phi_0) dt\right\Vert_{L^2_j} \\
       &  \le C \int_{[0,1]} p_{j,1}(\left\Vert a_t , \phi_t \right\Vert_{L^2_j}) \left\Vert (a_1,\phi_1) - (a_0,\phi_0) \right\Vert_{L_j^2}dt,
   \end{split}
\end{equation*}
where the last inequality follows from Proposition~\ref{prop perturbproperty}~(\ref{item propertyperturb2}). When $j=k$ and $(a_0 , \phi_0) =(0,0) $, this implies property~(\ref{item
defquadlike bound}) of Definition~\ref{quadratic-like map}. Property~(\ref{item defquadlike converge}) when $m=0$ also follows from the above inequality when $j=k-1$.

We now check property~(\ref{item defquadlike converge})
when $1 \le m \le k-2$. Suppose that $(\frac{d}{dt})^{s}\gamma_{n}(t)\rightarrow (\frac{d}{dt})^{s}\gamma_{\infty}(t)$
uniformly in $L^{2}_{k-1-s}$ for each $s=0,1, \dots ,m$. We observe that an expansion of $(\frac{d}{dt})^{m}\grad \bar{f}(\gamma_{}(t))$ consists of terms of the form
\begin{align*}
       \mathcal{D}^{s}_{\gamma_{}(t)}  \grad
        \bar{f}  \left(  \left(    \frac{d}{dt} \right)^{\alpha_{1}}    \gamma (t),
          \dots,  \left(\frac{d}{dt}  \right)^{\alpha_{s}}   \gamma(t)        \right)
 \text{ with }\alpha_i \ge 1 \text{ and } \alpha_1 + \cdots + \alpha_s = m.
\end{align*}
From Proposition~\ref{prop perturbproperty}~(\ref{item propertyperturb2}), $\Vert\mathcal{D}^{s}_{\gamma_{}(t)}\grad \bar{f}\Vert \le C \, p_{k-1-m , s}(\left\Vert \gamma(t) \right\Vert_{L^2_{k-1-m}})$ as an element of $\operatorname{Mult}^{s}(\times_{s}L^{2}_{k-1-m},L^{2}_{k-1-m}) $. We see that $\gamma_{n}$ is uniformly bounded in $L^2_{k-1-m} $ and that the convergence $(\frac{d}{dt})^{\alpha_{i}}\gamma_{n}(t)\rightarrow
(\frac{d}{dt})^{\alpha_{i}}\gamma_{\infty}(t)$ is uniform in $L^{2}_{k-1-m}$ as $\alpha_i \le m$. These imply property~(\ref{item defquadlike converge}).

Properties~(\ref{item defquadlike 4d}) easily follows from the fact that $\grad \bar{f} $ is a tame perturbation.

\end{proof}

As a result, we can deduce compactness property of the induced vector field on $Coul(Y) $.

\begin{cor}\label{c is quadratic-like}
The nonlinear part $c$ of the induced Seiberg-Witten vector field in (\ref{eq grad=l+c}) is quadratic-like.
\end{cor}
\begin{proof} It is clear that the composition of a quadratic-like map with a linear operator of nonpositive order is quadratic-like. Since the operator $\bar{\xi}$ in (\ref{eq nonlinearCoulproj}) is of order -1, Lemma~\ref{gradf is quadratic-like} and closure under pointwise multiplication imply that the map $c$ is quadratic-like.


\end{proof}

We are now ready to prove Proposition \ref{convegence of approximated trajectories}. Although, we will only give outline of the proof as the reader can find more details in \cite{Manolescu1} and \cite{Khandhawit2}.
\begin{proof}[Proof of Proposition~\ref{convegence of approximated trajectories}]
 Let $\{\gamma_{n}\}$ be an exhausting sequence of approximated trajectories which are all contained in a ball $B(R)$ in $L^2_k$. The norm $\|\frac{d}{dt}\gamma_{n}(t)\|_{L^{2}_{k-1}}$ is uniformly bounded by boundedness of the map $l+c $. By the Rellich lemma and the Arzela-Ascoli theorem, we can pass to a subsequence of $\{\gamma_{n}\}$ which converges to a path $\gamma_{\infty}$ uniformly in $L^{2}_{k-1}$. Moreover, it can be shown that $\gamma_\infty $ is a Seiberg-Witten trajectory. By property~(\ref{item
defquadlike converge}) of Definition~\ref{quadratic-like map} of $c$, we can inductively prove uniform convergence $(\frac{d}{dt})^{m}(\gamma_{n}(t))\rightarrow (\frac{d}{dt})^{m}(\gamma_{\infty}(t))$ in $L^{2}_{k-1-m}$ for $m=1, \dots, k-1$. This implies that $\hat{\gamma}_{n}\rightarrow \hat{\gamma}_{\infty}$
in $L^{2}_{k-1}([a,b]\times Y)$. (Here we treat $\gamma_{n}(t)$ and $\gamma_{\infty}(t)$ as sections over $I\times Y$ and denote them respectively by $\hat{\gamma}_{n}$ and $\hat{\gamma}_{\infty}$.) Property (\ref{item defquadlike 4d}) of Definition \ref{quadratic-like map} allows us to do the bootstrapping argument over any shorter cylinder $I\times Y$. This finishes the proof of the proposition.
\end{proof}

Proposition \ref{convegence of approximated trajectories} has the following consequence.

\begin{cor}\label{boundedness for approximated trajectories}
For a closed and bounded subset $S$ of $Coul(Y)$ in $L^2_k$, there exist  large numbers $-\bar{\lambda}, \bar{\mu}, -\bar{T} \gg 0$  such that if $\lambda < \bar{\lambda}$, $\mu > \bar{\mu}$ and $T > \bar{T}$ then for any approximated Seiberg-Witten trajectory $\gamma \colon [-T,T]\rightarrow V^{\mu}_{\lambda}$ contained
in $S$, we have $\gamma(0)\in Str(R_{0})$.  Here $R_{0}$ is the universal constant from Theorem \ref{Boundedness of trajectories}.
\end{cor}

\begin{proof}
Suppose the contrary: we can find an exhausting sequence of approximated trajectories $\gamma_{n} \colon [-T_{n},T_{n}]\rightarrow V^{\mu_{n}}_{\lambda_{n}}\cap S$,  with $T_{n}\rightarrow \infty$, with $\gamma_{n}(0)\notin Str(R_{0})$. Since $S$ is bounded, we can apply Proposition~\ref{convegence of approximated trajectories} and the diagonalization argument to find a Seiberg-Witten trajectory $\gamma_{\infty} \colon \mathbb{R}\rightarrow S$ of finite type such that, after passing to a subsequence, $\gamma_{n}(0)\rightarrow \gamma_{\infty}(0)$ in $L^{2}_{k}$.  However, $\gamma_{\infty}(0)$ is in the interior of $Str(R_{0})$ by Theorem \ref{Boundedness of trajectories}. This is a contradiction.
\end{proof}

\begin{rmk} In Corollary~\ref{boundedness for approximated trajectories}, we can also consider more generalized approximated trajectories. For example, we can use interpolation between two projections for approximation, i.e. a trajectory satisfying
\begin{align*}
-\frac{d\gamma(t)}{dt}= \left(l+((1-s)p^{\mu}_{\lambda} + s p^{\mu'}_{\lambda'})\circ c\right)(\gamma(t)),
\end{align*}
where $0 \le s \le 1 $ and $\lambda'  < \lambda < \bar{\lambda}$ and $\mu' > \mu > \bar{\mu} $.
\end{rmk}

\section{Categorical and topological preliminaries} \label{Section Cat Top}

\subsection{The stable categories.}\label{subsection stable categories} In this subsection, we will briefly review algebraic-topological constructions which will be needed later. In particular, we will define three stable categories $\mathfrak{C},\mathfrak{S}$ and $\mathfrak{S}^{*}$ in which our invariants live as objects. The categories $\mathfrak{S}$
and $\mathfrak{S}^{*} $ are defined as direct systems and inverse systems of \(\mathfrak{C}\) respectively. Our treatment follows closely with \cite{Manolescu1} and \cite{Manolescu2}. See \cite{Adams} and \cite{May} for more systematic and detailed discussions regarding equivariant stable homotopy theory.

The category $\mathfrak{C}$, which was defined in \cite{Manolescu1}, is the $S^{1}$-equivariant analog of the classical Spanier-Whitehead category with $\mathbb{R}^\infty \oplus \mathbb{C}^\infty$ as the universe. In other words,  we will only consider suspensions involving the following two representations:

\begin{enumerate}
\item $\mathbb{R}$ the one-dimensional trivial representation;
\item $\mathbb{C}$ the two-dimensional representation where $S^{1}=\{e^{i\theta}|\theta\in [0,2\pi)\}$ acts
by complex multiplication.
\end{enumerate}

For a representation $V$, we will denote by $V^+$ its one-point compactification and by $V^{S^1}$ its $S^1$-fixed point set. Note that the transposition $(\mathbb{R}^{u_1})^+ \wedge (\mathbb{R}^{u_2})^+ \rightarrow (\mathbb{R}^{u_2})^+ \wedge (\mathbb{R}^{u_1})^+$ is homotopic to identity only when $u_1$ or $u_2$ is even.

The objects of $\mathfrak{C}$ are triples $(A,m,n)$ consisting of a pointed topological space $A$ with an $S^{1}$-action, an even integer $m$ and a rational number $n$. We require that $A$ is $S^{1}$-homotopy equivalent to a finite $S^{1}$-CW complex. The set of morphisms between two objects is given by
\[
    \morp_{\mathfrak{C}}((A,m,n),(A',m',n')):=
         \operatorname*{colim}_{u,v \rightarrow \infty}
              [(\mathbb{R}^{u}\oplus\mathbb{C}^{v})^{+}\wedge A,
                      (\mathbb{R}^{u+m-m'} \oplus \mathbb{C}^{v+n-n'})^{+} \wedge A' )]_{S^{1}},
\]
if $n - n' \in \mathbb{Z}$, where $[\cdot, \cdot]_{S^{1}}$ denotes the set of  pointed $S^{1}$-equivariant homotopy classes. We define $\morp_{\mathfrak{C}}((A,m,n),(A',m',n'))$ to be the empty set if $n - n' \not\in \mathbb{Z}$. As in \cite{Manolescu1}, there is a full subcategory $\mathfrak{C}_{0}$ inside of $\mathfrak{C}$ consisting of objects of the form $(A,0,0)$, which we also denote by $A$. For an object $Z=(A,m,n) \in \ob\mathfrak{C}$, an even integer $m'$ and a rational number $n'$, we also write $(Z,m',n')$ for $(A,m+m',n+n')$.

An $S^{1}$-representation $E$ is called \emph{admissible} if it is isomorphic to $\mathbb{R}^a \oplus \mathbb{C}^b$ for some nonnegative integers $a,b$. For such a representation, we can define the suspension functor $\Sigma^{E} \colon \mathfrak{C}\rightarrow\mathfrak{C}$ by setting $\Sigma^{E}(A,m,n):=(\Sigma^{E}A,m,n)$. For a morphism $F$, we define $\Sigma^{E} F$  to be the $U(1)$-equivariant stable homotopy class represented by a composition
\begin{equation}\label{suspending morphism}
\begin{split}
&(\mathbb{R}^{u}\oplus\mathbb{C}^{v})^{+}\wedge E^{+}\wedge A \xrightarrow{\sigma_{1,2}} E^{+}\wedge(\mathbb{R}^{u}\oplus\mathbb{C}^{v})^{+}\wedge A \xrightarrow{\text{id}_{E^{+}}\wedge f}\\
&E^{+}\wedge(\mathbb{R}^{u+m-m'}\oplus\mathbb{C}^{v+n-n'})^{+}\wedge A' \xrightarrow{\sigma_{1,2}} (\mathbb{R}^{u+m-m'}\oplus\mathbb{C}^{v+n-n'})^{+}\wedge E^{+}\wedge A',
\end{split}
\end{equation}
where $f: (\mathbb{R}^{u}\oplus\mathbb{C}^{v})^{+} \wedge A \rightarrow ( \mathbb{R}^{u+m-m'}\oplus\mathbb{C}^{v+n-n'})^{+} \wedge A' $ is a $U(1)$-equivariant map representing $F$ and $\sigma_{1,2}$ is the interchanging map of the first and the second factor.
We can also define the desuspension functor $\Sigma^{-E}\colon\mathfrak{C}\rightarrow \mathfrak{C}$ by setting $\Sigma^{-E}(A,m,n):=(\Sigma^{E^{S^{1}}}A,m+2a,n+b)$ and define $\Sigma^{-E} F$ as in (\ref{suspending morphism}) but replacing $E^{+}$  with $(E^{S^{1}})^{+}$. The following lemma is straightforward and we omit the proof.

\begin{lem}\label{lem susp functor commute} As functors of $\mathfrak{C} $, we have $\Sigma^{E_{1}}\circ \Sigma^{E_{2}}=\Sigma^{E_{1}\oplus E_{2}}$ , $\Sigma^{-E_{1}}\circ
\Sigma^{-E_{2}}=\Sigma^{-(E_{1}\oplus E_{2})}$.
\end{lem}

Furthermore, we show that suspension and desuspension are inverse of each other as functors in $\mathfrak{C}$.

\begin{pro}\label{suspension desuspension functors cancels}
For an admissible representations $E$, there is a natural isomorphism $\eta$ from the functor $\Sigma^{-E_{}}\circ \Sigma^{E_{}} $ to $\operatorname{id}_{\mathfrak{C}}$ ,
where $\operatorname{id}_{\mathfrak{C}}$ is the identity functor on $\mathfrak{C}$.
\end{pro}

\begin{proof}

For each object $(A,m,n)$ of $\mathfrak{C} $, we want to construct an isomorphism $\eta_{(A,m,n)}$ from $\Sigma^{-E_{}}\circ
\Sigma^{E_{}}(A,m,n)=((E^{S^{1}}_{}\oplus E_{})^{+}\wedge A,m+2a,n+b)$ to $(A,m,n)$. First, we choose an isomorphism $\tau \colon E \rightarrow \mathbb{R}^a \oplus \mathbb{C}^b$. Then $\tau$ induces an isomorphism $$\tilde{\tau} \colon (E_{}\oplus E^{S^{1}}_{})^{+}\xrightarrow{\tau\oplus \tau^{S^{1}}} (\mathbb{R}^{a}\oplus \mathbb{C}^{b}\oplus \mathbb{R}^{a})^{+}\xrightarrow{\sigma_{2,3}}(\mathbb{R}^{2a}\oplus \mathbb{C}^{b})^{+}.$$  Note that by equivariant Hopf theorem (cf. \cite[Section 2.4]{TomDieck}), the homotopy class of $\tilde{\tau}$ is independent of the choice of $\tau$.  The {isomorphism} $\eta_{(A,m,n)} : \Sigma^{-E} \circ \Sigma^{E}(A,m,n) \rightarrow (A, m, n)$ is the $U(1)$-equivariant  stable homotopy class of
\begin{equation*}
          (\mathbb{R}^{u}\oplus \mathbb{C}^{v})^{+}\wedge (E^{S^{1}} \oplus E)^{+}\wedge A
           \xrightarrow{\tilde{\tau}_{u,v} \wedge \operatorname{id}_A}
          (\mathbb{R}^{u+2a}\oplus \mathbb{C}^{v+b})^{+}\wedge A,
\end{equation*}
where $\tilde{\tau}_{u,v}$ is the composition of $\operatorname{id}_{(\mathbb{R}^{u}\oplus \mathbb{C}^{v})^{+}} \wedge\ \tilde{\tau} $ with the transposition $(\mathbb{R}^{u}\oplus \mathbb{C}^{v})^{+} \wedge (\mathbb{R}^{2a}\oplus \mathbb{C}^{b})^{+} \rightarrow (\mathbb{R}^{u+2a}\oplus \mathbb{C}^{v+b})^{+} $.

For a map $f\colon (\mathbb{R}^{u}\oplus \mathbb{C}^{v})^{+}\wedge A\rightarrow (\mathbb{R}^{u+m-m'}\oplus \mathbb{C}^{v+n-n'})^{+}\wedge A'$, we want to check that $f \circ\eta_{(A,m,n)}  = \eta_{(A',m',n')}^{} \circ \Sigma^{-E} \circ \Sigma^{E}f$ up to stable homotopy. This follows from a commutative diagram
\begin{equation*}
\xymatrix{
        (\mathbb{R}^{u} \oplus \mathbb{C}^{v})^{+}   \wedge (E^{S^{1}}  \oplus E )^{+}   \wedge A
         \ar[d]_{\operatorname{id}_{(\mathbb{R}^{u}\oplus \mathbb{C}^{v})^{+}} \wedge \tilde{\tau} \wedge \operatorname{id}_A}
           \ar[r]^{\sigma_{1,2}} &
        (E^{S^{1}}  \oplus E )^{+} \wedge (\mathbb{R}^{u}\oplus \mathbb{C}^{v})^{+} \wedge A
          \ar[d]^{\tilde{\tau} \wedge \operatorname{id}_{(\mathbb{R}^{u}\oplus \mathbb{C}^{v})^{+}} \wedge \operatorname{id}_{A}}
  \\
   (\mathbb{R}^{u} \oplus \mathbb{C}^{v})^{+}    \wedge    (\mathbb{R}^{2a}\oplus \mathbb{C}^{b})^{+}    \wedge   A
    \ar[r]^{\sigma_{1,2}} &
    (\mathbb{R}^{2a}\oplus \mathbb{C}^{b})^{+}   \wedge    (\mathbb{R}^{u}\oplus \mathbb{C}^{v})^{+}      \wedge A
}
\end{equation*}
and a similar diagram for $A'$ and the fact that the transpositions in the diagrams are homotopic to identity.

\end{proof}

We now turn to the description of the category $ \mathfrak{S}$. An object of $\mathfrak{S}$ consists of a collection $Z = ( \{Z_p \}, \{ i_p \})_{p \in \mathbb{N}})$ of objects $\{ Z_p \}_{p \in \mathbb{N}}$ of \(\mathfrak{C}\)  and morphisms $\{ i_p \in \morp_{\mathfrak{C}}(Z_{p},Z_{p+1}) \}_{p \in \mathbb{N}} $. In other word, an object $Z$ of \(\mathfrak{S}\) is a direct system
\begin{equation*}\label{object of S}
                 Z_{1}   \xrightarrow{i_{1}} Z_{2}  \xrightarrow{i_{2}} \cdots.
\end{equation*}
For two objects $Z = ( \{ Z_p \}_p, \{ i_p \}_p)$ and $Z' = ( \{ Z_p' \}_p, \{ i_p' \}_p)$ of $\mathfrak{S}$,  we define the set of morphisms as
\begin{equation}\label{the morphism in S}
       \morp_{\mathfrak{S}}(Z,  Z')
      :=\mathop{\lim}\limits_{\infty\leftarrow p}\mathop{\lim}\limits_{q\rightarrow \infty}\morp_{\mathfrak{C}}(Z_{p},Z'_{q}).
\end{equation}
The identity morphism and the composition law are defined in the obvious way. Notice that here we first take the direct limit and then take the inverse limit. This order should not be changed.

As for the category $\mathfrak{S}^{*}$, its objects are the inverse systems
$$
      \bar{Z}_{1}  \xleftarrow{j_1}   \bar{Z}_{2}  \xleftarrow{j_2} \cdots,
$$
where $\bar{Z}_{p}\in \ob\mathfrak{C}$ and $j_{p}\in \morp_{\mathfrak{C}}(\bar{Z}_{p+1},\bar{Z}_{p})$.
 For two objects $\bar{Z} = ( \{  \bar{Z}_p \}_p, \{ j_p \}_p)$ and $\bar{Z}' = ( \{ \bar{Z}_p' \}_p, \{ j_p' \}_{p})$ of $\mathfrak{S}^*$,  we define  the set of morphisms as
\begin{equation}\label{moprhism in S*}
      \morp_{\mathfrak{S}^{*}}(\bar{Z}, \bar{Z}')
      :=\mathop{\lim}\limits_{\infty\leftarrow q}\mathop{\lim}\limits_{p\rightarrow \infty}\morp_{\mathfrak{C}}(\bar{Z}_{p},\bar{Z}'_{q}).
      \end{equation}
Again, we first take the direct limit and then take the inverse limit.


The suspension functor and the desuspension functor can be extended to functors on $\mathfrak{S}$ and $\mathfrak{S}^{*}$ in the obvious way. Lemma~\ref{lem susp functor commute} and Proposition~\ref{suspension desuspension functors cancels} continue to hold for these extended functors.  For an object $Z = ( \{ Z_p \}_p, \{ i_p \}_p)$ of  $\mathfrak{S}$, an even integer $m$ and a rational number $n$, we write $(Z, m, n)$ for $( \{ (Z_p, m, n) \}_p,  \{ i_p' \}_p)$, where $i_p':(Z_p, m,n) \rightarrow (Z_{p+1}, m, n)$ is the morphism induced by $i_p$.  For an object $\bar{Z}$ of $\mathfrak{S}^*$, we define $(\bar{Z}, m, n)$ similarly.

\begin{rmk}
The full subcategory of $\mathfrak{C}$ consisting of objects $\{(A,m,n) \mid m\in 2\mathbb{Z}, n \in \mathbb{Z} \}$ can be naturally embedded into the homotopy category of the $S^{1}$-equivariant spectra modeled on the standard universe $\mathbb{R}^{\infty}\oplus \mathbb{C}^{\infty}$. Therefore, an object $( \{ (A_p, m_p, n_p) \}_{p}, \{ i_p \}_p)$ of $\mathfrak{S}$ (resp. $\mathfrak{S}^{*}$)  with $m_p\in 2\mathbb{Z}$ and $n_p \in \mathbb{Z}$   corresponds to an  inductive system (resp. projective system) of $S^{1}$-equivariant spectra. For this reason, we call an object of $\mathfrak{S}$ an  ind-{spectrum} and  an object of $\mathfrak{S}^{*}$ a pro-spectrum. However, this is not so accurate because, in the usual sense, an ind-spectrum (resp. pro-spectrum) refers to an inductive system (resp. projective system) in the category of spectra, not the homotopy category of spectra. Also, with a slightly abuse of language, we call all our invariants spectrum invariants.
\end{rmk}
We end this subsection with the following useful lemma, which is directly implied by the definition of the direct limit and inverse limit.

\begin{lem}\label{subsyemtem}
Let $ Z =( \{ Z_p \}_{p \in \mathbb{N}}, \{ i_p \}_{p \in \mathbb{N}})$ be an object of $\mathfrak{S}$.  For any infinite sequence of positive integers $0<p_{1}<p_{2}< \cdots$,  the subsystem
$$
      Z_{p_{1}}    \xrightarrow{ i_{p_{2}-1} \circ  \cdots \circ i_{p_{1}}  }    Z_{p_{2}}
                         \xrightarrow {i_{p_{3}-1}  \circ \cdots \circ i_{p_{2}}}      Z_{p_{3}}    \rightarrow \cdots
$$
of $Z$ is canonically isomorphic to $Z$ as an object of $\mathfrak{S}$.
Similarly,  let $\bar{Z} = ( \{ \bar{Z}_{p} \}_{p \in \mathbb{N}}, \{ j_p \}_{p \in \mathbb{N}})$ be an object $\mathfrak{S}^*$, then the subsystem
$$
       \bar{Z}_{p_{1}}  \xleftarrow{ j_{p_{1}}  \circ \dots \circ j_{p_{2}-1}  }     \bar{Z}_{p_{2}}
                                 \xleftarrow { j_{p_{2}} \circ \cdots \circ j_{p_{3}-1} }    \bar{Z}_{p_{3}}   \leftarrow \cdots
$$
of $\bar{Z}$ is canonically isomorphic to $\bar{Z}$ as an object of $\mathfrak{S}^{*}$.
\end{lem}


\subsection{The Conley index}\label{subsection conley index}

In this section, we recall basic facts regarding the Conley index theory.
See \cite{Conley} ,\cite{Manolescu1} and \cite {Salamon} for more details.

Let $V$ be a finite dimensional manifold and $\varphi$ be a smooth flow on
$V$, i.e. a $C^{\infty}$-map $\varphi \colon V\times \mathbb{R}\rightarrow V$ such
that $\varphi(x,0)=x$ and $\varphi(x,s+t)=\varphi(\varphi(x,s),t)$ for any
$x\in V$ and $s,t \in \mathbb{R}$. We denote by $\inv(\varphi,A) := \{ x \in A \mid \varphi(x ,
\mathbb{R}) \subset A  \}$  the maximal invariant set  of $A$. We
sometimes write $\inv(A)$ when the flow $\varphi$ is obvious from the context.

A compact set $A \subset V$ is called an \emph{isolating neighborhood} if $\inv(A)$ lies in the interior
of $A$. A compact set $S\subset V$ is called an \emph{isolated invariant set} if there exists
an isolating neighborhood $A$ such that $\inv(A)=S$. In this situation, we also say that $A$ is
an isolating neighborhood of $S$. For an isolated invariant set $S$, a pair $(N,L)$ of compact
sets $L\subset N$ is called an \emph{index pair} of $S$ if the following conditions hold:

\begin{enumerate}[(i)]
\item  $\inv(N\setminus L)=S\subset \inti(N\setminus L)$, where $\inti(N \setminus L)$ is the interior of $N\setminus L$;
\item  $L$ is an exit set for $N$, i.e. for any $x\in N$ and $t>0$ such
that $\varphi(x,t)\notin N$, there exists $\tau\in [0,t)$ with $\varphi(x,\tau)\in
L$;
\item  $L$ is positively invariant in $N$, i.e. for $x\in L$ and $t>0$, if
we know $\varphi(x,[0,t])\subset N$, then we have $\varphi(x,[0,t])\subset
L$.
\end{enumerate}

We list two fundamental facts regarding index pairs:
\begin{itemize}
\item For an isolated invariant set $S$ with an isolating neighborhood $A$,
we can always find an index pair $(N,L)$ of $S$ such that $L\subset N\subset A$.
\item The pointed homotopy type of $N/L$ with $[L]$ as a base point only depends on $S$ and $\varphi$.
More precisely, for any two index pairs $(N,L)$ and $(N',L')$ of $S$,
there is a natural pointed homotopy equivalence $N/L\rightarrow N'/L'$ induced by the flow.
\end{itemize}
These lead to us the definition of the Conley index.

\begin{defi} \label{def conleyindex} Given an isolated invariant set $S$ of a flow $\varphi$, we denote
by $I(\varphi,S,N,L)$ the pointed space of $(N/L,[L])$, where $(N,L) $ is an index pair of $S$. This is called the \emph{Conley index} of $S$. We will always suppress $(N,L)$ from our notation and write $I(\varphi,S)$ instead. We may also write $I(S)$ when the flow is clear from the
context.
\end{defi}

\begin{rmk}\label{simple connected system}In \cite{Salamon}, the Conley {index} was defined as a connected simple system of pointed spaces. I.e., a collection of pointed spaces (given by different index pairs) together with natural homotopy equivalences between them (given by the flow map). In Definition \ref{def conleyindex}, we actually pick a representative of this connected simple system by making a choice of the index pair $(N,L)$. As we will see in {the} next section, we need to make choices of all kinds of index pairs in our construction of spectrum invariants. Just like the Riemannian metric $g$ and the perturbation on $f$, these choices will be treated as auxiliary data involved in the construction and we will prove that our spectrum invariant is independent of this data upto canonical isomorphism.
\end{rmk}

We further provide relevant properties of the Conley index.

\begin{enumerate}

\item(Product flow) If $\varphi_{j}$ is a flow on $V_{j}$ for $j=1,2$ and $S_{j}$ is an
isolated invariant set for $\varphi_{j}$, then we have a canonical homotopy equivalence $I(\varphi_{1}\times
\varphi_{2},S_{1}\times S_{2})\cong I(\varphi_{1},S_{1})\wedge I(\varphi_{2},S_{2})$,
where ``$\wedge$'' is the smash product.

\item (Continuation) Let $\varphi_t$ is a continuous family of flows parametrized by $t \in [0, 1]$.
Suppose  that $A$ is an isolating neighborhood of $\varphi_t$ for any $t \in [0, 1]$, and let $S_t$ be $\inv (\varphi_t, A)$.
Then we have a canonical homotopy equivalence $I(\varphi_{0},S_{0})\cong
I(\varphi_{1},S_{1})$.

\end{enumerate}

The following concept will be useful for explicitly computing the Conley index.

\begin{defi}[\cite{Krzysztof}] For a compact subset $A$, we
consider the following subsets of
its boundary
\begin{align*}
n^{+}(A) &:= \{x\in \partial A \mid \exists \epsilon>0 \text{ s.t }\varphi(x,(-\epsilon,0))\cap
A=\emptyset\}, \\
n^{-}(A) &:= \{x\in \partial A \mid \exists \epsilon>0 \text{ s.t }\varphi(x,
(0,\epsilon))\cap A=\emptyset\}.
\end{align*}
A compact subset $N$ is called an \emph{isolating block}  if $\partial
N=n^{+}(N)\cup n^{-}(N)$.
\end{defi}

It is easy to verify that an isolating block is an isolating neighborhood.
When $N$ is an isolating block, its index pair can be given by $(N,n^{-}(N))$.



Next, we consider a situation when an isolated invariant set can be decomposed to smaller isolated invariant
sets.

\begin{defi} \ \begin{enumerate}[(i)]
\item  For
a subset $A\subset V$, we define its $\alpha$-limit and $\omega$-limit
set as$$
\alpha(A)=\mathop{\cap}_{t<0}\overline{\varphi(A,(-\infty,t])} \quad \text{and } \quad
\omega(A)=\mathop{\cap}_{t>0}\overline{\varphi(A,[t,+\infty))}.
$$
\item Let $S$ be an isolated invariant set. A subset $T\subset
S$ is called an
\emph{attractor} (resp. \emph{repeller}) if there exists a neighborhood $U$
of $T$ in $S$ such that $\omega(U)=T$ (resp. $\alpha(U)=T$).
\item When $T$
is an attractor in $S$, we define the set $T^{*}:=\{x\in S \mid \omega(x)\cap
T=\emptyset\}$, which is a repeller in $S$. We call $(T,T^{*})$ an \emph{attractor-repeller
pair} in $S$.
\end{enumerate}
\end{defi}

Note
that an attractor and a repeller are always an isolated invariant sets.
We give an
important result relating Conley indices of an attractor-repeller pair.
\begin{pro}[Salamon \cite{Salamon}]\label{Attractor-repeller-exact sequence}Let
$S$ be an isolated invariant set with an isolating neighborhood $A$ and $(T,T^{*})$
be an attractor-repeller pair in $S$. Then there exist compact sets $\tilde{N}_{3}\subset
\tilde{N}_{2}\subset \tilde{N}_{1}\subset A$ such that the pairs $(\tilde{N}_{2},\tilde{N}_{3}),
(\tilde{N}_{1},\tilde{N}_{3}),(\tilde{N}_{1},\tilde{N}_{2})$ are index pairs
for $T,$ $S$ and $T^*$ respectively. The maps induced by inclusions give a natural coexact sequence of
Conley indices
$$
I(\varphi,T)\xrightarrow{i_{1}} I(\varphi,S)\xrightarrow{i_{2}}
I(\varphi,T^{*})\rightarrow \Sigma I(\varphi,T)\rightarrow \Sigma I(\varphi,
S) \rightarrow \cdots.
$$
We call the triple $(\tilde{N}_{3},\tilde{N}_{2},\tilde{N}_{1})$
an index triple for the pair $(T,T^{*})$ and call the maps $i_{1}$ and $i_{2}$ the attractor map and
the repeller map
respectively.
\end{pro}

By Corollary~4.4 of \cite{Salamon}, the attractor maps are transitive in
the following sense. Suppose that $S_{1}$ is an attractor in $S_{2}$ and $S_{2}$
is an attractor in $S_{3}$. Then $S_{1}$ is also an attractor in $S_{3}$.
Moreover, the corresponding attractor maps
$$
i_{1}:I(\varphi,S_{1})\rightarrow
I(\varphi,S_{2}), \; i_{1}':I(\varphi,S_{2})\rightarrow I(\varphi,S_{3})
\text{ and } i_{1}'':I(\varphi,S_{1})\rightarrow I(\varphi,S_{3})
$$
satisfy the relation $i''_{1} = i'_{1} \circ i_{1}$. Similar statements hold for the repeller maps.

%

Lastly, we briefly discuss the equivariant Conley index
theory, which has been developed in \cite{Floer1} and \cite{Pruszko}. Let $G$ be a compact Lie group
acting on $V$ while preserving the flow
$\varphi$.  For a $G$-invariant isolated invariant set $S$, we can find a $G$-invariant
isolating neighborhood as well as a $G$-invariant index pair $(N,L)$.
As in the non-equivariant case, with the choice of $(N,L)$, we denote by $I_{G}(\varphi,S)$ the pointed $G$-space $(N/L,[L])$, whose $G$-equivariant homotopy
type only depends on $S$ and $\varphi$. In particular, $I_{G}(\varphi,S)$ is the \emph{$G$-equivariant
Conley index} of $S$.
All the non-equivariant results stated above can be adapted to the $G$-equivariant
setting. From now on, we will work on this equivariant setting with $G=S^{1}\text{
or }Pin(2)$.

\section{Construction of the spectrum invariants}\label{construction of spectrum}
In this section, we define different versions of unfolded Seiberg-Witten-Floer spectra for the $\text{spin}^{c}$
manifold $(Y,\mathfrak{s})$. We first define the spectrum  $\underline{\text{swf}}^{A}(Y,\mathfrak{s},A_{0},g;S^{1})$
and $\underline{\text{swf}}^{R}(Y,\mathfrak{s},A_{0},g;S^{1})$ for a general spin$^{c}$ structure $\mathfrak{s}$.
In Section \ref{subsection torsion case}, we consider the situation when $\mathfrak{s}$ is torsion and define normalized spectra $\underline{\text{SWF}}^{A}(Y,\mathfrak{s};S^{1})$
and $\underline{\text{SWF}}^{R}(Y,\mathfrak{s};S^{1})$ which are independent of the choices of base connection $A_0$ and metric $g$. In Section \ref{subsection Pin(2)-spectrum}, we deal with the $Pin(2)$-equivariant
case for a spin structure $\mathfrak{s}$ and define $\underline{\text{SWF}}^{A}(Y,\mathfrak{s};Pin(2))$,
$\underline{\text{SWF}}^{R}(Y,\mathfrak{s};Pin(2))$.

\subsection{The spectrum invariants for general $\text{spin}^{c}$ structures.}  \label{Subsection general spin^c}

The main idea of the construction follows \cite{Kronheimer-Manolescu} and \cite{Khandhawit2}. In summary, we want to apply finite dimensional approximation of Conley indices to the set $Str(R)$ which contains all critical points and flow lines between them. However, the set $Str(R)$ is unbounded owing to the action of $\mathcal{G}^{h}_{Y}$. We then need to introduce transverse functions and use their level sets to obtain a collection of bounded subsets of $Str(R)$.

Notice that the space of imaginary-valued harmonic 1-forms, denoted
by $i\Omega^{1}_{h}(Y)$, is a subspace of $Coul(Y)$. Let $p_{\mathcal{H}} \colon Coul(Y)\rightarrow
i\Omega^{1}_{h}(Y)$ be the $L^{2}$-orthogonal projection.
Here, we identify $i\Omega^{1}_{h}(Y)$ with $\mathbb{R}^{b_1}$ by choosing harmonic forms $\{ h_1 , h_2 , \ldots h_{b_1} \} \subset i\Omega^{1}_{h}(Y)$ representing a set of free generators of the group
$$
        2\pi i \im  (H^{1}(Y;\mathbb{Z}) \rightarrow H^{1}(Y;\mathbb{R})) \cong \mathbb{Z}^{b_1}.
$$
With this identification, we can write the projection as
\[
     p_{\mathcal{H}} = (p_{\mathcal{H}, 1}, \dots ,p_{\mathcal{H}, b_1}).
\]
From now on, we assume that our perturbation $f$ is good (see Definition \ref{good perturbation}).
Together with the compactness result \cite[Theorem~10.7.1]{Kronheimer-Mrowka}, the critical
points of $\mathcal{L}$ in $Coul(Y) $ is finite modulo the action of $\mathcal{G}^{h}_{Y}$.
Consequently, we can find a small interval $[r,s] \subset (0,1)$ such that $\bigcup_{j=1}^{b_1} p_{\mathcal{H}, j}^{-1}( [-s,-r] \cup [r, s] )$  contains no critical point of $\mathcal{L}$. Let us pick a positive number $\tilde{R}$ greater than  the universal constant $R_{0}$ from Theorem \ref{Boundedness of trajectories}.

\begin{lem}\label{lem no critical points rs}
There exists a positive number $\tilde{\epsilon}>0$ such that we have $\|\gradtil \mathcal{L}(x)\|_{\tilde{g}}>\tilde{\epsilon}$ for any $x\in \left(\bigcup_{j=1}^{b_1} p_{\mathcal{H}, j}^{-1}( [-s, -r] \cup [r,s] ) \right) \cap Str(\tilde{R})$.
\end{lem}

\begin{proof}
Suppose that the result is not true. We can then find a sequence $\{x_{n}\}$ contained in $\left(\bigcup_{j=1}^{b_1} p_{\mathcal{H}, j}^{-1}([-s,-r] \cup [r, s]) \right) \cap Str(\tilde{R})$ with $\|\gradtil  \mathcal{L}(x_{n})\|_{\tilde{g}}\rightarrow 0$. Notice that the sequence $\{ x_n \} $ is contained in $p_{\mathcal{H}}^{-1}([-1,1]^{b_1})\cap Str(\tilde{R})$, which is bounded in $L^2_k$. Hence,  after passing to a subsequence, $x_n$ converges to some point $x_{\infty}$ of $Coul(Y)$ weakly in $L^2_k$ and strongly in $L^{2}_{k-1}$ by Rellich lemma.
Consequently, we have $p_{\mathcal{H}}(x_{n}) \rightarrow p_{\mathcal{H}}(x_{\infty})$ and  $\gradtil  \mathcal{L} (x_{\infty}) = 0$ by continuity.
This is a contradiction since $x_{\infty}$ is a critical point of $\gradtil{\mathcal{L}}$ and lies in $\bigcup_{j = 1}^{b_1} p_{\mathcal{H}, j}^{-1} ( [-s, -r] \cup [r, s])$.
\end{proof}

Note that $\tilde{\epsilon} $ in the above lemma depends on the choice of $r, s$ and $\tilde{R}$. With these data, we choose a smooth ``staircase'' function $\bar{g} \colon \mathbb{R}\rightarrow [0,\infty)$ satisfying the following properties:

\begin{enumerate}[(i)]
\item  $\bar{g}$ is even, i.e. $\bar{g}(x) = \bar{g}(-x)$ for all $x \in \mathbb{R}$;
\item There is a positive constant $\bar{\tau}$ such that $\bar{g}(x+1)=\bar{g}(x)+ \bar{\tau}$ for all $x \in [0,\infty)$;
\item $\bar{g}$ is increasing on the interval \([r,s]\) and $\bar{g}' =0 $ on $[0,r]\cup[s,1]$;
\item
$|\bar{g}'(x)| < \tilde{\epsilon} \cdot \epsilon''$ for all $x\in \mathbb{R}$, where $\epsilon''$ is a positive constant with the property that
\begin{equation}\label{stretch factor}
\epsilon''\cdot \|\sum_{j=1}^{b}a_{j}h_{j}\|_{L^{2}}\leq (\sum_{j=1}^{b} a_{j}^{2})^{\frac{1}{2}},\ \forall (a_{1},a_{2},...,a_{b})\in \mathbb{R}^{b}.
\end{equation}

\end{enumerate}
\begin{figure}[htbp]
\centering
\includegraphics[angle=0,width=1.0\textwidth]{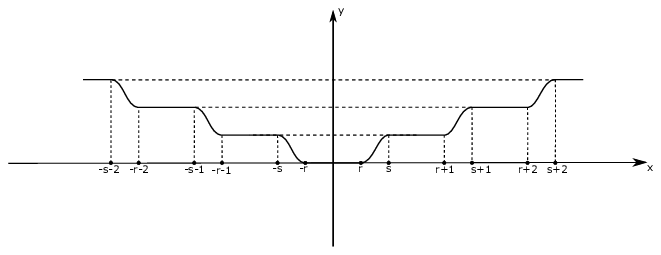}
\caption{the function $\bar{g}$} \label{meinv}
\end{figure}

Next we use the function $\bar{g}$ to define a small perturbation of \(\mathcal{L}\) which is not invariant under $\mathcal{G}^{h}_{Y}$ but transverse to level sets of \(\mathcal{L}\). For each $j = 1 , \ldots, b_1$, we define
\begin{equation*}
g_{j,+}=\bar{g}\circ p_{\mathcal{H}, j}+\mathcal{L} \text{ and } g_{j,-}=\bar{g}\circ p_{\mathcal{H}, j}-\mathcal{L}.
\end{equation*}
With our assumptions on $\bar{g}$, we have the following result.

\begin{lem} \label{lem g_j+ transverse} For each $j= 1 , \ldots, b_1$, we have
\begin{equation*}
\langle \gradtil \mathcal{L}(x),\gradtil g_{j,+}(x)\rangle_{\tilde{g}} \geq 0 \text{ and }
\langle \gradtil \mathcal{L}(x),\gradtil g_{j,-}(x)\rangle_{\tilde{g}} \leq 0,
\end{equation*}
where the equalities hold only when $x$ is a critical point of $\mathcal{L}$.
\end{lem}

\begin{proof}
By (\ref{two versions of gradient}) and a straightforward computation, we can prove that $$\left\Vert \gradtil { (\bar{g}\circ p_{\mathcal{H}, j}}) (x) \right\Vert_{\tilde{g}}=\|\grad { (\bar{g}\circ p_{\mathcal{H}, j}})(x)\|_{L^{2}}\leq \frac{1}{\epsilon''}\cdot |\bar{g}'(p_{\mathcal{H}, j}(x))|<\tilde{\epsilon}.$$

If $ \left\vert p_{\mathcal{H}, j}(x) \right\vert \in [n, n+r]$ or $\left\vert p_{\mathcal{H}_j}(x)\right\vert \in [n + s, n+1]$ for some integer $n$, then  $\bar{g}'(p_{\mathcal{H}, j}(x)) =0$ and $\langle \gradtil \mathcal{L}(x),\gradtil g_{j,+}(x)\rangle_{\tilde{g}}
= \|\gradtil \mathcal{L}(x)\|^2_{\tilde{g}}$ which is zero if and only if $x$ is a critical point of $\mathcal{L}$.  Otherwise, $ \left\vert p_{\mathcal{H}, j}(x)\right\vert
\in [n + r, n+s]$ for some integer $n$ and Lemma~\ref{lem no critical points rs} implies
\begin{align*}
      \langle \gradtil \mathcal{L}(x),\gradtil g_{j,+}(x)\rangle_{\tilde{g}}
     &= \|\gradtil \mathcal{L}(x)\|^2_{\tilde{g}}
      + \left\langle \gradtil \mathcal{L}(x) , \gradtil { (\bar{g}\circ  p_{\mathcal{H}, j}}) (x) \right \rangle_{\tilde{g}} \\
       &\geq \|\gradtil \mathcal{L}(x)\|^2_{\tilde{g}} -\left\Vert \gradtil { (\bar{g}\circ p_{\mathcal{H}, j}}) (x) \right\Vert_{\tilde{g}}\cdot \left\Vert \gradtil \mathcal{L}(x) \right\Vert_{\tilde{g}} \\
      & > \|\gradtil \mathcal{L}(x)\|_{\tilde{g}} \left( \|\gradtil \mathcal{L}(x)\|_{\tilde{g}} - \tilde{\epsilon}  \right)
         > 0.
\end{align*}
The same argument applies to the inner product $\langle \gradtil \mathcal{L}(x),\gradtil g_{j,-}(x)\rangle_{\tilde{g}} $.

\end{proof}
Since the number of critical points of $\mathcal{L}$ is finite modulo gauge, we can find a real number $\theta\in \mathbb{R}$ such that $g_{j,\pm}^{}(x) \neq \theta$  for any critical point $x$ of $\mathcal{L}$  and $j \in \{ 1 ,2 , \ldots , b_1 \} $. For convenience, we also  choose a decreasing sequence of negative real numbers $\{\lambda_{n}\}$ and an increasing sequence of positive real numbers $\{\mu_{n}\}$ such that $-\lambda_{n}, \mu_n \rightarrow \infty$. We are now ready to define a collection of bounded sets in $Str(\tilde{R})$.

\begin{defi} With the choice of $\tilde{R}, \bar{g}$ and $\theta$ above, we
 define the sets
 \begin{equation}  \label{cutting set for attractor}
   \begin{split}
        J^{+}_{m} &:= Str(\tilde{R})  \cap \bigcap_{1\leq j \leq b_1}g^{-1}_{j,+}(-\infty,\theta+m], \\
         J^{-}_{m} &:= Str(\tilde{R})   \cap \bigcap_{1\leq j\leq b_1}g^{-1}_{j,-}(-\infty,\theta+m],
   \end{split}
 \end{equation}
for each positive integer $m$. This collection of $J^+_m$ (resp. $J^-_m$) will be called a \emph{positive (resp. negative) transverse system}.
With the choice of  $\{ \lambda_{n} \}$ and  $\{ \mu_{n} \}$ , we also define
\[
        J_{m}^{n,\pm} :=J_{m}^{\pm}\cap V^{\mu_{n}}_{\lambda_{n}}.
\]
\end{defi}

 Notice that the functional $\mathcal{L}$ is bounded on $Str(\tilde{R})$, and the perturbed functional $g_{j,\pm}$ is bounded below on $Str(\tilde{R})$. Since a subset $S\subset Str(\tilde{R})$ is bounded if and only if $p_{\mathcal{H}}(S)$ is bounded, we can see that
 the set $J^{\pm}_{m}$ is bounded in the $L^{2}_{k}$-norm.

We will start to derive some properties of the finite-dimensional bounded sets $J_{m}^{n,\pm} $.
Although some of the following results are slightly stronger than what we need to define the 3-dimensional invariants,  they will be useful when we develop the 4-dimensional theory and prove the gluing theorem in \cite{KLS2, KLS3}.

\begin{lem}\label{approximated flow goes in}
For any positive integer $m$, there exist positive real numbers $\epsilon_{m}, \theta_{m}$ and an integer $N_{m} \gg 0 $ such that for any $n>N_{m} $ and $1\leq j \leq b_1$ we have
\begin{enumerate}[(i)]
\item $\langle(l+p^{\mu_{n}}_{\lambda_{n}}\circ c )(x),\gradtil  g_{j,+}(x)\rangle_{\tilde{g}}>\epsilon_{m}$
 for any $x\in J^{n,+}_{m}\cap g^{-1}_{j,+}[\theta+m-\theta_{m},\theta+m]$;

\item $\langle(l+p^{\mu_{n}}_{\lambda_{n}}\circ c )(x),\gradtil  g_{j,-}(x)\rangle_{\tilde{g}}<-\epsilon_{m}$
for any $x\in J^{n,-}_{m}\cap g^{-1}_{j,-}[\theta+m-\theta_{m},\theta+m]$.\end{enumerate}
\end{lem}

\begin{proof}
We only prove this lemma for $g_{1,+}$ and the other cases can be proved similarly. Suppose that the result is not true, then we can find sequences $n_{i}\rightarrow +\infty$, $\epsilon_{m,i},\theta_{m,i}\rightarrow 0$ and $\{x_{i}\}$ with $x_{i}\in J^{n_{i},+}_{m}\cap g^{-1}_{1,+}[\theta+m-\theta_{m,i},\theta+m]$ and $\langle(l+p^{\mu_{n_{i}}}_{\lambda_{n_{i}}}\circ c )(x_{i}),\gradtil  g_{1,+}(x_{i})\rangle_{\tilde{g}}\le \epsilon_{m,i}$. Since $\{x_{i}\}$ is contained in the $L^{2}_{k}$-bounded set $J_{m}^{+}$, we can pass to a convergent subsequence $x_{i}\rightarrow x_{\infty}$ in $L^{2}_{k-1}$ by the Rellich lemma. By continuity, we have $x_{\infty}\in g_{1,+}^{-1}(\theta+m)$ and $\gradtil g_{1,+}(x_{i})\rightarrow \gradtil g_{1,+}(x_{\infty})$ in $L^{2}_{k-2}$. Since $p^{\mu_{n}}_{\lambda_{n}}$ converges to the identity map pointwise, we also have $(l+p^{\mu_{n_{i}}}_{\lambda_{n_{i}}} \circ c )(x_{i}) \rightarrow (l+c)(x_{\infty}) = \gradtil \mathcal{L} (x_{\infty})$ in $L^{2}_{k-2}$. Therefore, we obtain
$$
        \langle(l+p^{\mu_{{ n_i}}}_{\lambda_{{ n_i}}}  \circ c )(x_{{n_i}}),\gradtil  g_{j,+}(x_{{ n_i}})\rangle_{\tilde{g}}
             \rightarrow
        \langle   \gradtil \mathcal{L} (x_{\infty}),\gradtil  g_{j,+}(x_{\infty})\rangle_{\tilde{g}},
$$
which implies that $\langle  \gradtil \mathcal{L}(x_{\infty}),  \gradtil  g_{j,+}(x_{\infty})\rangle_{\tilde{g}}  \le 0$ and $x_{\infty}$ is a critical point by Lemma~\ref{lem g_j+ transverse}. This is a contradiction with the choice of $\theta$.
\end{proof}

Now we start applying the Conley index theory to the flow on $V^{\mu_{n}}_{\lambda_{n}}$ generated by the vector field $-(l+p^{\mu_{n}}_{\lambda_{n}}\circ c)$. There is a technical point here. Since $V^{\mu_{n}}_{\lambda_{n}}$ is non-compact, this flow may go to infinity within a finite time. As in \cite{Manolescu1}, we can fix this by choosing a bump function $\iota_{m} \colon Coul(Y)\rightarrow [0,1]$ for each $m$ such that $\iota_{m}$ is supported in a bounded subset of $Coul(Y)$ and  $J^{\pm}_{m+1}$ is contained in the interior of $\iota_{m}^{-1}(1)$. We denote by $\varphi^{n}_{m}$ the flow on $V^{\mu_{n}}_{\lambda_{n}}$ generated by $-\iota_{m}\cdot (l+p^{\mu_{n}}_{\lambda_{n}}\circ c)$. Note that the flow $\varphi^n_{m'} $ on $J^{n, \pm}_m $ does not depend on $m'$ whenever $m' \ge m-1 $ so that its invariant set and its Conley index remain unchanged.

\begin{lem}\label{lem J^n,+m isolate}
For a positive integer $M$, there exist large numbers $N,T$ such that, for any positive integers $m\leq M$ and $n\geq N$, we have the following statements.
\begin{enumerate}[(a)]
\item \label{item J^n,+m isolate1} If $\gamma_{} \colon [-T,T]\rightarrow V^{\mu_{n}}_{\lambda_{n}}$ is an approximated Seiberg-Witten trajectory contained in $J^{n,+}_{m}$, then we have
$$
      \gamma_{}(0)\in Str(R_{0}) \cap \bigcap_{1 \leq j \leq b_1}g^{-1}_{j,+}(-\infty, \theta + m - \theta_{m}].
$$
In particular, $J^{n,+}_{m}$ is an isolating neighborhood for the flow $\varphi^{n}_{m}$.
\item \label{item J^n,+m isolate2} The set $\inv(\varphi^{n}_{m-1},J^{n,+}_{m-1})$ is an attractor in $\inv(\varphi^{n}_{m},J^{n,+}_{m})$ with respect to the flow $\varphi^{n}_{m} $.
\end{enumerate}
\end{lem}

\begin{proof}
Let $\bar{T}$, $\bar{\lambda}$ and $\bar{\mu}$ be the large constants from Corollary~\ref{boundedness for approximated trajectories} with $S=J_{M}^{+}$.  Let $\theta_{m},\epsilon_{m}$ and $N_{m}$ be the constants obtained from Lemma~\ref{approximated flow goes in} for $m=1, \ldots ,M$. Put $T=\max\{ \bar{T}, \frac{\theta_{1}}{\epsilon_{1}}, \dots,\frac{\theta_{M}}{\epsilon_{M}}\}+1$. We choose a positive integer $N > \max{\{N_1 , \ldots , N_M \}}$ such that $\lambda_N < \bar{\lambda}$ and $\mu_N > \bar{\mu}$.
Let $m \leq M$ and $n > N$ be arbitrary positive integers.

(\ref{item J^n,+m isolate1}) Let $\gamma_{} \colon [-T,T]\rightarrow V^{\mu_{n}}_{\lambda_{n}}$ be an approximated Seiberg-Witten trajectory contained
in $J^{n,+}_{m}$. Corollary~\ref{boundedness for approximated trajectories} and the choice of $N,T$ ensure that $\gamma_{}(0)\in Str(R_{0})$. For the sake of contradiction, let us suppose that
 $g^{}_{j,+}(\gamma_{}(0)) > \theta+m-\theta_{m}$ for some $j\in \{1 , \ldots , {b_1} \}$. By Lemma \ref{approximated flow goes in}, the value of $g_{j,+}(\gamma(t)) $ decreases along the trajectory $\gamma $ on $[-T,0]$  with
$$\frac{dg_{j,+}(\gamma_{}(t))}{dt} = \langle-(l+p^{\mu_{n}}_{\lambda_{n}}\circ c )(\gamma(t)),\gradtil  g_{j,+}(\gamma(t))\rangle_{\tilde{g}} < -\epsilon_{m}.$$
Hence, we obtain $g^{}_{j,+}(\gamma_{}(-T)) > g^{}_{j,+}(\gamma_{}(0)) + T \epsilon_m >\theta +m $ from the fundamental theorem of calculus. This is a contradiction with our assumption that $\gamma(-T) \in J^+_m \subset g^{-1}_{j,+}(-\infty,\theta+m]$.

(\ref{item J^n,+m isolate2}) From Lemma~\ref{approximated flow goes in} and the choice of $N$, we have
$\langle-(l+p^{\mu_{n}}_{\lambda_{n}}\circ c )(x),\gradtil  g_{j,+}(x)\rangle_{\tilde{g}} < 0$ for any $x\in J^{n,+}_{m-1} \cap g^{-1}_{j,+}(\theta+m-1)$. Consequently, the flow $\varphi_{m}^{n}$ goes inside $J^{n,+}_{m-1}$ along $\partial{J^{n,+}_{m-1}} \setminus \partial Str(\tilde{R})$ and $\inv(\varphi^{n}_{m-1},J^{n,+}_{m-1})$ is an attractor in $\inv(\varphi^{n}_{m},J^{n,+}_{m})$ with respect to the flow $\varphi^{n}_{m}$.
\end{proof}

Consequently, we can acquire the $S^{1}$-equivariant Conley index $I_{S^{1}}(\varphi_{m}^{n},\inv(J^{n,+}_{m}))$ from a compact finite-dimensional subset $J^{n,+}_m $ when $n$ is large enough relative to $m$ as in Lemma~\ref{lem J^n,+m isolate}. Using the orthogonal
complement $\bar{V}^{0}_{\lambda}$ of $i\Omega^{1}_{h}(Y)$ in $V^{0}_{\lambda}$, we define
\begin{equation*} \label{def I^{+}}
I^{n,+}_{m} := \Sigma^{-\bar{V}^{0}_{\lambda_{n}}}I_{S^{1}}(\varphi_{m}^{n},\inv(J^{n,+}_{m}))
\end{equation*}
as an object of $\mathfrak{C}$.  Note that here a choice of index pair for $\inv(J^{n,+}_{m})$ is made to get the Conley index (see the remark following  Definition \ref{def conleyindex}).
Eventually, we will show that our invariants are independent of this choice up to canonical isomorphisms.

Let $i^{n,+}_{m}\colon I_{S^{1}}(\varphi^{n}_{m},\inv(J^{n,+}_{m})) \rightarrow I_{S^{1}}(\varphi^{n}_{m},\inv(J^{n,+}_{m+1}))$ be the attractor map and denote by $\tilde{i}^{n,+}_{m}$ a morphism
$\Sigma^{-\bar{V}^{0}_{\lambda_{n}}}i^{n,+}_{m} \in \morp_{\mathfrak{C}}(I^{n,+}_{m},I^{n,+}_{m+1}) $.
We will show that the object $I^{n,+}_{m} $ is stable in the following sense.

\begin{pro}\label{stability of the conley index}
For any positive integer $M>0$, there exists a positive integer $N$ such that, for any positive integers $m\leq M$ and $n\geq N$, there is a canonical isomorphism $\tilde{\rho}^{n,+}_{m}\in \morp_{\mathfrak{C}}(I^{n,+}_{m},I^{n+1,+}_{m})$. Moreover, we have the following commutative diagram
\begin{equation}\label{commutative diagram for attractor}
\xymatrix{
I^{n,+}_{m-1} \ar[d]_{\tilde{\rho}^{n,+}_{m-1}} \ar[r]^{\tilde{i}^{n,+}_{m-1}} & I^{n,+}_{m}\ar[d]^{\tilde{\rho}^{n,+}_{m}}\\
I^{n+1,+}_{m-1} \ar[r]^{\tilde{i}^{n+1,+}_{m-1}} & I^{n+1,+}_{m}.}
\end{equation}
\end{pro}

\begin{proof} Following the remark after Corollary~\ref{boundedness for approximated trajectories}, we can extend the result of Lemma~\ref{lem J^n,+m isolate} to interpolated projections. With the integer $N$ depending on $M$ from Lemma~\ref{lem J^n,+m isolate}, we can deduce that $J^{n+1,+}_{m}$ is an isolating neighborhood for the flow generated by
$-\iota_{m}\cdot (l+(s p^{\mu_{n+1}}_{\lambda_{n+1}}+(1-s)p^{\mu_{n}}_{\lambda_{n}})\circ c )$ for any $n > N$ and $s \in [0,1]$.

The rest of proof follows from the arguments given in \cite[p.910]{Manolescu1} and \cite[Proposition~8]{Khandhawit2}.
By continuation property of the Conley index, we have a natural homotopy equivalence
$$\rho^{n,+}_{m} \colon \Sigma^{V^{\lambda_{n}}_{\lambda_{n+1}}}I_{S^{1}}(\varphi^{n}_{m},\inv(J^{n,+}_{m}))\rightarrow I_{S^{1}}(\varphi^{n+1}_{m},\inv(J^{n+1,+}_{m})).$$
The isomorphism $\tilde{\rho}^{n,+}_{m}$ is then given by the composition
\begin{equation*}
\begin{split}
\Sigma^{-\bar{V}^{0}_{\lambda_{n}}}I_{S^{1}}&(\varphi^{n}_{m},\inv(J^{n,+}_{m}))\rightarrow \Sigma^{-\bar{V}^{0}_{\lambda_{n}}}\Sigma^{-V^{\lambda_{n}}_{\lambda_{n+1}}}\Sigma^{V^{\lambda_{n}}_{\lambda_{n+1}}}I_{S^{1}}(\varphi^{n}_{m},\inv(J^{n,+}_{m}))\\
&\rightarrow\Sigma^{-\bar{V}^{0}_{\lambda_{n}}}\Sigma^{-V^{\lambda_{n}}_{\lambda_{n+1}}}I_{S^{1}}(\varphi^{n+1}_{m},\inv(J^{n+1,+}_{m}))=\Sigma^{-\bar{V}^{0}_{\lambda_{n+1}}}I_{S^{1}}(\varphi^{n+1}_{m},\inv(J^{n+1,+}_{m})),
\end{split}
\end{equation*}
where the first morphism is given by $\Sigma^{-\bar{V}^{0}_{\lambda_{n}}}\eta^{-1}$ (see Proposition \ref{suspension desuspension functors cancels}) and the second morphism equals $\Sigma^{-\bar{V}^{0}_{\lambda_{n}}}\rho^{n,+}_{m}$.
The diagram (\ref{commutative diagram for attractor}) commutes because of the continuation property of
attractor-repeller pairs \cite[Theorem~6.10]{Salamon}.

\end{proof}

For each positive integer $M$, we pick a positive integer $n_M$ larger than the constant $N$ from Proposition~\ref{stability of the conley index} and we require that $\{n_M\}$ is an increasing sequence. We are now ready to define the spectrum invariant.

\begin{defi}
The $S^{1}$-equivariant ind-spectrum $\underline{\text{swf}}^{A}(Y,\mathfrak{s}_{Y}, A_{0},g;S^{1})$ is defined to be an object of $\mathfrak{S}$ given by
\begin{equation}\label{attractor system}
        I^{n_{1},+}_{1} \rightarrow I^{n_{2},+}_{2} \rightarrow I^{n_{3},+}_{3} \rightarrow \cdots,
\end{equation}
where the morphism from $I^{n_{m},+}_{m}$ to $I^{n_{m+1},+}_{m+1}$ is a composition $\tilde{i}^{n_{m+1},+}_{m}\circ \tilde{\rho}^{n_{m+1}-1,+}_{m} \circ \cdots \circ \tilde{\rho}^{n_{m},+}_{m}$ of the morphisms in Proposition \ref{stability of the conley index}.

\label{def swf^A}
\end{defi}

We will prove in the next section that this gives a well-defined object of the category $\mathfrak{S}$ independent of the choices made in the construction up to canonical isomorphism.

To define another invariant $\underline{\text{swf}}^{R}(Y,\mathfrak{s}_{Y},A_{0},g;S^{1})$,  we follow almost the same steps for the construction of $\underline{\text{swf}}^{A} $ except that there are two main differences.
First, the set $\inv(\varphi^{n}_{m},J^{n,-}_{m})$ is a repeller in $\inv(\varphi^{n}_{m},J^{n,-}_{m+1})$, so the arrows in the system will be reversed. Second, we use $V^{0}_{\lambda_{n}}$ for desuspension instead of $\bar{V}^{0}_{\lambda_{n}}$. We define
\begin{equation*}
                    I^{n, -}_{m} := \Sigma^{-{V}^{0}_{\lambda_{n}}} I_{S^{1}}(\varphi_{m}^{n}, \inv(J^{n,-}_{m}))
                    \in
                    \ob \mathfrak{C},
\end{equation*}
where $n$ is large enough relative to $m$, and we have a morphism
\[
                  I^{n, -}_{m} \leftarrow I^{n,-}_{m+1}
\]
induced by the repeller map.
The following collection of results can be proved in the same way as the corresponding results for $J^{n,+}_m $.


\begin{pro} \label{properties in the repeller case}
For a positive integer $M$, there exist large numbers $N,T$ such that, for any positive integers $m\leq M$ and $n\geq N$, we have the following statements.

\begin{enumerate}[(a)]
\item  For any approximated Seiberg-Witten trajectory $\gamma:[-T,T] \rightarrow V^{\mu_{n}}_{\lambda_{n}}$ which is contained in $J^{n,-}_{m}$, we have
$$
     \gamma(0)\in Str(R_{0})\cap \bigcap_{1\leq j \leq b_1}g^{-1}_{j,-}(-\infty,\theta+m-\theta_{m}].
$$
In particular, $J^{n,-}_{m}$ is an isolating neighborhood for the flow $\varphi^{m}_{n}$.

\item The set $\inv(\varphi^{m}_{n},J^{n,-}_{m-1})$ is a repeller in $\inv(\varphi^{m}_{n},J^{n,-}_{m})$ with respect to the flow $\varphi^{m}_{n}$. Consequently, we have the repeller map
$$
         i^{n,-}_{m-1} \colon I_{S^{1}}(\varphi^{n}_{m},\inv(J^{n,-}_{m}))
              \rightarrow
         I_{S^{1}}(\varphi^{n}_{m},\inv(J^{n,-}_{m-1})).
$$

\item There is a canonical isomorphism $\tilde{\rho}^{n,-}_{m}\in \morp_{\mathfrak{C}}(I^{n,-}_{m},I^{n+1,-}_{m})$ such that the following diagram commutes

\begin{equation}\label{commutative diagram for repeller}
\xymatrix{
I^{n,-}_{m-1} \ar[d]_{\tilde{\rho}^{n,-}_{m-1}} & I^{n,-}_{m}\ar[d]^{\tilde{\rho}^{n,-}_{m}}\ar[l]_{\tilde{i}^{n,-}_{m-1}}\\
I^{n+1,-}_{m-1} & I^{n+1,-}_{m}\ar[l]_{\tilde{i}^{n+1,-}_{m-1}} ,}
\end{equation}
where $\tilde{i}^{n,-}_{m-1}$ is given by $\Sigma^{-V^{0}_{\lambda_{n}}}i^{n,-}_{m-1}$.
\end{enumerate}
\end{pro}

For each positive integer $M$, we also choose a positive integer $n_M$ larger than the constant $N$ from Proposition~\ref{properties in the repeller case} so that $\{n_M\}$ is an increasing sequence.

\begin{defi}
The $S^{1}$-equivariant pro-spectrum $\underline{\text{swf}}^{R}(Y,\mathfrak{s}_{Y},A_{0},g;S^{1})$ is defined to be an object of $\mathfrak{S}^*$ given by
\begin{equation}\label{repeller system}
I^{n_{1},-}_{1}\leftarrow I^{n_{2},-}_{2} \leftarrow I^{n_{3},-}_{3}\leftarrow \cdots,
\end{equation}
where the connecting morphisms are defined in the same manner as in Definition~\ref{def swf^A}.
\end{defi}

We will also prove well-definedness of $\underline{\text{swf}}^{R}(Y,\mathfrak{s}_{Y},A_{0},g;S^{1})$ in the next section.

\subsection{The torsion case}  \label{subsection torsion case}

When the spin$^c$ structure $\mathfrak{s}$ is torsion, we will be able to further normalize the spectrum invariants $\underline{\text{swf}}^{A}$ and $\underline{\text{swf}}^{R}$  following the idea of \cite{Manolescu1}.
The resulting objects will not depend on $A_{0}$ and $g$.

We will need to define a rational number $n(Y,\mathfrak{s}_{Y},A_{0},g)$. Choose a 4-manifold $X$ with boundary $Y$ with $H^{3}(X,Y;\mathbb{Z})\cong H_{1}(X;\mathbb{Z})=0$. Such $X$ always exists as we can construct $X$ by attaching 2-handles on $D^{4}$ according the surgery diagram of $Y$. By the homology long exact sequence for the pair $(X,Y)$, we see that $H^{2}(X,\mathbb{Z})\rightarrow H^{2}(Y,\mathbb{Z})$ is surjective. Therefore, we can extend $\mathfrak{s}$ to a $\text{spin}^{c}$ structure $\mathfrak{s}_{X}$ over $X$ and extend $A_{0}$ to a connection $\hat{A}_{0}$ over $X$. Recall that we have a nondegenerate pairing
\begin{equation*}\label{pairing for manifold with boundary}
     \cup : \im  (H^{2}(X,Y;\mathbb{Q})\rightarrow H^{2}(X;\mathbb{Q}))\otimes\im  (H^{2}(X,Y;\mathbb{Q})
     \rightarrow H^{2}(X;\mathbb{Q}))\rightarrow \mathbb{Q}.
\end{equation*}
Denote by $b^{+}(X)$ (resp. $b^{-}(X)$) the dimension of a maximal positive (resp. negative) subspace with respect to this pairing and denote by $\sigma(X)$ the signature of this pairing. Notice that we can define $c_{1}(\mathfrak{s}_{X})^{2}=c_{1}(\mathfrak{s}_{X})\cup c_{1}(\mathfrak{s}_{X})\in \mathbb{Q}$ because $c_{1}(\mathfrak{s}_{X})|_{Y}=c_{1}(\mathfrak{s})$ is torsion. We define
\begin{equation}  \label{def n(Y)}
         n(Y,\mathfrak{s}, A_{0}, g):=
        \text{Ind}_{\mathbb{C}}(\hat{\slashed{D}}^{+}_{\hat{A}_{0}})-\frac{c_{1}(\mathfrak{s}_{X})^{2}-\sigma(X)}{8},
\end{equation}
where $\hat{\slashed{D}}^{+}_{\hat{A}_{0}}$ is the positive Dirac operator on $X$ coupled with $\hat{A}_{0}$ and $\operatorname{Ind}_{\mathbb{C}} (\hat{\slashed{D}}^{+}_{\hat{A}_{0}})$ is its index defined by using  spectral boundary condition as in \cite{APS}. It was proved in \cite{Manolescu1} that $n(Y,\mathfrak{s},A_{0},g)$ does not depend on the choices of $X,\mathfrak{s}_{X}$ and $\hat{A}_{0}$ (\cite{Manolescu1} only considered a rational homology sphere $Y$ but the proof works for a general 3-manifold $Y$ without any changes). In fact, we have
 \begin{equation}\label{virtual dimension}
       n(Y,\mathfrak{s},A_{0},g)=
       \frac{1}{2} \left(  \eta(\slashed{D})  -  \text{dim}_{\mathbb{C}}(\ker \slashed{D})  +  \frac{\eta_{\text{sign}}}{4} \right),
 \end{equation}
 where $\eta(\slashed{D})$ and $\eta_{\text{sign}}$ denote the eta-invariant of the Dirac operator and the odd signature operator respectively (see \cite{Manolescu1} and \cite{APS}).

The normalized invariant $\underline{\text{SWF}}^{A}$ and $\underline{\text{SWF}}^{R}$ will be obtained by formally desuspending $ \underline{\text{swf}}^{A}$ and $\underline{\text{swf}}^{R}$ with the rational number $n(Y,\mathfrak{s},A_{0},g)$ as follows.

\begin{defi}\label{normalized spectrum invariants}
We define the $S^{1}$-equivariant ind-spectrum and pro-spectrum  by
\begin{align*}
    \underline{\text{SWF}}^{A}(Y,\mathfrak{s};S^{1})    &:=
          \left(\underline{\text{swf}}^{A}(Y,\mathfrak{s},A_{0},g;S^{1}),0,n(Y,\mathfrak{s},A_{0},g) \right), \\
    \underline{\text{SWF}}^{R}(Y,\mathfrak{s};S^{1})      &:=
          \left(\underline{\text{swf}}^{R}(Y,\mathfrak{s},A_{0},g;S^{1}),0,n(Y,\mathfrak{s},A_{0},g)\right).
\end{align*}
as objects of $\mathfrak{S}$ and $\mathfrak{S}^*$ respectively.
\end{defi}

The proof of invariance of $\underline{\text{SWF}}^{A}$ and $\underline{\text{SWF}}^{R}$ will also be in the next section.

\subsection{The $Pin(2)$-spectrum invariants for spin structures}  \label{subsection Pin(2)-spectrum}

In this subsection, we will define $Pin(2)$-analogue of the spectrum invariants for a 3-manifold $Y$ equipped with a spin structure $\mathfrak{s}$. Since all the constructions are similar to the $S^1$-case, some of the discussions will be brief.

We define the group $Pin(2)$ as the subgroup $S^{1}\cup jS^{1}\subset \mathbb{H}$ of the algebra of quaternions containing $S^1$ as the set of unit complex numbers. We are interested in the following real representations of $Pin(2)$:
\begin{enumerate}
\item $\mathbb{R}$ the trivial one-dimensional representation;
\item $\tilde{\mathbb{R}}$ the nontrivial one-dimensional representation where $S^{1}$ acts trivially and $j$ acts as  multiplication by $-1$;
\item $\mathbb{H}$ the four-dimensional representation where $Pin(2)$ acts by left quaternionic multiplication.
\end{enumerate}

We introduce the category $\mathfrak{C}_{Pin(2)}$, $\mathfrak{S}_{Pin(2)}$ and $\mathfrak{S}_{Pin(2)}^*$ which are the $Pin(2)$-version of the categories $\mathfrak{C},\mathfrak{S}$ and $\mathfrak{S}^{*}$.  The objects of $\mathfrak{C}_{Pin(2)}$ are triples $(A, m, n)$ consisting of an even integer $m$, a rational number $n$ and  a pointed $Pin(2)$-space $A$ which is $Pin(2)$-homotopy equivalent to a finite $Pin(2)$-CW complex. The set $\morp_{\mathfrak{C}_{Pin(2)}}((A, m, n ), (A',m',n'))$ is given by
$$
\mathop{\text{colim}}_{u,v,w \rightarrow \infty}[(\mathbb{R}^{u}\oplus\tilde{\mathbb{R}}^{v}\oplus \mathbb{H}^{w})^{+}\wedge A,(\mathbb{R}^{u}\oplus\tilde{\mathbb{R}}^{v+m-m'}\oplus \mathbb{H}^{w+n-n'})^{+}\wedge A' )]_{Pin(2)}
$$
when $n - n' \in \mathbb{Z}$ and is empty otherwise.
The objects of $\mathfrak{S}_{Pin(2)}$ (resp. $\mathfrak{S}^{*}_{Pin(2)}$) are the sequential direct systems (resp. sequential inverse systems) in $\mathfrak{C}_{Pin(2)}$. We call an object of $\mathfrak{S}_{Pin(2)}$ a $Pin(2)$-equivariant ind-spectrum and call an object of $\mathfrak{S}^{*}_{Pin(2)}$ a $Pin(2)$-equivariant pro-spectrum. The sets of morphisms are defined in the same way as (\ref{the morphism in S}) and (\ref{moprhism in S*}). For an object $W$ of $\mathfrak{C}_{Pin(2)}$, $\mathfrak{S}_{Pin(2)}$ and $\mathfrak{S}_{Pin(2)}^*$, the natation $(W,m,n )$ will be used as in the $S^{1}$-case. A $Pin(2)$-representation $E$ is called \emph{admissible} if it is isomorphic to $\tilde{\mathbb{R}}^a \oplus \mathbb{H}^b$. For an admissible representation $E$, we can define the suspension functor $\Sigma^{E}$ and the desuspension functor $\Sigma^{-E}$ in the same manner, e.g. $\Sigma^{-E}(A, m, n )$ is given by $( \Sigma^{E^{S^{1}}}A, m+2a, n+b )$.   All the results from Section \ref{Section Cat Top} can be adapted to this setting.

We now turn to the Seiberg-Witten theory of a spin 3-manifold. Recall that the spin structure $\mathfrak{s}$ induces a torsion spin$^{c}$ structure on $Y$. With a slight abuse of notations, we also denote this spin$^{c}$ structure by $\mathfrak{s}$. We will have the same setup from the spin$^c$ structure \(\mathfrak{s}\) with the following new features coming from  a spin structure.

\begin{enumerate}
\item The structure group of $S_{Y}$ can be reduced to $SU(2)\cong S(\mathbb{H})$. Therefore, $S_{Y}$ is a quaternionic bundle. Here we follow the convention of \cite{Manolescu3} and let the structure group act by the right multiplication.

\item The bundle $\text{det}(S_{Y})$ has a canonical trivialization. The Levi-Civita connection on $TY$ then induces a canonical spin connection $A_{0}$ on $S_{Y}$ with $F_{A_{0}^{t}}=0$. We will always choose $A_{0}$ for our base connection.

\item We have an additional action $\jmath:\mathcal{C}_{Y}\rightarrow \mathcal{C}_{Y}$ sending $(a,\phi)$ to $(-a,j\phi)$. This action, together with the constant gauge group $S^{1}$, gives a $Pin(2)$-action on $\mathcal{C}_{Y}$. All the objects in the setup are $Pin(2)$-invariant, e.g. the functional $CSD_{\nu_{0}}$, the Coulomb slice $Coul(Y)$ and the $L^{2}_{k}$-inner product etc.
\end{enumerate}

In order to respect the additional $\jmath$-symmetry, we have two new requirements in our construction.

\begin{enumerate}

\item The perturbation $f$ should be invariant under $\jmath$. In other worlds, we should have $f(a,\phi)=f(-a,j\phi)$.

\item The sets $J^{n,\pm}_{m}$ should be invariant under $\jmath$.

\end{enumerate}

A slight adaption of \cite[Theorem 2.6]{Francesco} shows that for any real number $\delta$, we can find a $\jmath$-invariant extended cylinder function $\bar{f}$ such that $(\delta,\bar{f})$ is a good perturbation. Since we required the staircase function $\bar{g}$ from Section \ref{Subsection general spin^c} is even, it is not hard to see that $J^{n,\pm}_{m}$ is $\jmath$-invariant once the perturbation $f$ is $\jmath$-invariant.

We can now follow the construction from Section \ref{Subsection general spin^c} .   In particular, the sets $J^{n,\pm}_{m}$ are  isolating neighborhoods for the $Pin(2)$-invariant flow $\varphi_{m}^{n}$ when $n$ is sufficiently large relative to $m$ and we define
\begin{align*}
I^{n,+}_{m}(Pin(2)) &:= \Sigma^{-\bar{V}^{0}_{\lambda_{n}}}I_{Pin(2)}(\varphi^{m}_{n},\inv(J^{n,+}_{m})), \\
I^{n,-}_{m}(Pin(2)) &:= \Sigma^{-V^{0}_{\lambda_{n}}}I_{Pin(2)}(\varphi^{m}_{n},\inv(J^{n,-}_{m}))
\end{align*}
as objects of $\mathfrak{C}_{Pin(2)}$. As before, we obtain an object $\underline{\operatorname{swf}}^A(Y, \mathfrak{s}, A_0, g; Pin(2))$ of $\mathfrak{S}_{Pin(2)}$  given by
\begin{equation*}\label{Pin(2) direct system}
            I^{n_{1},+}_{1}(Pin(2))  \rightarrow   I^{n_{2},+}_{2}(Pin(2))  \rightarrow  \cdots
\end{equation*}
and an object $\underline{ \operatorname{swf}}^R(Y, \mathfrak{s}, A_0, g; Pin(2))$ of $\mathfrak{S}^*_{Pin(2)}$ given by
\begin{equation*}\label{Pin(2) inverse system}
         I^{n_{1},-}_{1}(Pin(2))   \leftarrow I^{n_{2},-}_{2}(Pin(2))  \leftarrow \cdots
\end{equation*}
for an increasing sequence of large positive integers $\{n_i\}$. We define spectrum invariants as in the torsion spin$^c$ case.

\begin{defi}
With the above setup, the $Pin(2)$-equivariant ind-spectrum and pro-spectrum are defined by
\[
  \begin{split}
     \underline{\operatorname{SWF}}^{A}(Y,\mathfrak{s}; Pin(2))
     &:= \left(\underline{ \operatorname{swf}}^A(Y, \mathfrak{s}, A_0, g; Pin(2)) , 0, \frac{n(Y,\mathfrak{s}, A_{0}, g)}{2} \right),  \\
     \underline{\text{SWF}}^{R}(Y,\mathfrak{s};Pin(2))
    &:=\left( \underline{ \operatorname{swf}}^R(Y, \mathfrak{s}, A_0, g; Pin(2)), 0, \frac{n(Y,\mathfrak{s},A_{0},g)}{2} \right).
    \end{split}
\]
 as objects of $\mathfrak{S}_{Pin(2)}$ and $\mathfrak{S}^{*}_{Pin(2)}$ respectively. Here $n(Y, \mathfrak{s}, A_0, g)$ is the rational number defined in (\ref{def n(Y)}).  As before, these objects are independent of the choices made in the construction  up to canonical isomorphism.
\end{defi}

\section{The invariance for the spectrum}\label{invariance}
In this section we will prove the invariance  of our ind-spectrum (pro-spectrum). In other words, we will show that the spectra given by different choices of parameters are canonically isomorphic to each other (as objects of the category in which they are defined). We focus on the $S^{1}$-equivariant case and the $Pin(2)$-case can be proved in the same way.

First, let us list the parameters in the order that the choices of a parameter can only depend on the parameters listed before it (for example, $\tilde{R}$ is any number greater $R_{0}$, where $R_{0}$ is the constant of Theorem \ref{Boundedness of trajectories}  depending on $g,A_{0}$ and $f$):

\begin{enumerate}[(I)]
\item  The Riemannian metric $g$ and the base connection $A_{0}$;
\item  The good perturbation $f \colon Coul(Y)\rightarrow \mathbb{R}$;
\item  The sequences of real numbers $\{\lambda_{n}\}, \{\mu_{n}\}$;
\item  The number $\tilde{R}$ (in the definition of $Str(\tilde{R})$);
\item  The harmonic forms $\{h_{j}\}$, the cutting function $\bar{g}$ and the cutting value $\theta$;
\item  The positive integers $n_{m}$ in (\ref{attractor system}) and (\ref{repeller system});
\item  {The index pairs} for the isolated invariant sets.
\end{enumerate}

The invariance for (VII) is a direct consequence of the invariance of the Conley
index (see Subsection \ref{subsection conley index} and \cite{Salamon}).  The commutative diagrams (\ref{commutative diagram for attractor}) and (\ref{commutative
diagram for repeller})  imply the invariance for (VI).

In subsection 6.1, we will make a digression into the discussion of the finite dimensional approximation for a family of flows. In subsection 6.2, we will prove the invariance for (III), (IV), (V). The invariance for (II) (which is the most interesting one) and (I) will be proved in subsection 6.3 and subsection 6.4 respectively. In subsection 6.5, we will discuss the restriction of our invariant to the $S^{1}$-fixed point sets.

\subsection{The finite dimensional approximation for a family of flows}\label{subsection family flows}
In this subsection, we extend finite dimensional approximation results in Section~\ref{Section approx SW traj} for a continuous family of flows. This setup will be useful for proving the invariance and calculating examples.

Let $S$ be a compact manifold (possibly with boundary) and consider a smooth
family of Riemannian metrics $\{g_{s}\}_{s \in S}$ and a smooth family of base connections $\{A_{0,s}\}_{s\in S}$. As before, we require that $\frac{i}{2\pi} F_{A_{0,s}^{t}}$ equals the harmonic form representing $c_{1}(\mathfrak{s})$. We denote by $Coul(Y,s)$ the ($L^{2}_{k}$-completed) Coulomb slice for  $(g_s, A_{0, s} )$. For each $s$, we have an elliptic operator $l_{s} \colon Coul(Y,s)\rightarrow Coul(Y,s)$ given by $( { *_s}d,\slashed{D}_{A_{0,s}})$, { where $*_s$ is the Hodge {operator} of $g_s$}. Although $\{Coul(Y,s)|s\in S\}$ is
a Hilbert bundle over $S$, by the Kuiper's theorem, this bundle is trivial and we can
identify it with $S\times Coul(Y)$ by fixing a trivialization.
We have the following generalization of Definition \ref{quadratic-like map}:
\begin{defi}\label{a family of quadratic-like maps}
Let $E$ be a vector bundle over $Y$. A family of smooth and bounded maps $\{ Q_{s} \colon Coul(Y,s)\rightarrow L^{2}_{k}(\Gamma(E))\}_{s \in S}$ is called a \emph{continuous family of quadratic-like maps} if $Q_s$ is quadratic-like for each $s \in S$ and, for each nonnegative integer $m <k$, we have a uniform convergence$ $ $(\frac{d}{dt})^{m}Q_{s_n}(\gamma_{n}(t))\rightarrow
(\frac{d}{dt})^{m}Q_{s_\infty} (\gamma_{\infty}(t))$ in $L^{2}_{k-2-m}$ whenever there is a uniform convergent of compact paths $(\frac{d}{dt})^{j}\gamma_{n}(t)\rightarrow (\frac{d}{dt})^{j}\gamma_{\infty}(t)$
uniformly in $L^{2}_{k-1-j}$ for each $j=0,1,...,m$ with $\gamma_n \colon I \rightarrow Coul(Y,s_n) $ and $s_n \rightarrow s_\infty $.

\end{defi}

We now let $\{Q_{s} \colon Coul(Y,s)\rightarrow L^{2}_{k}(\ker d^{*}\oplus \Gamma(S_{Y}))\}_{s \in S}$ be a continuous family of quadratic-like maps. As before, for real numbers $\lambda < 0 \leq \mu$,  we define $V^{\mu}_{\lambda}(s)\subset Coul(Y,s)$ to be the space spanned by the eigenvectors of $l_{s}$ with eigenvalue in $(\lambda,\mu]$. We also consider $\bar{V}^{0}_{\lambda}(s)$, which is the orthogonal complement of $i\Omega^{1}_{h}(Y)$ in $V^{0}_{\lambda}(s)$. Note that these spaces usually do not change continuously with $s$ because the dimension can jump at eigenvalues of $l_{s}$.

Throughout this subsection, we say that, for an interval $I$, a path $\gamma \colon I\rightarrow Coul(Y,s)$ is an actual trajectory if it satisfies $\frac{d}{dt}\gamma(t)=-(l+Q_{s})\gamma(t)$
and a path $\gamma \colon I\rightarrow V^{\mu}_{\lambda}(s)$ is an approximated trajectory if it satisfies $\frac{d}{dt}\gamma(t)=-(l+p^{\mu}_{\lambda}\circ Q_{s})\gamma(t)$ for some $\mu,\lambda$.
We denote by $\varphi(\lambda,\mu,s)$ the flow generated by $-\iota\cdot(l+p^{\mu}_{\lambda}\circ Q_{s})$, where $\iota$ is a bump function which equals $1$ on any bounded subset involved in our discussion.

\begin{thm}\label{finite dimensional approximation for a family of flows}
Let $B$ be a closed and bounded subset of $Coul(Y)$ and suppose that there exists a closed subset $A\subset \text{int}(B)$ such that, for any $s\in S$ and any actual trajectory $\gamma \colon \mathbb{R}\rightarrow Coul(Y,s)$ contained in $B$, we have $\gamma$ contained in $ \operatorname{int}(A)$.
Then there exist constants $T,  -\bar{\lambda}, \bar{\mu} \gg 0$ such that the following statements hold:

\begin{enumerate}[(i)]
\item \label{item prop familyflow1} For any  $\lambda<\bar{\lambda}$, $\mu>\bar{\mu}$  and $s\in S$, if an approximated trajectory $\gamma \colon [-T,T]\rightarrow V^{\mu}_{\lambda}(s)$ is contained in $B$, then we have $\gamma(0)\in A$. In particular, $B \cap V^{\mu}_{\lambda}$ is an isolating neighborhood for the flow $\varphi(\lambda,\mu,s)$;

\item \label{item prop familyflow2} The spectra $\Sigma^{-V^{0}_{\lambda}(s)}I_{S^{1}}(\varphi(\lambda, \mu, s ),\operatorname{inv}(B \cap V^{\mu}_{\lambda}(s)))$ and $\Sigma^{-\bar{V}^{0}_{\lambda}(s)}I_{S^{1}}(\varphi(\lambda, \mu,s ),\operatorname{inv}(B \cap V^{\mu}_{\lambda}(s)))$ do not depend on the choice of  $\lambda<\bar{\lambda}$ and $\mu>\bar{\mu}$ up to canonical isomorphisms in $\mathfrak{C}$. We denote these objects by $I(B ,s)$ and $\bar{I}(B,s) $ respectively.

\item \label{item prop familyflow3} For any path $\alpha \colon [0,1]\rightarrow S$, we have well defined isomorphisms
$$\rho(B,\alpha) \colon I(B,\alpha(0))\rightarrow \Sigma^{\spf(-\slashed{D},\alpha)\mathbb{C}}I(B,\alpha(1)),$$
$$\bar{\rho}(B,\alpha) \colon \bar{I}(B,\alpha(0))\rightarrow \Sigma^{\spf(-\slashed{D},\alpha)\mathbb{C}}\bar{I}(B,\alpha(1)),$$
where $\spf(-\slashed{D},\alpha)$ denotes the spectral flow of $-\slashed{D}$ along the path $\alpha$.
Moreover, the isomorphisms $\rho$ and $\bar{\rho}$ only depend on the homotopy class of $\alpha$ relative to its end points.
\end{enumerate}

\end{thm}

\begin{proof}
For the first part, the proof is similar to that of Corollary ~\ref{boundedness for approximated trajectories}: we suppose there exists no such $\bar{\lambda},\bar{\mu},T$. Then we can find a sequence of approximated trajectories $\gamma_{n}\colon [-T_{n},T_{n}]\rightarrow Coul(Y, s_{n})$ with $T_{n}, - \lambda_n, \mu_{n}\rightarrow +\infty$ such that $\gamma_{n}$ is contained in $B$ but $\gamma_{n}(0)\notin A$. Since $S$ is compact, we can assume $s_{n}\rightarrow s_{\infty}$ after passing to a subsequence. The properties in Definition~\ref{a family of quadratic-like maps} allow us to repeat the argument in the proof of Proposition~\ref{convegence of approximated trajectories} and find an actual trajectory $\gamma_{\infty} \colon \mathbb{R} \rightarrow Coul(Y,s_{\infty})$ as the limit of $\gamma_{n}$. Consequently, we have $\gamma$ contained in $B$ and $\gamma_{\infty}(0)\notin \operatorname{int}(A)$. This is a contradiction with our hypothesis.  Thus, the proof of (\ref{item prop familyflow1}) is finished.

The proof of (\ref{item prop familyflow2}) is a straight forward adaption of arguments from Proposition~\ref{stability of the conley index} and we omit it.
For (\ref{item prop familyflow3}), we will focus on the case $\rho(B,\alpha)$ as the other case can be proved similarly.
For brevity, we will denote by $E_\lambda^\mu (s)$ the Conley index $ I_{S^{1}}(\varphi(\lambda_{}, \mu, s),\operatorname{inv}(B \cap V^{\mu}_{\lambda} (s))) $. The isomorphism $\rho(B, \alpha)$ is constructed as follows: we consider the interval $[0,1]$ as the union of subintervals  $[t_{j},t_{j+1}]$ with $j=1,...,m$ such that, for each $j$, we can find $\mu_{j}>\bar{\mu}$ and $\lambda_{j}<\bar{\lambda}$  which are not eigenvalues of $l_{\alpha(t)}$ for any $t \in [t_{j},t_{j+1}]$.  Then $V^{\mu_{j}}_{\lambda_{j}}(\alpha(t))$ from $t=t_{j}$ to $t=t_{j+1}$ is a continuous family of linear subspaces and $\varphi(\lambda_{j}, \mu_j, \alpha(t))$ is a continuous family of flows on them. By the homotopy invariance of the Conley index \cite[Section~6]{Salamon}, we get an isomorphism
\begin{equation}\label{isomorphism 1}
 \rho_j \colon E^{\mu_j}_{\lambda_j} (\alpha(t_{j}))
        \xrightarrow{\simeq}
       E^{\mu_j}_{\lambda_j}(\alpha(t_{j+1})).
   \end{equation}
Notice that
$$
[V^{0}_{\lambda_j}(\alpha(t_{j}))]+ [\spf(-\slashed{D},\alpha([t_{j},t_{j+1}]))\mathbb{C}]=[V^{0}_{\lambda_{j}}(\alpha(t_{j+1}))]
$$
as elements of the representation ring of $S^{1}$.
We can desuspend both sides of (\ref{isomorphism 1}) and get an isomorphism$$I(B,\alpha(t_{j}))\rightarrow \Sigma^{\spf(-\slashed{D},\alpha([t_{j},t_{j+1}]))\mathbb{C}}I(B,\alpha(t_{j+1})).$$
The isomorphism $\rho(B,\alpha)$ is defined as the composition of the above isomorphisms for $j = 1, \ldots , m $.

We will see that $\rho(B, \alpha)$ is independent of the choices of $ t_j $, $ \lambda_j $ and $\mu_j $.  First, fix a choice of $\{ t_j \}$ and choose different choices of $\{ \lambda_j' \}_{}$ and $\{ \mu_j' \}_{}$.  Without loss of generality, we may assume that $\lambda_j < \lambda_j'$, $\mu_j > \mu_j'$.  As before, we have an isomorphism
$$ \rho'_j \colon E^{\mu'_j}_{\lambda'_j} (\alpha(t_{j}))
        \xrightarrow{\simeq}
       E^{\mu'_j}_{\lambda'_j}(\alpha(t_{j+1})).$$
As in Proposition~\ref{stability of the conley index}, we have isomorphisms for stability of {Conley} indices
\begin{align*}
\sigma_j &\colon E^{\mu_j}_{\lambda_j} (\alpha(t_{j})) \xrightarrow{\simeq}
  \Sigma^{ V_{\lambda_j}^{\lambda'_j} } E^{\mu'_j}_{\lambda'_j} (\alpha(t_{j})), \\
  \sigma_{j+1} &\colon E^{\mu_j}_{\lambda_j} (\alpha(t_{j+1})) \xrightarrow{\simeq}
  \Sigma^{ V_{\lambda_j}^{\lambda'_j} } E^{\mu'_j}_{\lambda'_j} (\alpha(t_{j+1})).
\end{align*}
Using the formula in \cite[Theorem 6.7]{Salamon}, we can easily see that  $\sigma_{j+1} \circ \rho_j$ is $S^1$-equivariantly homotopic to   $ \rho_{j}' \circ \sigma_{j}$. This implies that  $\Sigma^{-V_{\lambda_j}^0} \rho_j$ and  $\Sigma^{-V_{\lambda_j'}^{0}} \rho_j'$  are equal to each other as morphisms in $\mathfrak{C}$. Therefore $\rho(B, \alpha)$ does not depended on the choices of $\{ \lambda_j \}$ and $\{ \mu_j \}$.
Next we prove the independence of the choice of $\{ t_j \}$. Let us pick another sequence $\{ t_j' \}_{j=1}^{m'}$. Without loss of generality, we will only work on  the case $\{ t_j' \} \subset \{ t_j \}$, i.e. $\{ t_j \}$ is a finer subdivision. Let us suppose that
\[
          t_j = t_{j'}' < t_{j+1} < t_{j+2} = t_{j'+1}'
\]
for some $j' \in \{ 1, \dots, m' \}$. An equivariant version of \cite[Corollary~6.8]{Salamon} implies that $\rho_{j+1} \circ \rho_j$ is $S^1$-equivariantly homotopic to $\rho_{j'}'$. This discussion implies that $\rho(B, \alpha)$ is independent of the choice of $\{ t_j \}$.

Now suppose that we have two paths $\alpha_{0},\alpha_{1}$ which are homotopic to each other relative to their end points by a homotopy  $\alpha_{u}$ as $u \in [0,1] $. For any $(t_{0},u_{0}) \in [0,1]^2 $, one can also find $\mu_{}>\bar{\mu}$
and $\lambda_{}<\bar{\lambda}$ and a small neighborhood $O$ of $(t_{0},u_{0}) $ such that $\mu, \lambda$ are not eigenvalues of $l_{\alpha_u (t)} $ for any $(t,u)$ in $O$.
By the definition of $\rho$ and the homotopy invariance of the Conley index, we see that $\rho(B,\alpha_{u})$ does not change as $u$ varies inside $O$. By considering a finite cover of $[0,1]^2$ by such neighborhoods, we see that $\rho(B,\alpha_{0})=\rho(B,\alpha_{1})$. This finishes the proof of the theorem.
\end{proof}
The following corollary is directly implied by the homotopy invariance of the attractor-repeller map.
\begin{cor}  Let $B_{1}\subset B_{2}$ be two closed and bounded sets both satisfying the hypothesis of Theorem~\ref{finite dimensional approximation for a family of flows}. Suppose that for any sufficiently large $-\lambda, \mu$ and any $s\in S$, the set $\operatorname{inv}(\varphi(\lambda, \mu, s), B_{1}\cap V^{\mu}_{\lambda}(s))$ is an attractor
in  $\operatorname{inv}(\varphi(\lambda, \mu,s), B_{2}\cap V^{\mu}_{\lambda}(s))$. Then the desuspensions of the corresponding
attractor maps give well defined morphisms $i(s):I(B_{1},s)\rightarrow I(B_{2},s)$ and $\bar{i}(s):\bar{I}(B_{1},s)\rightarrow
\bar{I}(B_{2},s)$. Moreover, for any path $\alpha:[0,1]\rightarrow S$, we have
$$\rho(B_{2},\alpha)\circ  i(\alpha(0))=(\Sigma^{\spf(-\slashed{D},\alpha)\mathbb{C}}i(\alpha(1)))\circ \rho(B_{1},\alpha),$$
$$\bar{\rho}(B_{2},\alpha)\circ  \bar{i}(\alpha(0))=(\Sigma^{\spf(-\slashed{D},\alpha)\mathbb{C}}\bar{i}(\alpha(1)))\circ \bar{\rho}(B_{1},\alpha).$$

The repeller version of this result also holds given that  $\operatorname{inv}(\varphi(\lambda, \mu, s), B_{1}\cap V^{\mu}_{\lambda}(s))$
is a repeller in  $\operatorname{inv}(\varphi(\lambda, \mu,s), B_{2}\cap V^{\mu}_{\lambda}(s))$ for any $s\in S$.
\end{cor}

\subsection{The invariance for (III),(IV),(V)}\label{subsection invariance III,IV,V}

Notice that the three parameters in (V) only affect our results through the definition of the bounded set $J^{\pm}_{m}$. Suppose that we choose two different triples of parameters $(\{h_{j}\},\bar{g},\theta)$ and $(\{\tilde{h}_{j}\},\tilde{g},\tilde{\theta})$ and use them to define the sets $J^{+}_{m}$ and $\tilde{J}^{+}_{m}$ respectively. From these subsets, we construct two direct systems, which we denote by (\ref{attractor system}) and (\ref{attractor system}') respectively. Notice that $J^{+}_{m}$ and $\tilde{J}^{+}_{m}$ are bounded subsets of $Str(\tilde{R})$. We can find $0< m_{1}< m_{2}...$ and $0< \tilde{m}_{1}<\tilde{m}_{2}<...$ such that:
\begin{equation}\label{mixed system of sets}
J^{+}_{m_{1}}\subset \tilde{J}^{+}_{\tilde{m}_{1}}\subset J^{+}_{m_{2}}\subset \tilde{J}^{+}_{\tilde{m}_{2}}\subset \cdots ,\end{equation}
which also implies the following inclusions for any positive integer $n$
$$ J^{n,+}_{m_{1}}\subset \tilde{J}^{n,+}_{\tilde{m}_{1}}\subset J^{n,+}_{m_{2}}\subset \tilde{J}^{n,+}_{\tilde{m}_{2}}\subset \cdots {.}$$
Notice that for any $j>0$ and any $n,m$ large enough relative to $m_{j},\tilde{m}_{j}$. The flow $\varphi^{n}_{m}$ goes inside $J^{n,+}_{m_{j}}$ and $\tilde{J}^{n,+}_{\tilde{m}_{j}}$ along $\partial J^{n,+}_{m_{j}}\setminus \partial Str( \tilde{R} )$
and $\partial \tilde{J}^{n,+}_{\tilde{m}_{j}}\setminus \partial Str( \tilde{R})$ respectively. Therefore, the attractor maps, together with the isomorphisms $\tilde{\rho}^{*,+}_{*}$ (as defined in Proposition \ref{stability of the conley index}) give a direct system in the category $\mathfrak{C}$
\begin{equation}\label{mixed attractor system}
I^{n_{1},+}_{m_{1}}\rightarrow \tilde{I}^{\tilde{n}_{1},+}_{\tilde{m}_{1}}\rightarrow I^{n_{2},+}_{m_{2}}\rightarrow \tilde{I}^{\tilde{n}_{2},+}_{\tilde{m}_{2}}\rightarrow \cdots
\end{equation}
for suitable choices of $n_{1}<\tilde{n}_{1}<n_{2}<\tilde{n}_{2} < \cdots$, where the connecting maps are defined in a similar way as (\ref{attractor system}). Since the attractor maps are transitive  as mentioned after Proposition \ref{Attractor-repeller-exact sequence}, the composition of the connecting morphisms $I^{n_{m_{j}},+}_{m_{j}}\rightarrow \tilde{I}^{\tilde{n}_{j},+}_{\tilde{m}_{j}}\rightarrow I^{n_{j+1},+}_{m_{j+1}}$ is the same as the attractor map for $\text{inv}(J^{n,+}_{m_{j}})\subset \text{inv}(J^{n,+}_{m_{j+1}})$. Therefore, we see that (\ref{mixed attractor system}) contains both a subsystem of (\ref{attractor system}) and a subsystem of (\ref{attractor system}'). By Lemma \ref{subsyemtem}, this implies that (\ref{attractor system}) and (\ref{attractor system}') are canonically isomorphic as objects of $\mathfrak{S}$ . In other words, up to canonical isomorphisms, the spectrum invariants $\underline{\text{swf}}^{A}$ and $\underline{\text{SWF}}^{A}$ do not depend on the choice of $\{h_{j}\},\bar{g}$ and $\theta$. The case of $\underline{\text{swf}}^{R}$ and $\underline{\text{SWF}}^{R}$ can be shown similarly. We have proved the invariance for (V).

The proof of the invariance for (IV) is easy: Let $\tilde{R}_{0}<\tilde{R}_{1}$ be two numbers which are both larger than the constant $R_{0}$ from Theorem~\ref{Boundedness of trajectories}. Notice that when we choose a suitable choice of parameters   $(\{h_{j}\},\bar{g},\theta)$ for $\tilde{R}_{1}$, these parameters also work for $\tilde{R}_{0}$ since $\tilde{R}_{0}<\tilde{R}_{1}$. Denote by $J^{n,\pm}_{m}(\tilde{R}_{i})$ the corresponding bounded set corresponding for $i=0,1$. Then it is straightforward to see that, for any positive integer $m$ and any sufficiently large integer $n$ (relative to $m$), the sets $J^{n,\pm}_{m}(\tilde{R}_{0})$ and $J^{n,\pm}_{m}(\tilde{R}_{1})$ are both isolating neighborhoods of the same isolated invariant set. Therefore, their Conley indices are related to each other by canonical {isomorphisms} which are compatible with attractor-repeller maps. This implies the invariance for (IV).
\begin{rmk}\label{different strips}
Actually, from the above argument, we can replace $Str(R)$ in our construction with any set $C \subset Coul(Y)$ satisfying the following conditions:
\begin{enumerate}
\item For any bounded subset $A\subset i\Omega^{1}_{h}(Y)$, the set $p_{\mathcal{H}}^{-1}(A)\cap C$ is also bounded;
\item Any finite type Seiberg-Witten trajectory is contained in the interior of $C$.
\end{enumerate}
Also, we can define $\{J^{\pm}_{m}\}$ to be any sequence of bounded, closed subsets of $C$ such that $J_m^{\pm} \subset J_{m+1}^{\pm}$, $\cup_{m=1}^{\infty} J_m^{\pm} = C$ and for any $m>0$ and $n$ large enough relative to $m$ the flow $\varphi^{n}_{m}$ goes inside (resp. outside) $J^{n,+}_{m}$ (resp. $J^{n,-}_{m}$) along $\partial J^{n,+}_{m}\setminus \partial C$ (resp. $\partial J^{n,-}_{m}\setminus \partial C $).
\end{rmk}

As for (III),  we choose different sequences $\{\lambda_{n}\},\{\mu_{n}\}$ and $\{\tilde{\lambda}_{n}\},\{\tilde{\mu}_{n}\}$. By Lemma~\ref{subsyemtem}, we can pass to their subsequences and assume that $\lambda_{n+1}<\tilde{\lambda}_{n}<\lambda_{n}$ and $\mu_{n}<\tilde{\mu}_{n}<\mu_{n+1}$  for any $n$. Let $I^{n,+}_{m}$ and $\tilde{I}^{n,+}_{m}$ be the objects of $\mathfrak{C}$ obtained by desuspending the Conley indices corresponding to $\{\lambda_{n}\},\{\mu_{n}\}$ and $\{\tilde{\lambda}_{n}\},\{\tilde{\mu}_{n}\}$ respectively. We can repeat the proof of Proposition~\ref{stability of the conley index} and establish canonical isomorphisms $I^{n,+}_{m}\cong \tilde{I}^{n,+}_{m}$ and $I^{n+1,+}_{m} \cong \tilde{I}^{n+1,+}_{m}$ for any positive integer $m$ and any  sufficiently large integer $n$ (relative to $m$). Moreover, they form commutative diagrams similar to (\ref{commutative diagram for attractor}). This implies that $\underline{\text{swf}}^{A}$ and $\underline{\text{SWF}}^{A}$ are independent of (III). The repeller case follows in the same manner.

\subsection{The invariance for (II)}\label{subsection invariance II} In this subsection, we will consider any two choices of good perturbation $f_{j} \colon \mathcal{C}_{Y}\rightarrow \mathbb{R}$  for $j=1,2$.
Recall that $f_{j}(a,\phi)=\frac{\delta_{j}}{2}\|\phi\|_{L^{2}}^{2}+\bar{f}_{j}(a,\phi)$, where $\delta_{j}$ is a real constant and $\bar{f}_{j}$ is an extended cylinder function. We first assume that $\delta_{1}=\delta_{2}=\delta$. Since we do not know whether the space of good perturbation is path connected, the usual homotopy invariance argument does not work. Therefore, we follow a different approach here. Because the whole argument is relatively long and technical, we first sketch the the rough idea as follows.

Denote by ${\mathcal{L}}_{j}$ the restriction of $CSD_{\nu_{0},f_{j}}$ to $Coul(Y)$. Recall that we identify $i\Omega^{1}_{h}(Y)$ with $\mathbb{R}^{b_1}$ by choosing independent harmonic forms $\{ h_{j} \}$. For any real number $e\geq 1$, we will construct a family of ``mixed'' functionals $\stL^s_e $ for $s \in [0,1]$ such that $\stL^1_e = {\mathcal{L}}_{2} $ and $\stL^{0}_{e}$ equals ${\mathcal{L}}_{1}$ on $p_{\mathcal{H}}^{-1}([-e+1,e-1]^{b_1})$ and equals
${\mathcal{L}}_{2}$ on $p_{\mathcal{H}}^{-1}(\mathbb{R}^{b_1}\setminus (-e,e)^{b_1})$. Suppose that all finite type flow lines of $\mathcal{L}^{s}_{e}$ are contained in $Str( \tilde{R} )$ and consider an increasing sequence of bounded subsets
$$
J^{+}_{m_{1}}\subset \tilde{J}^{+}_{\tilde{m}_{1}}\subset  J^{+}_{m_{2}}\subset \tilde{J}^{+}_{\tilde{m}_{2}}\subset \cdots
$$
where $J^{+}_{m_{j}}$  and $\tilde{J}^{+}_{m_{j}}$ are the bounded subsets of $Str( \tilde{R} )$ corresponding to ${\mathcal{L}}_{1}$ and ${\mathcal{L}}_{2}$ respectively. We will require that, for each positive integer $j$, there exists a real number $e_{j} \ge 1 $ satisfying
$$
        J^{+}_{m_{j}}
     \subset p_{\mathcal{H}}^{-1}([-e_{j}+1,e_{j}-1]^{b})  \cap Str(\tilde{R})
      \subset p_{\mathcal{H}}^{-1}([-e_{j},e_{j}]^{b}) \cap Str( \tilde{R})
      \subset J^{+}_{\tilde{m}_{j}}.
$$
Let $\varphi^n (\mathcal{L}) $ be the approximated gradient flow of $\mathcal{L} $ on the compact set $J^{n,+}_{m_{}} $. Since ${\mathcal{L}}_{1}$ equals $\stL^{0}_{e_{j}}$ when restricted to $J^{+}_{m_{j}}$ and the flow goes inside $J^{n,+}_{m_{j}}$, we have an attractor map
\[
   I_{S^{1}}(\varphi^{n}({\mathcal{L}}_{1}),\text{inv}(J^{n,+}_{m_{j}})) = I_{S^{1}}(\varphi^{n}(\stL^{0}_{e_{j}}),\text{inv}(J^{n,+}_{m_{j}}))\rightarrow I_{S^{1}}(\varphi^{n}(\stL^{0}_{e_{j}}),\text{inv}(\tilde{J}^{n,+}_{\tilde{m}_{j}})).
\]
On the other hand, we have
$I_{S^{1}}(\varphi^{n}(\stL^{0}_{e_{j}}),\text{inv}(\tilde{J}^{n,+}_{\tilde{m}_{j}}))\cong I_{S^{1}}(\varphi^{n}(\mathcal{L}_{2}),\text{inv}(\tilde{J}^{n,+}_{\tilde{m}_{j}}))$
by continuity of Conley indices. We combine these  and obtain a map
$$
I_{S^{1}}(\varphi^{n}({\mathcal{L}}_{1}),\text{inv}(J^{n,+}_{m_{j}}))\rightarrow I_{S^{1}}(\varphi^{n}({\mathcal{L}}_{2}),\text{inv}(\tilde{J}^{n,+}_{\tilde{m}_{j}})).
$$

We also construct another family of functionals $\tilde{\stL}^{s}_{e}$ to obtain a map $I_{S^{1}}(\varphi^{n}({\mathcal{L}}_{2}),\text{inv}(\tilde{J}^{n,+}_{ \tilde{m}_{j}}))\rightarrow  I_{S^{1}}(\varphi^{n}({\mathcal{L}}_{1}),\text{inv}(J^{n,+}_{m_{j+1}}))$.
We will then prove that the composition
\[
      I_{S^{1}}(\varphi^{n}({\mathcal{L}}_{1}),\text{inv}(J^{n,+}_{m_{j}}))
              \rightarrow
      I_{S^{1}}(\varphi^{n}({\mathcal{L}}_{2}),\text{inv}(\tilde{J}^{n,+}_{ \tilde{m}_{j}}))
              \rightarrow
      I_{S^{1}}(\varphi^{n}({\mathcal{L}}_{1}),\text{inv}(J^{n,+}_{m_{j+1}}))
\]
is just the attractor map corresponding to ${\mathcal{L}}_{1}$. A similar result holds for ${\mathcal{L}}_{2}$. Therefore, we have constructed a ``mixed direct system'' in the category $\mathfrak{C}$ and the spectra corresponding to $f_{1},f_{2}$ are both subsequential colimit of it. Therefore, the invariance of $\underline{\text{swf}}^{A}$ is implied by Lemma~\ref{subsyemtem}. The $\underline{\text{swf}}^{R}$ case can be proved similarly.

There is one technical difficulty here. We need to find a uniform constant $R_{2}$ (independent of $e,s$) such that $Str(R_{2})$ contains all the finite type trajectories of $\stL^{s}_{e}$ and $\tilde{\stL}^{s}_{e}$. This will be taken care by Lemma~\ref{the uniform bound for trajectories} and Lemma~\ref{double bump functions}, which generalize Theorem~\ref{Boundedness of trajectories}.


Let us prepare some general results regarding the perturbations. Recall that we have a canonical isomorphism
$$
\pi_{0}(\mathcal{G}_{Y})\cong \pi_{0}(\mathcal{G}^{h}_{Y})\cong H^{1}(Y;\mathbb{Z}).
$$
 For any positive integer $m$, we denote by $m\mathcal{G}_{Y}$ (resp. $m\mathcal{G}_{Y}^{h}$) the subgroup of $\mathcal{G}_{Y}$ (resp. $\mathcal{G}_{Y}^{h}$) consisting of the connected components corresponding to $m\cdot H^{1}(Y;\mathbb{Z})$.

\begin{defi} \label{def m-periodic}
For a positive integer $m$, a continuous function $f \colon Coul(Y)\rightarrow \mathbb{R}$ is called $m$-periodic
if $f$ is invariant under the action of $m\mathcal{G}_{Y}^{h}$, which implies that $f\circ \Pi$ is invariant under $m\mathcal{G}_{Y}$.
\end{defi}

We will also need the following definition of tame functions.
\begin{defi} \label{m-periodic tame functions}
A smooth function $f \colon Coul(Y)\rightarrow \mathbb{R}$ is called a tame function if the formal gradient $\operatorname{grad} (f\circ \Pi)$  satisfies all the conditions of the tame perturbations \cite[Definition~10.5.1]{Kronheimer-Mrowka} except that it needs not be invariant under the full gauge group $\mathcal{G}_{Y}${,}   where $\Pi:\mathcal{C}_Y \rightarrow Coul(Y)$ is the non-linear Coulomb projection.

Furthermore, a continuous family of functions $\{ f_w \} $ parametrized by a compact manifold $W$ (possibly with boundary)  is called  a continuous family of tame functions if each function is tame and $\grad{(f_w \circ \Pi)} $ extends to a continuous family of maps on the cylinder $I \times Y $. In addition, we require that the constant $m_{2}$ and the function $\mu_{1}$ from \cite[Definition~10.5.1]{Kronheimer-Mrowka} are uniform with respect to $w \in W$.

%
%
\end{defi}

Now we describe a way to construct a continuous family of tame functions from any pair of extended cylinder function, given a family of smooth function.

\begin{lem} \label{lem f_w continuous family}
Let $W$ be a compact manifold and $\tilde{\tau}_{w} \colon i\Omega_{h}^{1}(Y)\cong\mathbb{R}^{b_1}\rightarrow \mathbb{R} $ be a smooth family of smooth functions parametrized by $w \in W $.
Then, we can choose a sequence of constants $\{C_{j}\}$ in the definition of the space of perturbations $\mathcal{P}$ (c.f. Definition~\ref{def space perturbation}) so that, for any $\delta \in \mathbb{R}$ and any $\bar{f}_{1},\bar{f}_{2}\in \mathcal{P}$, a family of functions $\tilde{f}_{w} \colon Coul(Y) \rightarrow \mathbb{R}$ given by
\begin{equation}\label{mixed tame functions}
         \tilde{f}_{w}(a,\phi)
     :=    \frac{\delta}{2} \| \phi \|^{2}_{L^{2}}  +
            (\tilde{\tau}_{w} \circ  \pi_{\mathcal{H}} (a))\cdot \bar{f}_{1}(a,\phi) +
            (1-\tilde{\tau}_{w} \circ\pi_{\mathcal{H}} (a))\cdot \bar{f}_{2}(a,\phi).
\end{equation}
is a continuous family of tame functions. Moreover, if $\tilde{\tau}_w $ is $m \mathbb{Z}^{b_1} $-periodic, then $\tilde{f}_w $ is $m$-periodic.
\end{lem}
\begin{proof}
This is actually a parametrized version of \cite[Theorem~11.6.1]{Kronheimer-Mrowka} and we will focus only on the
term $(\tilde{\tau}_{w} \circ  \pi_{\mathcal{H}} (a))\cdot \bar{f}_{1}(a,\phi)$.
To avoid repeating complicated analysis
there, we introduce a trick turning a family of functions into a single function. Let $Y'$ be another spin$^{c}$
3-manifold with $b_{1}(Y') > 2\dim W$ so that we can embed $W$ in the torus $i\Omega_{h}^{1}(Y')/\mathcal{G}_{Y'}^{h,o}$.
We now consider the family $\{\tilde{\tau}_{w}\}_{w \in W}$ as a single function on $i\Omega_{h}^{1}(Y) \times W$ and extend it to $\tilde{\tau} \colon i\Omega_{h}^{1}(Y)\times i\Omega_{h}^{1}(Y')\rightarrow
\mathbb{R}$. Recall that $\bar{f}_1 =  \sum_{j=1}^{\infty} \eta_j \hat{f}_j$, where $\hat{f}_j $ is a cylinder function of $Y$ with $\sum_{j=1}^{\infty} C_j \left\vert\eta_j\right\vert < \infty  $.
We define a function $$\hat{f}'_{j} \colon \mathcal{C}_{Y}\times \mathcal{C}_{Y'}\rightarrow \mathbb{R}$$
given by
$$
(a,\phi)\times (a',\phi')\mapsto \bar{\tau}(\pi_{\mathcal{H}}'(a'),\pi_{\mathcal{H}}(a))\cdot \hat{f}_{j}(a,\phi),
$$
where $\pi_{\mathcal{H}}' \colon i\Omega^{1}(Y')\rightarrow i\Omega^{1}_{h}(Y')$ denotes the projection onto harmonic forms on $Y'$.
These functions $\hat{f}'_j $ almost fit into the definition of cylinder functions (cf. \cite[Section~11]{Kronheimer-Mrowka}), on $\mathcal{C}(Y)\times \mathcal{C}(Y')$. We can still repeat the argument the proof of \cite[Theorem~11.6.1]{Kronheimer-Mrowka} and show that, by setting $\{C_{j}\}$ to increase fast enough, the formal gradient $\operatorname{grad}(\sum_{j}\eta_{j}\hat{f}'_{j})$ is a tame perturbation for the manifold $Y\cup Y'$ except that it is not invariant under the full gauge group. As a result, it is not hard to see that this actually implies that $(\tilde{\tau}_{w} \circ  \pi_{\mathcal{H}} (a))\cdot \bar{f}_{1}(a,\phi)$
 is a continuous family of tame functions.
\end{proof}

For a general functional $\mathcal{L} \colon Coul(Y)\rightarrow \mathbb{R}$, we can consider its negative gradient flow line $\gamma \colon I\rightarrow Coul(Y)$, described by the equation $\frac{d\gamma(t)}{dt}=-\gradtil\mathcal{L}(\gamma(t))$. Such a trajectory will be called an $\mathcal{L}$-trajectory. As before, we define the topological energy by

$$
   \mathcal{E}^{\text{top}}(\gamma,\mathcal{L})
      :=2(\mathop{\sup} \limits_{t \in I}\mathcal{L}( \gamma(t) ) - \mathop{\inf} \limits_{t\in I}\mathcal{L}( \gamma(t) ) ) .
$$
Recall that a trajectory is called finite type if it is contained in a bounded subset of $Coul(Y)$. We have the following uniform boundedness result for functionals perturbed by periodic functions.

\begin{pro}\label{uniform bound for compact set}
Let $\{{f}_{w} \}$ be a continuous family of $m$-periodic tame functions parametrized by a compact manifold $W$ and consider a family of functionals $\mathcal{L}_{w}=CSD_{\nu_{0}}|_{Coul(Y)} + f_{w}$.  Then for any $C>0$, there exist constants $R,C'$ such that for any $w \in W$ and any $\mathcal{L}_{w}$-trajectory $\gamma \colon [-1,1]\rightarrow Coul(Y)$ with topological energy $\mathcal{E}^{\text{top}}(\gamma,\mathcal{L}_{w})\leq C $, we have $\gamma(0)\in Str(R)$ and $|\mathcal{L}_{w}(\gamma(0))|\leq C'$.
\end{pro}

\begin{proof}
The proof is a slight adaption of \cite[Theorem~10.7.1]{Kronheimer-Mrowka}.  Suppose that the statement is not true. Then we can find a sequence  $\{\gamma_{j}\}$ of $\mathcal{L}_{w_{j}}$-trajectory $\gamma_{j} \colon[-1,1]\rightarrow Coul(Y)$
with $w_j \in W $  such that at least one of the following two situations happens:
\begin{itemize}
\item $\limsup_{j \rightarrow \infty} \|u_{j}\cdot \gamma_{j}(0)\|_{L^{2}_{k}} =\infty$  for any sequence $\{u_{j}\} \subset  m \mathcal{G}^{h}_{Y}$; \item $\limsup_{j \rightarrow \infty} | \mathcal{L}_{w_j}( \gamma_j(0)) | = \infty$.
\end{itemize}
Since $W$ is compact,  after passing to a subsequence, we may assume that $w_{j}\rightarrow w_{\infty}$.

We lift $\gamma_{j}$ to a path $\tilde{\gamma}_{j} \colon [-\frac{1}{2},\frac{1}{2}]\rightarrow \mathcal{C}_{Y}$, which is the negative gradient flow line of $CSD_{\nu_{0}} +  f_{w_{j}} \circ \Pi$. Note that we only consider an interior domain here to avoid a possible regularity issue. With $X=[-\frac{1}{2},\frac{1}{2}]\times Y$, we treat $\tilde{\gamma}_{j}$ as a section over $X$ and denote it by $(\hat{a}_{j},\hat{\phi}_{j})$. We can find a gauge transformation $\hat{u}_{j}$ over $X$  whose  restrictions to $\{0\}\times Y$ belong to $m\mathcal{G}_{Y}$ such that the following conditions hold:
\begin{enumerate}
\item
$\hat{d}^* ( \hat{a}_j -  \hat{u}_j^{-1} d \hat{u}_j) = 0$ on $X${ ;}

\item
$(\hat{a}_{j}-\hat{u}_{j}^{-1}d\hat{u}_{j})(\mathbf{n})=0$ on $\partial
X$, where $\mathbf{n}$ is the outward normal vector ;

\item For each for $l=1,\ldots, b_1 $, we have
 $\int_{\tilde{X}} (\hat{a}_{j}-\hat{u}_{j}^{-1}d\hat{u}_{j})\wedge (*_{4}\hat{h}_{l})\in [0,m)$ where $\hat{h}_{l}$ is the pull-back of $h_{l}$ on $X$;


\end{enumerate}
 The conditions in Definition~\ref{m-periodic tame functions} allow us to repeat the bootstrapping argument in the proof of \cite[Theorem~10.7.1]{Kronheimer-Mrowka} and obtain the following statement. After passing to a further subsequence, $(\hat{a}_{j}-\hat{u}_{j}^{-1}d\hat{u}_{j}, \hat{u}_{j}\cdot \hat{\phi}_{j})$ is  convergent in $L^{2}_{k+\frac{1}{2}}$ when restricted to any interior cylinder. In particular, this implies that $\Pi(\hat{u}_{j}|_{\{0\}\times Y}\cdot \tilde{\gamma}_{j}(0))$ is convergent in $L^{2}_{k}$. Notice that $\Pi(\hat{u}_{j}|_{\{0\}\times Y}\cdot \tilde{\gamma}_{j}(0))$ equals $u_{j}\cdot \gamma_{j}(0)$ for some $u_{j}\in m \mathcal{G}_{Y}^{h}$.  Also $\mathcal{L}_{w_j}( \gamma_j(0)) = \mathcal{L}_{w_j}( u_j \cdot \gamma_j(0))$ is a convergent sequence since $\mathcal{L}_w(a, \phi)$ is continuous in $w$ and $(a, \phi)$. Therefore, we arrive at a contradiction with the above two situations.

\end{proof}

We also have the following lemma, whose
proof is essentially the same as Lemma~\ref{c is quadratic-like} and we omit it.

\begin{lem}\label{a family of nonlinear term}
Let $\{{f}_{w}\}$ be a continuous family of tame functions. For
each $w \in W$, we define a nonlinear term $c_{w} \colon Coul(Y)\rightarrow L^{2}_{k}(i\ker d^{*}\oplus \Gamma(S_{Y}))$ of the gradient of $CSD_{\nu_{0}}|_{Coul(Y)} + f_{w}$ as in (\ref{def of c^1}) and (\ref{def of c^2}). Then $\{ c_{w}\}$ is a continuous family of quadratic-like maps.
\end{lem}

Now we construct explicit ``the mixed perturbation'' as follows. Choose a smooth function  $\tau \colon\mathbb{R}\rightarrow [0,1]$  satisfying $\tau|_{(-\infty,\frac{1}{4}]}\equiv 0$ and $\tau|_{[\frac{1}{2},\infty)}\equiv 1$. For any real number $e \geq 1$, we define a bump function $\tau_{e} \colon i\Omega^{1}_{h}(Y)\rightarrow [0,1]$ from $\tau$ by
$$\tau_{e}(x_{1},x_{2},...,x_{b_1})=\prod_{1\leq j\leq b_1}\tau(e+x_{j})\tau(e-x_{j}).$$
Each $\tau_e$   gives an induced tame function ${\tilde{f}}_{e}^{0} \colon Coul(Y)\rightarrow \mathbb{R}$ as in (\ref{mixed tame functions}), i.e.
\begin{equation*}
       {\tilde{f}}_{e}^{0}(a,\phi) :=
       \frac{\delta}{2} \|\phi\|_{L^{2}}^{2} +
       (\tau_{e} \circ p_{\mathcal{H}}(a,\phi)))  \cdot \bar{f}_{1}(a,\phi)+
       (1-\tau_{e} \circ p_{\mathcal{H}}(a,\phi)) \cdot \bar{f}_{2}(a,\phi),
\end{equation*}
where $\bar{f}_1 , \bar{f}_2 \in \mathcal{P}$.
With $f_j = \frac{\delta}{2} \|\phi\|_{L^{2}}^{2} + \bar{f}_j$, we note that the function ${\tilde{f}}_{e}^{0} $ equals $f_{1}$ on $p_{\mathcal{H}}^{-1}([-e+1,e-1]^{b})$ and equals
$f_{2}$ on $p_{\mathcal{H}}^{-1}(\mathbb{R}^{b}\setminus (-e,e)^{b})$. For $s\in [0,1]$, we also consider an interpolation
$\tau^s_e = (1-s)\tau_e $ and define
 \begin{equation} \label{eq familymixfunc}
    {\tilde{f}}_{e}^{s} = (1-s) {\tilde{f}}_{e}^{0}+s  f_{2} \text{ and }
     \stL^{s}_{e}=CSD_{\nu_{0}}|_{Coul(Y)}+{\tilde{f}}^{e}_{s}.
\end{equation}
Notice that ${\tilde{f}}^s_e $ is essentially a tame function induced from $\tau^s_e $ which is not $m$-periodic for any positive integer $m$. To utilize Proposition~\ref{uniform bound for compact set}, we will introduce an explicit family of smooth periodic functions such that the induced periodic tame functions agree with ${\tilde{f}}^s_e $ on desirable regions.

For any positive integer $M$, we consider a family of $(6M+6)$-periodic smooth functions parametrized by  compact manifold $W_M$ described as follows. The manifold $W_M$ is of the form $W_{M,1} \amalg W_{M,2} \amalg W_{M,3} $ where $W_{M,1} :=[1,M+1] \times [0,1] $ and $W_{M,2} := \{ (B,\sigma) \mid \emptyset \neq B \subset \{1, 2, \ldots , b_1 \} \text{ and } \sigma \colon B \rightarrow \{\pm1\} \} \times (\mathbb{R}/(6M+6)\mathbb{Z}) \times [0,1]$ and $W_{M,3} = \{1,2\}$. We construct a family of functions $\{ \tilde{\tau}_w \} $  as following:
\begin{itemize}

\item For each positive integer $M$ and  $(e,0) \in W_{M,1} $, we assign the unique $(6M+6)\mathbb{Z}^{b_1}$-periodic
function $\tilde{\tau}_{e} \colon \mathbb{R}^{b_1} \rightarrow \mathbb{R}\ $ which extends $\tau_{e}|_{[-3M-3, 3M+3]^{b_1}}$.

\item For each positive integer $M$, we pick a $(6M+6)$- periodic function $\bar{\tau}_M \colon \mathbb{R} \rightarrow [0,1] $ which extends $\tau|_{[-2M-2, 2M+2]}$. For each $(B,\sigma, \theta,0) \in W_{M,2} $,
we assign a function $\tilde{\tau}_{B,\sigma,\theta} \colon \mathbb{R}^{b_1}\rightarrow
[0,1]$ given by
$$
           \tilde{\tau }_{(B,\sigma,\theta)}(x_{1},...,x_{b_1}):=\prod_{j\in B} \bar{\tau}_M (\theta+\sigma(j) x_{j}).
$$

\item For general $w= (w' , s) \in W_{M,1} \coprod W_{M,2}$, we simply define $\tilde{\tau}_{(w',s)}:=(1-s)\tilde{\tau}_{w'}. $
\item We set $\tilde{\tau}_j \equiv 2-j $ for $j \in W_{M,3} $ so that $\tilde{f}_j = f_j $.
\end{itemize}

\begin{lem}\label{restriction to the cube}
For each positive integer $M$, any $(s,e) \in [0,1] \times [1,\infty)$ and $(e_{1},e_{2},...,e_{b_1})\in
\mathbb{R}^{b_1}$, there exists an element $w \in W_M $ such that the induced function $\tilde{f}_w $ equals ${\tilde{f}}_{e}^{s} $  on  $p_{\mathcal{H}}^{-1}([e_{1}-M,e_{1}+M]\times ...\times [e_{b_1}-M,e_{b_1}+M])$.
\end{lem}

\begin{proof}



%


For convenience, we denote $E=[e_{1}-M,e_{1}+M]\times...\times [e_{b_1}-M,e_{b_1}+M]$. We will consider two main cases with several subcases:

Case  $e \in [1,M+1]$; If $E\cap [-M-1,M+1]^{b_1}\neq \emptyset$, then we have $E\subset [-3M-3,3M+3]^{b_1}$. This implies $\tilde{\tau}_{e}|_{A}=\tau_{e}|_{A}$. Therefore, we can just choose $w = (e,s) \in W_{M, 1} $.
If $E \cap [-M-1,M+1]^{b_1}=\emptyset$, then we have $p_{\mathcal{H}}^{-1}(E)\subset p_{\mathcal{H}}^{-1}(\mathbb{R}^{b_1}\setminus (-e,e)^{b_1})$ and ${\tilde{f}}_{e}^{s} = f_{2}$ on $p_{\mathcal{H}}^{-1}(E)$. We just take $w = 2 \in W_{M,3}$ so that $\tilde{f}_w = f_2 $ in this case.

Case $e>M+1$; We consider the following subsets of $[1,2,...,b]$:
\[
   \begin{split}
     & B_{1}=\{ j \mid [e_{j}-M,e_{j}+M] \cap [e-1,e]\neq \emptyset \},   \\
     & B_{2}=\{ j \mid [e_{j}-M,e_{j}+M] \cap [-e,-e+1]\neq \emptyset \},  \\
     & B_{3}=\{ j \mid [e_{j}-M,e_{j}+M] \cap [-e,e]=\emptyset \}.
  \end{split}
\]
 If $B_1 \cup B_2=\emptyset$, then $E$ is either contained in  $[-e+1,e-1]^{b_1}$ or
$\mathbb{R}^{b_1}\setminus (-e,e)^{b_1}$ and we can just take $w \in W_{M, 3}$. If $B_{3}\neq \emptyset$, then we have $\tau_{e}|_{E}\equiv 0$ and ${\tilde{f}}^{s}_{e}|_{p_{\mathcal{H}}^{-1}(E)}=f_{2}|_{p_{\mathcal{H}}^{-1}(E)}$. We can take $w = 2 \in W_{M, 3}$ again in this subcase. We are now left with the case $B_1 \cup B_2 \neq \emptyset$ and $B_{3}=\emptyset$. Notice that for any $(x_{1},...,x_{b_1})\in E$, the following holds:
\[
    \begin{split}
        & j  \in B_{1} \Rightarrow e+x_{j}\geq 2e-1-2M\geq 1 \text{ and } e-x_{j}\in [-2M,2M+1]; \\
        & j  \in B_{2} \Rightarrow e-x_{j}\geq 2e-1-2M\geq 1 \text{ and } e+x_{j}\in [-2M,2M+1]; \\
        & j   \notin B_1 \cup B_2\Rightarrow e - \left\vert x_{j}\right\vert\geq 1.
   \end{split}
\]
Therefore, for such $(x_{1},...,x_{b_1})$, we have
\[
     \tau_{e}(x_{1},...,x_{b})
   =  \prod_{j\in B_{1}} \tau(e-x_{j})\cdot \prod_{j \in B_{2}}\tau(e+x_{j})
   =  \prod_{j\in B_{1}} \bar{\tau}_M (e-x_{j})\cdot \prod_{j \in B_{2}}\bar{\tau}_M (e+x_{j}),
\]
where we use the fact that $\bar{\tau}_M |_{[-2M-2,2M+2]}=\tau|_{[-2M-2,2M+2]}$. As a result, we see that ${\tilde{f}}^{s}_{e} = \tilde{f}_{w}$  on $p_{\mathcal{H}}^{-1}(E)$ when  we set
$ w = (B_1 \cup B_2 ,\sigma , e ,s) \in W_{M,2}$ with
  $\sigma \colon B_1 \cup B_2 \rightarrow \{\pm 1\}$ sending $B_{1}$ to $-1$ and $B_{2}$ to $1$. Notice that $B_{1}\cap B_{2}=\emptyset$
because $e>M+1$.

\end{proof}

We also have the following extension of Lemma~\ref{lem f_w continuous family} to a countable union of compact sets.

\begin{lem}\label{lem C_j W_M}
We can choose a sequence of constants $\{C_{j}\}$ in the definition of $\mathcal{P}$ (see Definition~\ref{def space perturbation})  such that  for any positive integer $M $ and any $\bar{f}_{1},\bar{f}_{2}\in \mathcal{P}$, the induced family $\{\tilde{f}_{w}\}_{w\in W_{M}}$ is a continuous family of $(6M+6)$-periodic tame functions.
\end{lem}

\begin{proof}
For each $W_M$,  there exists a sequence $\{ C_{M, j} \}_{j}$ such that, for any $f_1, f_2 \in \mathcal{P}( \{ C_{M, j} \}_j)$, the family $\{ \tilde{f}_{w}\}_{w \in W_M}$ is a continuous family of $(6M+6)$-periodic tame functions.
It is straightforward to see that  a sequence of positive real numbers $\{ C_j \}$ such that
\begin{equation*}
         C_j  \geq \max_{ 1 \leq M \leq j } C_{M, j}
\end{equation*}
satisfies our requirement.
\end{proof}

Next is the boundedness result for functionals with mixed perturbations.

\begin{lem}\label{uniform bound for finite trajectory}
For any $C>0$, there exist constants $R,C'$ such that for any $(e,s) \in [1,\infty)\times [0,1]$ and any $\stL_{e}^{s}$-trajectory $\gamma \colon [-2,2]\rightarrow Coul(Y)$ with topological energy $\mathcal{E}^{\text{top}}(\gamma;\stL_{e}^{s})\leq C$,
we have $\gamma(0)\in Str(R)$ and $|\stL_{e}^{s}(\gamma(0))|<C'$.
\end{lem}
\begin{proof}
We first write down $\gradtil {\tilde{f}}^{s}_{e}$ as
\begin{equation*}
\begin{split}
&\gradtil {\tilde{f}}^{s}_{e}(a,\phi)=\delta \phi+(1-s)(\bar{f}_{1}(a,\phi)-\bar{f}_{2}(a,\phi)) \gradtil (\tau_{e}\circ p_{\mathcal{H}})(a,\phi)\\
&+(1-s) (\tau_{e}\circ p_{\mathcal{H}}(a,\phi)) \gradtil\bar{f}_{1}(a,\phi)+(1-(1-s) (\tau_{e}\circ p_{\mathcal{H}}(a,\phi)))\gradtil\bar{f}_{2}(a,\phi).
\end{split}
\end{equation*}
By boundedness and tameness conditions of $\bar{f}_{j}$, we see that
$$\|\grad (\xi^{s}_{e}\circ \Pi) (a,\phi)\|_{L^{2}}=\|\gradtil \xi^{s}_{e} (a,\phi)\|_{\tilde{g}}\leq m(1+\|\phi\|_{L^{2}}),$$ where $m$ is a constant independent of $(e,s)$.  This implies
\begin{equation}\label{norm of perturbation}
    \|\grad (\xi^{s}_{e}\circ \Pi) (a,\phi)\|^{2}_{L^{2}}\leq 2m^{2}+2m^{2}\|\phi\|^{2}_{L^{2}}
\end{equation}

We can lift $\gamma|_{[-1,1]}$ to $\tilde{\gamma}:[-1,1]\rightarrow \mathcal{C}_{Y}$, which is a negative gradient flow line for the functional $\mathcal{L}^{s}_{e}\circ \Pi$. Now we follow the argument on Page 161 of \cite{Kronheimer-Mrowka}. Since $\mathcal{L}^{s}_{e}\circ \Pi=CSD_{\nu_{0}}+\xi^{s}_{e}\circ \Pi$, we have
$$
\|\grad CSD_{\nu_{0}}\|^{2}_{L^{2}}-2\|\grad (\xi_{e}^{s}\circ \Pi)\|^{2}_{L^{2}} \leq 2\|\grad (\mathcal{L}_{e}^{s}\circ \Pi)\|_{L^{2}}^{2}.
$$
By formula (\ref{norm of perturbation}), this implies
\begin{equation}
 \begin{split} \int_{-1}^{1}(\|\grad  CSD_{\nu_{0}}(\tilde{\gamma}(t))\|^{2}_{L^{2}}+\|\tilde{\gamma}'(t)\|^{2}_{L^{2}})dt-2m^2 \int^{1}_{-1}\|\phi(t)\|^{2}_{L^{2}}dt-4m^2\\
 \leq 2 \int_{-1}^{1}(\|\grad  (\mathcal{L}^{s}_{e}\circ \Pi)(\tilde{\gamma}(t))\|^{2}_{L^{2}}+\|\tilde{\gamma}'(t)\|^{2}_{L^{2}})dt<2\mathcal{E}^{top}(\gamma,\mathcal{L}^{e}_{s})\leq 2C.
 \end{split}
 \end{equation}
We can treat $\tilde{\gamma}$ as a section over the 4-manifold $[-1,1]\times Y$ and denote it by $(\hat{a},\hat{\phi})$. By Definition 4.5.4 and formula (4.19) of \cite{Kronheimer-Mrowka}, the above estimate on the analytical energy actually implies
$$
\frac{1}{4}\int_{[-1,1]\times Y}|d\hat{a}|^{2}+ \int_{[-1,1]\times Y}| \nabla_{\hat{A}}\hat{\phi}|^{2}+\frac{1}{4} \int_{[-1,1]\times Y}(|\hat{\phi}|^{2}-C_{2})^{2}\leq C_{3}
$$
where $\hat{A}$ is the connection corresponding to $\hat{a}$ and $C_{2}$ is a constant independent of $e,s$.
By Corollary 4.5.3, Lemma 5.1.2 and Lemma 5.1.3 of \cite{Kronheimer-Mrowka}, we can find a gauge transformation $u:[-1,1]\times Y\rightarrow S^{1}$ such that $\|u\cdot \tilde{\gamma}\|_{L_{1}^{2}([-1,1]\times Y)}$ is bounded by a uniform constant $C_{4}$. Let $u_{t}$ equals $u|_{\{t\}\times Y}$. Then there exists $C_{5}$ such that for any $t_{1},t_{2}\in [-1,1]$, we have
$$
      \| \Pi_{\mathcal{H}} (u_{t_{1}} \cdot  \tilde{\gamma}(t_{1})) -
         \Pi_{\mathcal{H}} (u_{t_{2}} \cdot  \tilde{\gamma}(t_{2})  )\|_{L^{2}}
        \leq
      \|   u_{t_{1}} \cdot \tilde{\gamma}(t_{1}) \|_{L^{2}} +
      \| u_{t_{2}} \cdot \tilde{\gamma}(t_{2}) \|_{L^{2}}
        \leq
       C_{5}
$$
Recall that $\Pi_{\mathcal{H}} : \mathcal{C}_Y \rightarrow i\Omega^{1}_{h}(Y)$ is just the orthogonal projection. Since $u_{t_{1}}$ and $u_{t_{2}}$ are in the same component of the gauge group $\mathcal{G}_{Y}$, we have
\[
             \|p_{\mathcal{H}}( \gamma(t_{1})) -  p_{\mathcal{H}}(\gamma(t_{2})  )\|_{L^{2}}  =
             \| \Pi_{\mathcal{H}} (u_{t_{1}}\cdot \tilde{\gamma}(t_{1})) -
             \Pi_{\mathcal{H}} (u_{t_{2}} \cdot \tilde{\gamma}(t_{2})  )\|_{L^{2}}
             \leq C_{5}.
\]
This implies that $\gamma([-1,1])$ is contained in $p_{\mathcal{H}}^{-1}([e_{1}-M_{0},e_{1}+M_{0}]\times...\times [e_{b}-M_{0},e_{b}+M_{0}] )$ for some $(e_{1},...,e_{b})\in \mathbb{R}^{b}$ and some uniform constant $M_{0}\in \mathbb{Z}_{\geq 1}$. By Lemma \ref{restriction to the cube}, we have $\xi^{s}_{e}|_{p_{\mathcal{H}}^{-1}([e_{1}-M_{0},e_{1}+M_{0}]\times...\times [e_{b}-M_{0},e_{b}+M_{0}])}=f_{w}$ for some $w\in W_{M_{0}}$. This implies that $\gamma|_{[-1,1]}$ is also a trajectory for $CSD_{\nu_{0}}|_{Coul(Y)}+f_{w}$. Notice that $\mathcal{E}^{\text{top}}(\gamma|_{[-1,1]},CSD_{\nu_{0}}|_{Coul(Y)}+f_{w})<C$. Our result is directly implied by Proposition \ref{uniform bound for compact set}.

\end{proof}

The previous results implies uniform boundedness for finite type trajectories for the family $\{\stL^s_e \} $.
For convenience, we will say that functional  $\mathcal{L} \colon Coul(Y)\rightarrow \mathbb{R}$ is called $R$-bounded if any finite type $\mathcal{L}$-trajectory  is contained in $Str(R)$.

\begin{cor}\label{the uniform bound for trajectories}
There exists a uniform constant $R_{1}>0$ such that for any $e\in \mathbb{R}_{\geq 1}$ and $s\in [0,1]$, the functionals $\mathcal{L}_{e}^{s}$ is $R_{1}$-bounded.
\end{cor}

\begin{proof}
Let $\gamma \colon \mathbb{R} \rightarrow Coul(Y) $ be a finite type $\stL^{s}_{e}$-trajectory. Since $\mathcal{E}^{\text{top}}(\gamma,\stL^{s}_{e})<\infty$, we have $\mathcal{E}^{\text{top}}(\gamma|_{[t-1,t+1]},\stL^{s}_{e})<1$ for any $t$ with $\left\vert t \right\vert$ sufficiently large. By Lemma~\ref{uniform bound for finite trajectory} (with $C=1$), we have $|\stL^{s}_{e}(\gamma(t))|\leq C'$ for such $t$. Since $\stL^s_e $ is decreasing along $\gamma $, we see that  $\stL_{e}^{s}(\gamma(t-1))-\stL_{e}^{s}(\gamma(t+1))<2C'$ for any $t\in \mathbb{R}$. We apply Lemma~\ref{uniform bound for finite trajectory} again (now $C=2C')$, so there is a uniform constant $R_{1}$ such that $\gamma(t)\in Str(R_{1})$ for any $t \in \mathbb{R}$.
\end{proof}

%
%
%

For the reader's convenience, we summarize the functionals we will be dealing with.
Two extended cylinder functions $\bar{f}_1  , \bar{f}_2$ are now fixed, along with their corresponding functional $\mathcal{L}_1 , \mathcal{L}_2 $. We have the continuous family of functionals $\{ \stL^s_e \} $ (see (\ref{eq familymixfunc})) such that,
for each $(e,s) \in [1,\infty) \times [0,1]  $, they satisfy
\[
   \begin{split}
 \stL^{1}_{e} &= {\mathcal{L}}_{2}, \\
   \stL_{e}^{0}(x) &=
      \left\{
     \begin{array}{ll}
    {\mathcal{L}}_{1}(x) &  \text{if }x \in p_{\mathcal{H}}^{-1}([-e+1,e-1]^{b_1}),    \\
    {\mathcal{L}}_{2}(x) &  \text{if }x \in p_{\mathcal{H}}^{-1}(\mathbb{R}^{b_1}\setminus (-e,e)^{b_1}),
    \end{array}
    \right.    \\
   \stL^{s}_{e}(x)  &=  {\mathcal{L}}_2(x)  \ \text{if $x \in p_{\mathcal{H}}^{-1}( \mathbb{R}^{b_1} \setminus (-e, e)^{b_1} )$. }
    \end{split}
\]
Since the above construction is asymmetrical in $\bar{f}_1 $ and $\bar{f}_2 $, we also consider another family of functionals
$\{\tilde{\mathcal{L}}^{s}_{e} \}$ where the role of $\bar{f}_1 $ and $\bar{f}_2 $ are reversed. In other words, we have
\[
  \begin{split}
  \tilde{\stL}^{1}_{e} &= {\mathcal{L}}_{1},  \\
        \tilde{\stL}_{e}^{0}(x)  &=
   \left\{
       \begin{array}{ll}
          {\mathcal{L}}_{2}(x) &  \text{if  $x  \in p_{\mathcal{H}}^{-1}([-e+1,e-1]^{b})$},   \\
          {\mathcal{L}}_{1}(x) &  \text{if  $x  \in p_{\mathcal{H}}^{-1}(\mathbb{R}^{b}\setminus (-e,e)^{b})$},
       \end{array}
    \right. \\
    \tilde{\stL}^{s}_{e}(x) & = {\mathcal{L}}_1(x)
                    \text{ if $x \in p_{\mathcal{H}}^{-1}( \mathbb{R}^b \setminus (-e, e)^b )$.  }
  \end{split}
\]
Roughly speaking, the family $\{ \stL^s_e \}  $ will give a morphism from Conley indices given by $\mathcal{L}_1 $ to Conley indices given by $\mathcal{L}_2 $ and vice versa.

To show equivalence, we need to introduce (final) two more families of functionals. For two real numbers $e, e' $
with $e-1\geq e'\geq 1$ and $s\in [0,1]$, we define
\[
    \begin{split}
        \stL_{e,e'}^{s}(x)
      &=
    \left\{
    \begin{array}{l l}
    \tilde{\stL}_{e'}^{s}(x) & \quad \text{if }x\in p_{\mathcal{H}}^{-1}([-e',e']^{b_1})\\
    \stL_{e}^{0}(x) & \quad \text{otherwise,}
    \end{array} \right. \\
    \tilde{\stL}_{e,e'}^{s}(x)
     &=
     \left\{
     \begin{array}{l l}
     \stL_{e'}^{s}(x) & \quad \text{if }x\in p_{\mathcal{H}}^{-1}([-e',e']^{b_1})\\
     \tilde{\stL}_{e}^{0}(x) & \quad \text{otherwise.}
\end{array} \right.
\end{split}
\]
These functionals have the following properties:
\begin{enumerate}
\item $\stL^{1}_{e,e'}=\stL^{0}_{e}$ and $\tilde{\stL}^{1}_{e,e'}=\tilde{\stL}^{0}_{e}$.
\item $\stL_{e,e'}^{0}(x) = \left\{
  \begin{array}{l l}
    {\mathcal{L}}_{2}(x)
           & \quad \text{if } x \in p_{\mathcal{H}}^{-1}([-e'+1,e'-1]^{b_1}\cup (\mathbb{R}^{b_1}\setminus (-e,e)^{b_1})),\\
    {\mathcal{L}}_{1}(x)
           & \quad \text{if } x \in p_{\mathcal{H}}^{-1}([-e+1,e-1]^{b_1}\setminus (-e',e')^{b_1}).
  \end{array} \right.   $

\item
$\tilde{\stL}_{e,e'}^{0}(x)=\left\{
  \begin{array}{l l}
    {\mathcal{L}}_{1}(x)
         & \quad \text{if }x\in p_{\mathcal{H}}^{-1}([-e'+1,e'-1]^{b_1}\cup (\mathbb{R}^{b_1}\setminus (-e,e)^{b_1})),   \\
    {\mathcal{L}}_{2}(x)
            & \quad \text{if }x\in p_{\mathcal{H}}^{-1}([-e+1,e-1]^{b_1}\setminus (-e',e')^{b_1})).
  \end{array} \right. $

\end{enumerate}

We have the following extension of Lemma \ref{restriction to the cube} and \ref{lem C_j W_M}. The proof is essentially the same and we omit it.
\begin{lem}\label{double mixed tame functions}
(1) For each positive integer $M$, we can find a smooth family of $(6M+6)\mathbb{Z}^{b_{1}}$-periodic functions $\tilde{\tau}_{w}:\mathbb{R}^{b_{1}}\rightarrow [0,1]$, parametrized by a compact manifold $W'_{M}$, with the following property: for any functional in the family $\{\stL^{s}_{e,e'} \mid s\in [0,1],e-1\geq e'\geq 1\}$ and any $(e_{1},...,e_{b_{1}})\in \mathbb{R}^{b_{1}}$, we can find $w\in W_{M}'$ such that $$\stL^{s}_{e,e'}=CSD_{\nu}|_{Coul(Y)}+f_{w}$$ when restricted to $p_{\mathcal{H}}^{-1}([e_{1}-M,e_{1}+M]\times ...[e_{b_{1}}-M,e_{b_{1}}+M])$. Here
$f_{w}$ is the function on $Coul(Y)$ induced by $\tilde{\tau}_{w}$ (see (\ref{mixed tame functions})).
\\(2) We can choose a sequence of constants $\{C_{j}\}$ in the definition of $\mathcal{P}$ (see Definition~\ref{def space perturbation})  such that  for any positive integer $M $ and any $\bar{f}_{1},\bar{f}_{2}\in \mathcal{P}$, the induced family $\{\tilde{f}_{w}\}_{w\in W'_{M}}$ is a continuous family of $(6M+6)$-periodic tame functions.
\\(3) Similar result holds if we consider any one of the following families instead
\begin{itemize}
\item $\{\tilde{\stL}^{s}_{e,e'} \mid s\in [0,1],e-1\geq e'\geq 1\}$;
\item $\{ (1-s){\mathcal{L}}_{2}+s\stL^{s'}_{e,e'} \mid s,s'\in [0,1],e-1\geq e'+1,\}$;
\item $\{(1-s){\mathcal{L}}_{1}+s\tilde{\stL}^{s'}_{e,e'} \mid s,s'\in [0,1],e\geq e'+1,\}$.
\end{itemize}
\end{lem} With Lemma \ref{double mixed tame functions} at hand, we can repeat the proof of Corollary~\ref{the uniform bound for trajectories} and get the following result.
\begin{lem}\label{double bump functions}
There exists a uniform $R_{2}$ such that for $e-1\geq e'\geq 1$ and $s,s'\in [0,1]$, the functionals $\stL^{s}_{e,e'},\tilde{\stL}^{s}_{e,e'},
(1-s){\mathcal{L}}_{2}+s\stL^{s'}_{e,e'}$ and $(1-s){\mathcal{L}}_{1}+s\tilde{\stL}^{s'}_{e,e'}$ are
all $R_{2}$-bounded.
\end{lem}

Now we start constructing a mixed direct system relating the spectra given by ${\mathcal{L}}_{1}$ and ${\mathcal{L}}_{2}$. As usual, we focus on the case of $\underline{\text{swf}}^{A}$.
We first choose a constant $\tilde{R}$  greater than $\max(R_{1},R_{2})$, where $R_{1}$ is the constant in Corollary~\ref{the uniform bound for trajectories} and $R_{2}$ is a constant that we will specify later in Lemma~\ref{double bump functions}. Let $J^{+}_{1}\subset J^{+}_{2}\subset \ldots$ and $\tilde{J}^{+}_{1}\subset \tilde{J}^{+}_{2}\subset \ldots$ be increasing sequences of bounded subsets corresponding to ${\mathcal{L}}_{1}$ and ${\mathcal{L}}_{2}$ respectively (see (\ref{cutting
set for attractor})). Although these bounded sets come from $ Str(\tilde{R})$, they are different as we use different cutting functions and different cutting values.
Since the sequences of subsets are increasing, we can find increasing sequences of positive integers $\{m_j\}$, $\{ \tilde{m}_j\}$, $ \{ e_j \}$ and $\{ \tilde{e}_j \}$ such that
\begin{equation}\label{inclusion relations}
\begin{split}
      &J^{+}_{m_{j}}\subset
         p_{\mathcal{H}}^{-1}([-e_{j}+1,e_{j}-1]^{b})\cap Str(\tilde{R})       \subset
          p_{\mathcal{H}}^{-1}([-e_{j},e_{j}]^{b})\cap Str(\tilde{R})              \subset
          \tilde{J}^{+}_{\tilde{m}_{j}}\\
&       \subset p_{\mathcal{H}}^{-1}([-\tilde{e}_{j}+1,\tilde{e}_{j}-1]^{b})  \cap Str(\tilde{R})            \subset
         p_{\mathcal{H}}^{-1}([-\tilde{e}_{j},\tilde{e}_{j}]^{b})    \cap    Str(\tilde{R})                           \subset
          J^{+}_{m_{j+1}}.
\end{split}
\end{equation}

Let $\{ \mu_n \}$ and $\{ \lambda_n\}$ be an increasing sequence and a decreasing sequence of real numbers with $-\lambda_n , \mu_n \rightarrow \infty $ and denote by $V^{\mu_{n}}_{\lambda_{n}} $ the corresponding eigenspace. For a functional $\mathcal{L} $ on $Coul(Y) $, we denote by $\varphi^n (\mathcal{L}) $ the flow generated by $\iota \circ p^{\mu_{n}}_{\lambda_{n}} \gradtil \mathcal{L} $ on $V^{\mu_{n}}_{\lambda_{n}} $ where $\iota$ is a bump function with value $1$ on a specific bounded set. Since we are only interested in the Conley index which will be independent of $\iota$, we can drop $\iota$ from our notation.


Consider $J^{n,+}_{m_{j}} = J^{+}_{m_{j}} \cap V^{\mu_{n}}_{\lambda_{n}} $ and $\tilde{J}^{n,+}_{m_{j}} = \tilde{J}^{+}_{m_{j}} \cap V^{\mu_{n}}_{\lambda_{n}} $. By Theorem~\ref{finite dimensional approximation for a family of flows}, we can fix a sufficiently large integer $n$ so that $J^{n,+}_{m_{j}} , \tilde{J}^{n,+}_{m_{j}}$ are isolating neighborhoods for all of the above families of approximated flows. For the family $\{ \stL^s_{e_j} \} $, we get a homotopy
equivalence from homotopy invariance of Conley indices
$$
\rho_{1} \colon I_{S^{1}}(\varphi^{n}(\stL^{0}_{e_{j}}),\text{inv}(\tilde{J}_{\tilde{m}_{j}}^{n,+}))
\xrightarrow\cong  I_{S^{1}}(\varphi^{n}(\mathcal{L}_{2}),\text{inv}(\tilde{J}_{\tilde{m}_{j}}^{n,+})),
$$
where we recall that  $\stL^{1}_{e_{j}}= {\mathcal{L}}_{2}$.
Since $\stL^{0}_{e_{j}}$ is equal to ${\mathcal{L}}_{1}$ on $p_{\mathcal{H}}^{-1}([-e_{j}+1,e_{j}-1]^{b_1})$,
which contains $J^{+}_{m_{j}}$, we see that the flow $\varphi^n(\stL^0_{e_j}) $ goes inside $J^{n,+}_{m_{j}}$ along the boundary $\partial J^{n,+}_{m_{j}}\setminus
\partial Str_{Y}(\tilde{R})$. Consequently, the subset $ J^{n,+}_{m_{j}} \subset \tilde{J}^{n,+}_{\tilde{m}_{1}} $ is an attractor with respect to $\varphi^n(\stL^0_{e_j}) $ and we obtain an attractor map
$$
\rho_{2} \colon I_{S^{1}}(\varphi^{n}({\mathcal{L}}_{1}),\text{inv}(J_{m_{j}}^{n,+}))\rightarrow
I_{S^{1}}(\varphi^{n}(\stL^{0}_{e_{j}}),\text{inv}(\tilde{J}_{\tilde{m}_{j}}^{n,+})).
$$
We combine the above two maps and obtain the following map
\begin{equation}\label{attracor map 1 in the mixed system}
\bar{i}^{n,+}_{m_{j}} :=\rho_{1}\circ \rho_{2} \colon I_{S^{1}}(\varphi^{n}({\mathcal{L}}_{1}),\text{inv}(J^{n,+}_{m_{j}}))\rightarrow I_{S^{1}}(\varphi^{n}({\mathcal{L}}_{2}),\text{inv}(\tilde{J}^{n,+}_{\tilde{m}_{j}})).
\end{equation}

Similarly, we use the family $\{ \tilde{\stL}^{s}_{\tilde{e}_{j}} \}$  to get a homotopy equivalence
$$
\tilde{\rho}_{1} \colon I_{S^{1}}(\varphi^{n}(\tilde{\stL}^{0}_{\tilde{e}_{j}}),\text{inv}(J_{m_{j+1}}^{n,+}))
\xrightarrow\cong  I_{S^{1}}(\varphi^{n}({\mathcal{L}}_{1}),\text{inv}(J_{m_{j+1}}^{n,+})).$$
Since $\tilde{J}_{\tilde{m}_{j}}^{n,+} \subset J_{m_{j+1}}^{n,+}$ is an attractor with respect to $\tilde{\stL}^{0}_{\tilde{e}_{j}} $, we also get an attractor map
$$
\tilde{\rho_{2}} \colon I_{S^{1}}(\varphi^{n}({\mathcal{L}}_{2}),\text{inv}(\tilde{J}_{\tilde{m}_{j}}^{n,+}))\rightarrow I_{S^{1}}(\varphi^{n}(\tilde{\stL}^{0}_{\tilde{e}_{j}}),\text{inv}(J_{m_{j+1}}^{n,+})).
$$
We compose the above two maps and get the following map
\begin{equation}\label{attractor map 2 in the mixed system}
\hat{i}^{n,+}_{\tilde{m}_{j}}:=\tilde{\rho}_{1}\circ \tilde{\rho}_{2} \colon I_{S^{1}}(\varphi^{n}({\mathcal{L}}_{2}),\text{inv}(\tilde{J}^{n,+}_{\tilde{m}_{j}}))\rightarrow I_{S^{1}}(\varphi^{n}({\mathcal{L}}_{1}),\text{inv}(J^{n,+}_{m_{j+1}})).
\end{equation}
After appropriate desuspension, we obtain a direct system in the category $\mathfrak{C}$
\begin{equation}\label{mixed direct system}
I^{n_{1},+}_{m_{1}}\rightarrow \tilde{I}^{\tilde{n}_{1},+}_{\tilde{m}_{1}}\rightarrow I^{n_{2},+}_{m_{2}}\rightarrow \tilde{I}^{\tilde{n}_{2},+}_{\tilde{m}_{2}}\rightarrow \cdots,
\end{equation}
where $I^{n,+}_{m}$ (resp. $\tilde{I}^{n,+}_{m}$) be the object of $\mathfrak{C}$ obtained from desuspending
the {Conley} indices of $J^{n,+}_{m}$ (resp. $\tilde{J}^{n,+}_{m}$) by $\bar{V}^{0}_{-\lambda_{n}}$ and  we can pick a suitable sequence of integers $0\ll n_{1} < \tilde{n}_{1} <n_{2} <\tilde{n}_{2} < \cdots$.
The main result of this section follows from the following proposition.

\begin{pro}\label{subsystem of mixed system}
The map $\hat{i}^{n,+}_{\tilde{m}_{j}}\circ \bar{i}^{n,+}_{m_{j}}$ is $S^{1}$-homotopic to attractor map for the attractor $\text{inv}(\varphi^{n}(\mathcal{L}_{1}),J^{n,+}_{m_{j}})\subset \text{inv}(\varphi^{n}(\mathcal{L}_{1}),J^{n,+}_{m_{j+1}})$.
\end{pro}

\begin{proof} We consider the following commutative (up to $S^{1}$-homotopy) diagram.
\\
\\
\begindc{\commdiag}[500]
\obj(0,5)[aa]{$I_{S^{1}}(\varphi^{n}(\mathcal{L}_{1}),\text{inv}(J^{n,+}_{m_{j}}))$}
\obj(0,4)[ab]{$I_{S^{1}}(\varphi^{n}(\stL^{0}_{e_{j}}),\text{inv}(J^{n,+}_{m_{j}}))$}
\obj(3,4)[bb]{$I_{S^{1}}(\varphi^{n}(\stL^{0}_{e_{j}}),\text{inv}(\tilde{J}^{n,+}_{\tilde{m}_{j}}))$}
\obj(6,4)[cb]{$I_{S^{1}}(\varphi^{n}(\stL^{1}_{e_{j}}),\text{inv}(\tilde{J}^{n,+}_{\tilde{m}_{j}}))$}
\obj(0,3)[ac]{$I_{S^{1}}(\varphi^{n}(\tilde{\stL}^{0}_{\tilde{e}_{j},e_{j}}),\text{inv}(J^{n,+}_{m_{j}}))$}
\obj(3,3)[bc]{$I_{S^{1}}(\varphi^{n}(\tilde{\stL}^{0}_{\tilde{e}_{j},e_{j}}),\text{inv}(\tilde{J}^{n,+}_{\tilde{m}_{j}}))$}
\obj(6,3)[cc]{$I_{S^{1}}(\varphi^{n}(\mathcal{L}_{2}),\text{inv}(\tilde{J}^{n,+}_{\tilde{m}_{j}}))$}
\obj(0,2)[ad]{$I_{S^{1}}(\varphi^{n}(\tilde{\stL}^{0}_{\tilde{e}_{j},e_{j}}),\text{inv}(J^{n,+}_{m_{j+1}}))$}
\obj(3,2)[bd]{$I_{S^{1}}(\varphi^{n}(\tilde{\stL}^{0}_{\tilde{e}_{j}}),\text{inv}(J^{n,+}_{m_{j+1}}))$}
\obj(6,2)[cd]{$I_{S^{1}}(\varphi^{n}(\tilde{\stL}^{0}_{\tilde{e}_{j}}),\text{inv}(\tilde{J}^{n,+}_{\tilde{m}_{j}}))$}
\obj(3,1)[be]{$I_{S^{1}}(\varphi^{n}(\mathcal{L}_{1}),\text{inv}(J^{n,+}_{m_{j+1}}))$}
\obj(6,1)[bf]{$I_{S^{1}}(\varphi^{n}(\tilde{\stL}^{1}_{\tilde{e}_{j}}),\text{inv}(J^{n,+}_{m_{j+1}}))$}\mor(1,4)(2,4){$\rho_{2}$}
\mor(4,4)(5,4){$\rho_{1}$}
\mor(1,3)(2,3){$\rho_{6}$}
\mor(5,2)(4,3){$\rho_{3}$}
\mor(1,2)(2,2){$\rho_{5}$}
\mor(5,2)(4,2){$\tilde{\rho}_{2}$}
\mor(4,2)(5,1){$\tilde{\rho}_{1}$}
\mor(0,3)(0,2){$\rho_{7}$}
\mor(2,3)(1,2){$\rho_{4}$}
\mor(1,2)(2,1){$\rho_{8}$}
\mor{be}{bf}{}[+1,9]
\mor(0,5)(0,4){}[+1,9]
\mor(0,4)(0,3){}[+1,9]
\mor(3,4)(3,3){}[+1,9]
\mor(6,4)(6,3){}[+1,9]
\mor(6,3)(6,2){}[+1,9]
\enddc

The maps are defined as follows.
\begin{enumerate}
\item Different flows are generated by the same vector field when restricted to some isolating neighborhood. This defines all the identifications ``$=$'' in the diagram.

\item The maps $\rho_{1},\rho_{2},\tilde{\rho}_{1},\tilde{\rho}_{2}$ are defined as before.

\item The maps $\rho_{3},\rho_{5}$ are the homotopy equivalences given by the deformation $\tilde{\stL}^{s}_{\tilde{e}_{j},e_{j}}$, $s\in [0,1]$.

\item The maps $\rho_{4},\rho_{6},\rho_{7}$  are the attractor maps for the flow $\varphi^{n}(\tilde{\stL}^{0}_{\tilde{e}_{j},e_{j}})$.

\item The map $\rho_{8}$ is homotopy equivalence given by the deformation
\begin{equation}\label{deformation1}
(1-s)\mathcal{L}_{1}+s\tilde{\stL}^{0}_{\tilde{e}_{j},e_{j}}, \text{ }s\in [0,1].
\end{equation}

\end{enumerate}
Now we check that the above diagram commutes:
\begin{enumerate}
\item The maps $\rho_{2}$ and $\rho_{6}$  are the defined as attractor maps for the flows $\varphi^{n}(\stL^{0}_{e_{j}})$ and $\varphi^{n}(\tilde{\stL}^{0}_{\tilde{e}_{j},e_{j}})$ respectively. Since these two flows are generated by the same vector field when restricted to  $\tilde{J}^{n,+}_{\tilde{m}_{j}}$, we see that $\rho_{2}$ is $S^{1}$-homotopic to $\rho_{6}$, written as $\rho_{2}\cong \rho_{6}$.
\item Because the attractor maps for the same flow are transitive, we have  $\rho_{7}\cong\rho_{4}\circ \rho_{6}$.
\item We deform $\tilde{\stL}^{0}_{\tilde{e}_{j}}= \tilde{\stL}^{1}_{\tilde{e}_{j},e_{j}}$ to $\tilde{\stL}^{0}_{\tilde{e}_{j},e_{j}}$ through the family $\tilde{\stL}^{s}_{\tilde{e}_{j},e_{j}}$. In the process of this deformation, nothing is changed on the set $p_{\mathcal{H}}^{-1}(\mathbb{R}^{b}\setminus (-e_{j},e_{j})^{b})$, which contains both  $\partial J^{n,+}_{m_{j+1}}\setminus \partial Str(\tilde{R})$ and $\partial \tilde{J}^{n,+}_{\tilde{m}_{j}}\setminus \partial Str(\tilde{R})$. Therefore, we obtain a family of attractor maps: we get $\rho_{4}$ when $s=0$ and get $\tilde{\rho}_{2}$ when $s=1$. Notice that $\rho_{3}$ and $\rho_{5}$ are the homotopy equivalences induced by this deformation. The identity $\tilde{\rho}_{2}\cong \rho_{5}\circ \rho_{4}\circ \rho_{3}$ can be proved using the homotopy invariance of the attractor map.
\item The map $\rho_{3}$ is induced by the deformation $\tilde{\stL}^{s}_{\tilde{e}_{j},e_{j}}$ with $s$ going from $1$ to $0$. We just get $\stL^{s}_{e_{j}}$ if we restrict this deformation to the set $\tilde{J}^{n,+}_{\tilde{m}_{j}}$. Therefore, we have $\rho_{1}\cong\rho_{3}^{-1}$.
\item Notice that $\tilde{\rho}_{1}\circ \rho_{5}$ is the homotopy equivalence induced by the following deformation:
\begin{equation}\label{deformation2}
\tilde{\stL}^{0}_{\tilde{e}_{j},e_{j}}|_{J^{+}_{m_{j+1}}}\rightarrow \tilde{\stL}^{1}_{\tilde{e}_{j},e_{j}}|_{J^{+}_{m_{j+1}}}=\tilde{\stL}^{0}_{\tilde{e}_{j}}|_{J^{+}_{m_{j+1}}}\rightarrow \tilde{\stL}^{1}_{\tilde{e}_{j}}|_{J^{+}_{m_{j+1}}}=\mathcal{L}_{1}|_{J^{+}_{m_{j+1}}}.
\end{equation}
In order to prove the identity $\rho_{8} { \cong}  \tilde{\rho}_{1}\circ \rho_{5}$, we just need to show that the homotopy equivalences $I_{S^{1}}(\varphi^{n}(\tilde{\stL}^{0}_{\tilde{e}_{j},e_{j}}),\text{inv}(J^{n,+}_{m_{j+1}})) \xrightarrow{\cong }I_{S^{1}}(\varphi^{n}(\mathcal{L}_{1}),\text{inv}(J^{n,+}_{m_{j+1}}))$ which are induced by  deformations (\ref{deformation1}) and (\ref{deformation2}) are $S^{1}$-homotopic to each other. To see this, for any $r\in[0,1]$, we consider the following 2-step deformation.

\begin{enumerate}

\item
First deform $\tilde{\stL}^{0}_{\tilde{e}_{j},e_{j}}$ to $r\mathcal{L}_{1}+(1-r)\tilde{\stL}^{1}_{\tilde{e}_{j},e_{j}}=\tilde{\stL}^{r}_{\tilde{e}_{j}}$ through the family $rs\mathcal{L}_{1}+(1-rs)\tilde{\stL}^{s}_{\tilde{e}_{j},e_{j}}$, with $s$ going from $0$ to $1$.

\item
Then deform $\tilde{\stL}_{\tilde{e}_{j}}^{r}$ to $\tilde{\stL}_{\tilde{e}_{j}}^{1}=\mathcal{L}_{1}$ through the family $\tilde{\stL}_{\tilde{e}_{j}}^{s}$, with $s$ going from $r$ to $1$.

\end{enumerate}
Setting $r$ to be $0$ and $1$ in the above deformation, we will get (\ref{deformation2}) and  (\ref{deformation1}) respectively. As before, the flow near $\partial J^{+}_{m_{j+1}}\setminus \partial Str_{Y}(\tilde{R})$ is not changed. By Lemma \ref{double bump functions}, all the functionals involved in the above deformation are $R_{2}$-bounded. Since $\tilde{R}>R_{2}$, $J^{n,+}_{m_{j+1}}$ is an isolating neighborhood for all these functionals when $n$ is large enough. Therefore, as $r$ goes from $0$ to $1$, we get a $S^{1}$-homotopy between the homotopy equivalences induced by (\ref{deformation1}) and (\ref{deformation2}).
\end{enumerate}
We have proved that the diagram is commutative up to $S^{1}$-homotopy. As a corollary, the map $\hat{i}^{n,+}_{\tilde{m}_{j}}\circ \bar{i}^{n,+}_{m_{j}}=\tilde{\rho}_{1}\circ \tilde{\rho}_{2}\circ \rho_{1}\circ \rho_{2}$ is $S^{1}$-homotopic to $\rho_{8}\circ\rho_{7}$.
Now we consider the attractor map for the flow $\mathcal{L}_{1}$, which we denote by
$$
i^{+}:I_{S^{1}}(\varphi^{n}(\mathcal{L}_{1}),\text{inv}(J^{n,+}_{m_{j}}))\rightarrow I_{S^{1}}(\varphi^{n}(\mathcal{L}_{1}),\text{inv}(J^{n,+}_{m_{j+1}})).
$$
We deform $\mathcal{L}_{1}$ to $\tilde{\stL}^{0}_{\tilde{e}_{j},e_{j}}$ through the family $(1-s)\mathcal{L}_{1}+s\tilde{\stL}^{0}_{\tilde{e}_{j},e_{j}}$ $(s \in [0,1])$. Notice that for any $s$, $(1-s)\mathcal{L}_{1}+s\tilde{\stL}^{0}_{\tilde{e}_{j},e_{j}}$ equals $\mathcal{L}_{1}$ on the set $p_{\mathcal{H}}^{-1}([-e_{j}+1,e_{j}-1]^{b}\cup (\mathbb{R}^{b}\setminus (-\tilde{e}_{j},\tilde{e}_{j})^{b}))$, which contains both $\partial J^{+}_{m_{j}}\setminus \partial Str(\tilde{R})$ and $\partial J^{+}_{m_{j+1}}\setminus \partial Str(\tilde{R})$. Therefore, we get a family of attractors:
$$
\text{inv}(\varphi^{n}((1-s)\mathcal{L}_{1}+s\tilde{\stL}^{0}_{\tilde{e}_{j},e_{j}}),J^{n,+}_{m_{j}})\subset \text{inv}(\varphi^{n}((1-s)\mathcal{L}_{1}+s\tilde{\stL}^{0}_{\tilde{e}_{j},e_{j}}),J^{n,+}_{m_{j+1}}).
$$
By the homotopy invariance of the attractor maps, we see that $i^{+}$ also is homotopic to $\rho_{8}\circ \rho_{7}$. This finish the proof of the proposition.
\end{proof}

Proposition~\ref{subsystem of mixed system} actually implies that the direct system (\ref{mixed direct system}) contains a subsystem whose colimit gives the ind-spectrum $\underline{\text{swf}}^{A}$ for the perturbation $f_{1}$. Similarly, we can prove that the ind-spectrum for the perturbation $f_{2}$ is also a subsequential colimit of (\ref{mixed direct system}). Therefore, by Lemma~\ref{subsyemtem}, we see that $f_{1}$ and $f_{2}$ gives the same ind-spectrum up to canonical isomorphism.

Finally, we address the situation when $f_{1}(a,\phi)=\frac{\delta_{1}}{2}\|\phi\|_{L^{2}}^{2}+\bar{f}_{1}(a,\phi)$ and $f_{2}(a,\phi)=\frac{\delta_{2}}{2}\|\phi\|_{L^{2}}^{2}+\bar{f}_{2}(a,\phi)$ with $\delta_{1}\neq \delta_{2}$. This can now be proved the standard homotopy invariance argument as follows.
We set $\delta_{t}=(2-t)\delta_{1}+(t-1)\delta_{2}$.  For each $t_{0} \in [1,2]$, we can find an extended cylinder function $\bar{f}$ such that the pair $(\delta_{t_{}},\bar{f})$ gives a perturbed Chern-Simons-Dirac functional whose critical points are all nondegenerate in the sense of \cite[Definition~12.1.1]{Kronheimer-Mrowka} for any $t$ near $t_0$. Here we essentially use the compactness result for critical points, which is a special case of \cite[Proposition~11.6.4]{Kronheimer-Mrowka}. Hence, we can find a subdivision $1=t_{1} < \cdots < t_{n}=2$ and $\bar{f}_{1}', \dots ,\bar{f}'_{n-1}\in \mathcal{P}$ with $\bar{f}_{1}' = \bar{f}_{1} $ and $\bar{f}_{n-1}' = \bar{f}_2 $ such that the pair $(\delta_{t},\bar{f}'_{j})$ gives a good perturbation for any $t\in [t_{j},t_{j+1}]$. By homotopy invariance of the Conley index, we see that $(\delta_{t_{j}},\bar{f}'_{j})$ and $(\delta_{t_{j+1}},\bar{f}'_{j})$ give the same ind-spectrum $\underline{\text{swf}}^{A}$. Since we already showed that the ind-spectrum does not depend on the choice of the extended cylinder function when $\delta$ is fixed, we can conclude that $f_{1}$ and $f_{2}$ give the same $\underline{\text{swf}}^{A}$ (up to canonical isomorphisms). This finishes the proof of the invariance for (II).
\subsection{The invariance for (I)}\label{subsection invariance I}
Now we discuss what happens when we vary the metric $g$ and the base connection $A_{0}$. Let $(A^{}_{0},g_{0})$ and $(A^{}_{1},g_{1})$ be two pairs of base connections and metrics. We can connect them by a smooth path $\alpha(s)=(A_{s}^{},g_{s})$ with $s\in [0,1]$. As in the proof of the invariance for $\delta$, we can divide $[0,1]$ into small subintervals $[s_{j},s_{j+1}]$ such that, for each subinterval $[s_{j},s_{j+1}]$, we can fix the choice of the auxiliary data $(f,\bar{g},\theta,\tilde{R},\{\lambda_{n}\}, \{ \mu_n \})$.

As $s$ varies between $s_{j}$ and $s_{j+1}$, we get a continuous family of Coulomb slices $Coul(Y,s)$ and a family of sequences of bounded sets
$$
J^{+}_{1,s}\subset J^{+}_{2,s}\subset \cdots.
$$
For any positive integer $n$, we have a (usually not continuous) family of finite-dimensional spaces $V^{\mu_{n}}_{\lambda_{n}}(s)$. As before, we denote by $\bar{V}^{0}_{\mu_{n}}(s)$ the orthogonal complement of $i\Omega^{1}_{h}(Y)$ in $V^{0}_{\mu_{n}}(s)$. Let $J^{n,+}_{m,s}=J^{+}_{m,s}\cap V^{\mu_{n}}_{\lambda_{n}}(s)$ and $\varphi_{n,s}$ be the approximated Seiberg-Witten flow on $V^{\mu_{n}}_{\lambda_{n}}(s)$. The following lemma is a direct consequence of Theorem~\ref{finite dimensional approximation for a family of flows}.
\begin{lem}
For any positive integer $m$ and a sufficiently large integer $n$ relative to $m$, we have
\begin{equation*}
\Sigma^{-\bar{V}^{0}_{\mu_{n}}(s_{j})}I_{S^{1}}(\varphi_{n,s_{j}},\operatorname{inv}(J^{n,+}_{m,s_{j}}))
\cong
\Sigma^{\spf(-\slashed{D},\alpha[s_{j},s_{j+1}])\mathbb{C}}\Sigma^{-\bar{V}^{0}_{\mu_{n}}(s_{j+1})}I_{S^{1}}(\varphi_{n,s_{j+1}},\operatorname{inv}(J^{n,+}_{m,s_{j+1}}))
\end{equation*}
as objects of $\mathfrak{C}$.
\end{lem}

The above isomorphisms induce an isomorphism
$$
\underline{\text{swf}}^{A}(Y,\mathfrak{s},A_{s_{j}},g_{s_{j}};S^{1})\cong\Sigma^{\spf(-\slashed{D},\alpha([s_{j},s_{j+1}]))\mathbb{C}}\underline{\text{swf}}^{A}(Y,\mathfrak{s},A_{s_{j+1}},g_{s_{j+1}};S^{1}).
$$
 By additivity of spectral flow, we can conclude that
 \begin{equation}
  \label{change by suspension (attractor)}
\underline{\text{swf}}^{A}(Y,\mathfrak{s}_{Y},A_{0}^{},g_{0};S^{1})
\cong
\Sigma^{\spf(-\slashed{D},\alpha)\mathbb{C}}\underline{\text{swf}}^{A}(Y,\mathfrak{s}_{Y},A^{}_{1},g_{1};S^{1}).
\end{equation}
Therefore, $\underline{\text{swf}}^{A}$ can only change by suspension or desuspension of copies of $\mathbb{C}$ when we vary the pair $(A_{0},g_0)$. Now  we discuss the following two cases separately.

\textbf{(1) $\mathfrak{s}$ is torsion}:   In this case, we recall that there is a well defined quantity $n(Y,\mathfrak{s},A_{0},g)$. By excision argument as in \cite{Manolescu1}, we have
$$
n(Y,\mathfrak{s},A^{}_{0},g_{0}) + \spf(-\slashed{D},\alpha)=n(Y,\mathfrak{s},A^{}_{1},g_{1}).
$$
This implies
\[
     \begin{split}
&
(\underline{\text{swf}}^{A}(Y, \mathfrak{s}, A_{0}^{}, g_{0}; S^{1}), 0, n(Y,\mathfrak{s},A^{}_{0},g_{0}))
\cong
(\underline{\text{swf}}^{A}(Y,\mathfrak{s},A_{1}^{}, g_{1}; S^{1}), 0, n(Y,\mathfrak{s},A^{}_{1},g_{1}))
    \end{split}
\]
and the same result holds for $\underline{\text{swf}}^{R}$.
This finishes the proof of invariance of  $\underline{\text{SWF}}^{A}(Y,\mathfrak{s};S^{1})$ and $\underline{\text{SWF}}^{R}(Y,\mathfrak{s};S^{1})$ in the torsion case.

\textbf{(2) $\mathfrak{s}$ is non-torsion}: In this case, let $l=\text{g.c.d}\{(c_{1}(\mathfrak{s})\cup h)[Y] \mid h\in H^{1}(Y;\mathbb{Z})\}$. We pick a  harmonic gauge  transformation $u_{0}\in \mathcal{G}^{h,o}_{Y}=H^{1}(Y;\mathbb{Z})$ such that $(c_{1}(\mathfrak{s})\cup [u_{0}])[Y]=l$ and denote by $Coul(Y,A_{0})$ and $Coul(Y,u_{0}(A_{0}))$ the Coulomb slices with the base connections $A_{0}$ and $u_{0}(A_{0})=A_{0}-u_{0}^{-1}du_{0}$ respectively. (Actually, these two slices correspond the same subspace of $\mathcal{C}_{Y}$. However, since the base connections are different, this subspace is identified with $L^{2}_{k}(i \ker d^* \oplus \Gamma(S))$ in different ways. For this reason, we distinguish them for clarity.) The gauge transformation $u_{0}:Coul(Y,A_{0})\rightarrow Coul(Y,u_{0}(A_{0}))$ preserves  the functional $CSD_{\nu_{0},f}$, its formal gradient, the subspace $i\Omega^{1}_{h}(Y)$ , the finite dimensional subspaces $V^{\mu_{n}}_{\lambda_{n}}$ and both the $L^{2}$-metric and the non-linear metric $\|\cdot\|_{\tilde{g}}$. From this fact, we get a natural isomorphism
\begin{equation}
\label{eq gaugeequivswf}
     \begin{split}
      &   \underline{\text{swf}}^{A}(Y,\mathfrak{s},A_{0}, g; S^{1})
           \cong
            \underline{\text{swf}}^{A}(Y,\mathfrak{s},u_{0}(A_{0}), g; S^{1}).
   \end{split}
\end{equation}

Let $\alpha$ be any path going from $A_{0}$ to $u_{0}(A_{0})$. As the spectral flow  $\spf(-\slashed{D}_{A},\alpha)$ can be calculated using excision and the Atiyah-Singer index theorem (see  of \cite[Lemma~14.4.6]{Kronheimer-Mrowka}), it is not hard to check that $\spf(-\slashed{D}_{A},\alpha)=\frac{l}{2}$. Combining the above two equivalences with (\ref{change by suspension (attractor)}) and (\ref{eq gaugeequivswf}), we get
\[
    \begin{split}
       &
        \underline{\text{swf}}^{A}(Y,\mathfrak{s},A_{0},g;S^{1})
        \cong
        \Sigma^{\frac{l}{2}\mathbb{C}}\underline{\text{swf}}^{A}(Y, \mathfrak{s},A_{0}, g; S^{1})
    \end{split}
\]
and similar results hold for $\underline{\text{swf}}^{R} $. This proves the periodicity result in the main theorem.

\section{The linearized Seiberg-Witten flow and its associated spectra}\label{section linearized flow}

 Let $(Y,\mathfrak{s},A_0,g)$ be fixed as in the previous  sections, except that $\mathfrak{s} $ is assumed to be torsion.  We first recall a decomposition on the Coulomb slice
\begin{equation}
Coul(Y)=  L^2_k(i \im d^{*} \oplus i\Omega^{1}_{h}(Y)\oplus \Gamma(S_{Y})).
\end{equation}
The linearized Seiberg-Witten flow we will consider is obtained by scaling the (perturbed) Seiberg-Witten flow in the direction of $i \im d^{*} \oplus \Gamma(S_{Y}) $.

  A general setup of the linearized Seiberg-Witten flow can be described as follows. Pick a real number $\delta$ and consider $\mathbf{D}(h)= \slashed{D}_{A_{0}+h}+\delta$ as a smooth family of self-adjoint elliptic operators on the spinor bundle parametrized by $i\Omega^{1}_{h}(Y)$. Choose an even Morse function $g_{H} \colon \mathbb{R} \rightarrow \mathbb{R} $ such that $g_H (\theta +1)=g_H (\theta) $ (usually we use $g_H (\theta) = -\cos(2\pi \theta)$).   After identifying $i\Omega^{1}_{h}(Y) $ with $\mathbb{R}^{b_1} $ and the action of  $\mathcal{G}_{Y}^{h,o}$ with addition by $\mathbb{Z}^{b_1}
$,    we consider a function $f_H \colon i\Omega^{1}_{h}(Y)\rightarrow \mathbb{R}$ given by $f_H (\theta_{1}, \ldots ,\theta_{b_1}) = \sum^{b_1}_{i=1} g_H (\theta_i)$. The function $f_H$ can be viewed as a Morse function on the Picard torus of $Y$. The linearized Seiberg-Witten flow is given by a trajectory on $Coul(Y)$ satisfying
 \begin{equation}\label{eq general linearized flow}
-\frac{d}{dt}\left(\beta(t),h(t),\phi(t)\right)=\left(*d\beta(t), \grad f_H(h(t)),\mathbf{D}(h(t))\phi(t)\right).
\end{equation}

Recall that a trajectory $\gamma$ is called finite type if the image $\gamma(\mathbb{R})$ is contained in a bounded subset of $Coul(Y)$. We will be interested in a particular situation where all the finite type trajectories are reducible and lie in the space of harmonic 1-forms.
\begin{lem}\label{trajectories for the linearized flow are reducible}
Suppose that the family of self-adjoint operators $2\mathbf{D}^{2}+ \rho(\grad f_H) $ is positive definite for all  $ h \in  i\Omega^{1}_{h}(Y)$. Then, any finite type trajectory of the linearized Seiberg-Witten flow is contained in $i\Omega^{1}_{h}(Y)$.
\end{lem}
\begin{proof}
Let $\gamma(t)=(\beta(t),h(t),\phi(t))$ be such a trajectory. We want to show that $\beta(t)=0$ and $\phi(t)=0$ for any $t$. Since $\gamma$ is of finite type, we have $\mathop{\sup} \limits_{t \in \mathbb{R}} \|\beta(t)\|_{L^{2}}^{2}<\infty$ and $\mathop{\sup} \limits_{t \in \mathbb{R}} \|\phi(t)\|_{L^{2}}^{2}<\infty$. Since $*d $ has no kernel on $i \im d^{*} $, it is easy to deduce that $\beta \equiv0 $.

For the spinor part, we compute the second derivative of its $L^2$-norm
\begin{align*}
\frac{d^2}{dt^{2}}\|\phi(t)\|_{L^{2}}^{2}&=-2\frac{d}{dt}\langle \mathbf{D}(h(t))\phi(t),\phi(t) \rangle_{L^{2}}\\
&=-2\langle(\frac{d}{dt}\mathbf{D}(h(t)))\phi(t)+\mathbf{D}(h(t))(\frac{d}{dt}\phi(t)),\phi(t) \rangle _{}-2\langle\mathbf{D}(h(t))\phi(t),\frac{d}{dt}\phi(t) \rangle _{}\\
&=2\langle (\mathcal{D}_{\grad f_H} \mathbf{D} + 2\mathbf{D}^2)(h(t)) \phi(t),\phi(t)\rangle_{} \\
&= 2\langle (\rho(\grad f_H) + 2\mathbf{D}^2)(h(t)) \phi(t),\phi(t)\rangle_{}.
\end{align*}
By the assumption, we have $\frac{d^2}{dt^{2}}\|\phi(t)\|_{L^{2}}^{2}\geq 0$ and the equality hods if and only if $\phi(t) = 0$. It is not hard to see that $\phi \equiv 0$.
 \end{proof}

 We can decompose the linearized Seiberg-Witten vector field (right hand side of (\ref{eq general linearized flow})) as $l + c_0 $ where $l(\beta, h, \phi) = (*d \beta, 0, \slashed{D}_{A_{0}}\phi) $ is linear and $c_0$ is a nonlinear part given by
\begin{equation}\label{linearized c}c_{0}(\beta,h,\phi):=(0,\grad f_H(h),(\rho(h)+\delta)\phi).\end{equation}
It is not hard to check that $c_{0}$ is quadratic-like. As a result, we can obtain Conley indices from finite dimensional approximation.   We then construct associated spectra as before, although some  modification to the choices of a strip of balls and cut-off functions are made in order to simplify computation.

Recall that $V^{\mu}_{\lambda}$ is the span of the eigenspaces of $l$
with eigenvalues in the interval $(\lambda,\mu]$. Furthermore, we will consider the space $V^{\mu , \nu}_{\lambda}$ spanned by the eigenvectors of $*d|_{\ker d^{*}}$ with eigenvalue in $(\lambda_{},\mu]$
and the eigenvectors of $\slashed{D}_{A_{0}}$ with eigenvalue in $(\lambda_{}, \nu]$. The space $\bar{V}^{\mu}_{\lambda}$ and  $\bar{V}^{\mu , \nu}_{\lambda_{}}$
is the orthogonal complement of $i\Omega^{1}_{h}(Y)$ in $V^{\mu}_{\lambda}$ and $V^{\mu,\nu}_{\lambda_{}}$ respectively.

\begin{enumerate}
\item Choose  real numbers $\theta^{\pm}\in (0,1)$ with $\frac{d}{dt}g_H (\theta^{+})>0$ and $\frac{d}{dt}g_H (\theta^{-})<0$.
\item For a suitable real number $\bar{\epsilon}>0$, we define a set
$$
               \widebar{Str}(\bar{\epsilon}):=
          \{u \cdot (\beta,h,\phi)  \mid  u \in \mathcal{G}^{h,o}_{Y},  \| \beta \|_{L^{2}_{k}}\leq 1,
                                                         h \in   [0,1]^{b_1}, \|\phi\|_{L^{2}_{k}}\leq \bar{\epsilon} \}
$$
and let
\[
       \bar{J}_{m}^{\pm}=
       \widebar{Str}(\bar{\epsilon})   \cap   p_\mathcal{H}^{-1}([-\theta^{\pm}-m,\theta^{\pm}+m]^{b_1}).
\]

\item For sequences $\{\lambda_{n}\}$, $\{ \mu_n \}$ with $-\lambda_n , \mu_n \rightarrow \infty $, we consider $\bar{J}^{n,\pm}_{m}=\bar{J}^{\pm}_{m}\cap V^{\mu_{n}}_{\lambda_{n}}$ as before. Let $\bar\varphi^{n}_{m}$ be the flow on $V^{\mu_{n}}_{\lambda_{n}}$ generated by the vector field $\iota_{m} \cdot (l+p^{\mu_{n}}_{\lambda_{n}}\circ c_{0})$ where $\iota_{m}$ is a bump function with value 1 on $\bar{J}^{\pm}_{m+1}$. When $n$ is sufficiently large (depending on $m$), the compact sets $\bar{J}^{n,\pm}_{m}$ are isolating neighborhoods under the flow $\bar\varphi^{n}_{m}$. This allows us to define spectra
\[
     \bar{I}^{n,+}_{m} :=\Sigma^{-\bar{V}^{0,-\delta}_{\lambda_{n}}}I_{S^{1}}(\bar\varphi^{n}_{m},\inv (\bar{J}^{n,+}_{m})), \
     \bar{I}^{n,-}_{m} :=\Sigma^{-V^{0,-\delta}_{\lambda_{n}}}I_{S^{1}}(\bar\varphi^{n}_{m},\inv (\bar{J}^{n,-}_{m})).
\]
As before, there are isomorphisms $\bar{I}^{n,\pm}_{m}  \cong \bar{I}^{n+1,\pm}_{m}$.

\item It is straightforward to check that the flow $\bar\varphi^{n}_{m}$ goes inside $\bar{J}^{n,+}_{m}$
along $\partial \bar{J}^{n,+}_{m}\setminus \partial \bar{J}^{n,+}_{m+1}$ and goes outside $\bar{J}^{n,-}_{m}$ along $\partial
\bar{J}^{n,-}_{m}\setminus \partial \bar{J}^{n,-}_{m+1}$. Therefore, the
attractor  maps give a direct system
\begin{equation} \label{eq bar I +}
      \bar{I}^{n_{1},+}_{1}   \rightarrow   \bar{I}^{n_{2},+}_{2}  \rightarrow   \cdots
\end{equation}
and the repeller maps give an inverse system
\begin{equation} \label{eq bar I -}
    \bar{I}^{n_{1},-}_{1}   \leftarrow \bar{I}^{n_{2},-}_{2}   \leftarrow \cdots
\end{equation}
for a suitable increasing sequence of positive integers $\{ n_m \}$.

\item When $\mathfrak{s}$ is a spin structure, we define an direct system and inverse system of $Pin(2)$-equivariant spectra $\bar{I}^{n,\pm}_{m}(Pin(2))$ in the same manner:
\begin{equation} \label{eq bar I Pin(2)}
     \begin{split}
   &    \bar{I}_1^{n_1, +}(Pin(2)) \rightarrow \bar{I}^{n_2, +}_2(Pin(2)) \rightarrow \cdots,   \\
   &    \bar{I}_1^{n_1,-}(Pin(2)) \leftarrow \bar{I}^{n_2, -}_2(Pin(2)) \leftarrow \cdots.
     \end{split}
\end{equation}

\end{enumerate}


To keep track of the number of eigenvalues near 0, we introduce a definition.

\begin{defi} Let $L$ be a self-adjoint elliptic operator, we define a signed count with multiplicity of the eigenvalues of $L$, i.e.
$$
m(L,\delta):=\left\{\begin{array}{l l}
   \# \text{ eigenvalues of }L \text{ in } (-\delta,0]  & \quad \text{if }\delta\geq 0\\
  -(\# \text{ eigenvalues of }L \text{ in } (0,-\delta]) & \quad \text{if }\delta< 0\\
  \end{array} \right..
  $$
\end{defi}
The following theorem is the main tool for our calculations of the spectrum invariants when the Seiberg-Witten Floer spectra agree with the linearized ones.

\begin{thm}\label{spectrum for the linearized flow}
Let $Y$ be a 3-manifold equipped with a torsion spin$^c$ structure $\mathfrak{s}$. Suppose that we can find an  even Morse function $g_{H} \colon \mathbb{R} \rightarrow \mathbb{R} $ with $g_H (\theta +1)=g_H (\theta) $   and an extended cylinder function $\bar{f}\in \mathcal{P}$ satisfying the following conditions:\begin{enumerate}[(i)]
\item The function  $f_H \colon i\Omega^{1}_{h}(Y)\rightarrow
\mathbb{R}$ given by $f_H (\theta_{1}, \ldots ,\theta_{b_1})
= \sum^{b_1}_{i=1} g_H (\theta_i)$ satisfies the hypothesis of Lemma~\ref{trajectories for the linearized flow are reducible};

\item There exists $\epsilon'>0$ such that $\bar{f}(a,\phi)=0$ for any $(a,\phi)$ with $\|\phi\|_{L^{2}}\leq \epsilon'$;

\item All critical points and finite type gradient flow lines of the functional
\begin{equation} \label{eq functional for deforming}
\bar{\mathcal{L}}(a,\phi) := CSD|_{Coul(Y)}(a,\phi)+\frac{\delta}{2}\|\phi\|_{L^{2}}^{2}+ f_H (p_\mathcal{H} (a,\phi))+\bar{f}(a,\phi)
\end{equation}
on the Coulomb slice $Coul(Y)$ are contained in $i\Omega^{1}_{h}(Y)$.
\end{enumerate}
Then, we have
\[
    \begin{split}
     \underline{\operatorname{SWF}}^{A}(Y,\mathfrak{s};S^{1})
       &  \cong
    ( \bar{I}^+,  0,  n(\mathfrak{s},A_{0},g)+m(\slashed{D},\delta)),   \\
   \underline{\operatorname{SWF}}^{R}(Y,\mathfrak{s};S^{1})
     &\cong
     (\bar{I}^-,  0,  n(\mathfrak{s},A_{0},g)+m(\slashed{D},\delta)),
     \end{split}
\]
as objects in $\mathfrak{S}$ and $\mathfrak{S}^*$ respectively. Here $\bar{I}^{+}$ and $\bar{I}^-$ are  the inductive system (\ref{eq bar I +}) and the inverse system (\ref{eq bar I -}) of the spectra {associated} to the linearized flow constructed above.

Furthermore, when $\mathfrak{s}$ is a spin structure, then we also have
\[
    \begin{split}
        \underline{\operatorname{SWF}}^{A}(Y,\mathfrak{s};Pin(2))
          &  \cong
            \left( \bar{I}^+(Pin(2)), 0, \frac{n(\mathfrak{s}, A_{0}, g)+m(\slashed{D},\delta)}{2} \right), \\
        \underline{\operatorname{SWF}}^{R}(Y,\mathfrak{s};Pin(2))
          &\cong
        \left( \bar{I}^-(Pin(2)),0,\frac{n(\mathfrak{s}, A_{0}, g)+m(\slashed{D},\delta)}{2} \right).
    \end{split}
\]
as objects in $\mathfrak{S}_{Pin(2)}$ and $\mathfrak{S}_{Pin(2)}^*$ {respectively}. Here $\bar{I}^+(Pin(2))$ and $\bar{I}^-(Pin(2))$ are the direct system and inverse system of spectra in (\ref{eq bar I Pin(2)}) respectively.
\end{thm}

\begin{proof}
We first consider the spectrum invariants $\underline{\text{swf}}^{}(Y,\mathfrak{s},g,A_{0};S^{1})$ associated to the functional $\bar{\mathcal{L}}$.
By Remark~\ref{different strips} and the hypothesis, we can use the strip $\widebar{Str}(\bar{\epsilon})$ together with level sets $p_\mathcal{H}^{-1}([-\theta^{\pm}-m,\theta^{\pm}+m]^{b_1})$ for constructing these spectrum invariants. Note that the set of critical points of this functional is discrete
since they correspond to the critical points of the Morse function $f_H$.

Let $\{h_j\}$ be a chosen basis of $i\Omega^{1}_{h}(Y)$. Since a function $\pi_\mathcal{H} (\rho^{-1}(\phi\phi^{*})_{0}) $ is bounded on the strip and is homogeneous of degree 2, we can choose $\bar{\epsilon}$ such that $\left\Vert\pi_\mathcal{H} (\rho^{-1}(\phi\phi^{*})_{0})\right\Vert  \le \left\vert \frac{d}{dt}g_H (\theta^{\pm}) \right\vert / \sum \left\Vert h_j \right\Vert_{L^2} $ for any $(a,\phi) \in \widebar{Str}(\bar{\epsilon})$. Furthermore, we require that $\bar{\epsilon} < \epsilon' $.

We now consider a continuous family of vector fields $l+c_{s}$ parametrized by $s\in [0,1]$ on $\widebar{Str}(\bar{\epsilon})$ with $l = (*d,0,\slashed{D}) $ and $c_{s}$ is a family of quadratic-like maps  given by
\begin{equation*}\label{rescaled nonlinear term}
\begin{split}c_{s}(\beta,h,\phi)=&(0 , \grad \tilde{f}(h), (\rho(h)+\delta)\phi)\\
&+s(\pi_{\im d^{*}}\rho^{-1}(\phi\phi^{*})_{0},s\pi_{\mathcal{H}}\rho^{-1}(\phi\phi^{*})_{0},\rho(\beta)\phi+ s^{}\bar{\xi}(\rho^{-1}(\phi\phi^{*})_{0})\phi)
\end{split}\end{equation*}
(see (\ref{solving laplace}) for the definition of $\bar{\xi}$). Observe that $l+c_1$ is the gradient of $\bar{\mathcal{L}}$ as $\bar f$ vanishes on $\widebar{Str}(\bar{\epsilon})$ whereas $l + c_0 $ is exactly the linearized Seiberg-Witten vector field described earlier. For $0<s<1$, the scaling automorphism gives the following commutative diagram:

\begin{equation*}\label{rescaling diagram}
\xymatrix{
 Coul(Y) \ar[d]_{l+c} \ar[r]& Coul(Y)\ar[d]^{l+c_{s}}\\
L^{2}_{k-1}(i\ker (d^{*})\oplus \Gamma(S_{Y})) \ar[r]& L^{2}_{k-1}(i\ker (d^{*})\oplus \Gamma(S_{Y})),}
\end{equation*}
where both the horizontal maps send $(\beta,h,\phi)$ to $(s\cdot \beta,h,s\cdot \phi)$. In other words, the flow given by $l + c_s $ on $\widebar{Str}(\bar{\epsilon})$ corresponds to the flow $l + c_1 $ on $\widebar{Str}(s^{-1} \bar{\epsilon})$.

All that is left is to apply the homotopy invariance of Conley index. We will focus on the case $\underline{\text{swf}}^{A}$ as the repeller case can be done similarly.
Since $\frac{d}{dt} |_{t = \theta^+}  g_H (t) > 0 $, it is easy to find $\epsilon_1 > 0$ such that any gradient trajectory of $f_H$ in $[-\theta^{+}-m,\theta^{+}+m]^{b_1}$ actually lies in $[-\theta^{+}-m+\epsilon_{1},\theta^{+}+m-\epsilon_{1}]^{b_1} $. We have to check that the inner product $\left\langle p_\mathcal{H} \circ(l +p^{\mu_{n}}_{\lambda_{n}}\circ c_s)(a,\phi)  , h_j \right\rangle  $ is positive on $p_\mathcal{H}^{-1} (\theta^+ h_j) \cap \widebar{Str}(\bar{\epsilon})$.
We can see that
\begin{align*} \left\langle p_\mathcal{H} \circ(l +p^{\mu_{n}}_{\lambda_{n}}\circ
c_s)(a,\phi)  , h_j \right\rangle &= \left\langle \grad f_H (\theta^+ h_j) + s \pi_\mathcal{H}
(\rho^{-1}(\phi\phi^{*})_{0}),h_j \right\rangle \\
& \ge \frac{d}{dt}g_H (\theta^{+}) -  \left\Vert  \pi_\mathcal{H}
(\rho^{-1}(\phi\phi^{*})_{0}) \right\Vert \left\Vert h_j \right\Vert
\end{align*}
is indeed positive by our choice of $\bar{\epsilon} $.



Therefore, we have an isomorphism $\underline{\text{swf}}^{A}(Y,\mathfrak{s},g,A_{0};S^{1})\cong ( \bar{I}^+,0,m(\slashed{D},\delta))$. Note that $m(\slashed{D},\delta)$ appears because we desuspend the Conley index by $-\bar{V}^{0,-\delta}_{\lambda_{n}}$  instead of  $\bar{V}^{0}_{\lambda_{n}}$ in the definition of $\bar{I}^{n,+}_{m}$.

\end{proof}

In practice, we will further deform the family of operators $\mathbf{D}$ to simplify the calculation of the Conley index of $\inv (\bar{J}^{n,\pm}_{m},\bar\varphi^{n}_{m})$ as follows; Let $\{\mathbf{Q}_{s}(h)\}$ be a smooth family of 0-th order symmetric operator defined on $\Gamma(S_{Y})$ parametrized by $(s,h)\in [0,1]\times i\Omega^{1}_{h}(Y)$ with $\mathbf{D} = \slashed{D} + \mathbf{Q}_0 $.
When $2(\slashed{D}+\mathbf{Q}_s)^{2}+ \rho(\grad f_H) $ is positive definite
for any $h\in [-\theta^{\pm}-m,\theta^{\pm}+m]^{b_1}$, we can repeat our construction for the linearized Seiberg-Witten vector field to obtain spectra $\bar{I}^{n,\pm}_{m'}(s)$ associated to $\slashed{D}+\mathbf{Q}_s$
for any $m' \le m$ and $n$ is sufficiently large.
The following lemma is immediate from Theorem~\ref{finite dimensional approximation for a family of flows}.
\begin{lem}\label{deformation of linearized flow}
Suppose that $2(\slashed{D}+\mathbf{Q}_s)^{2}+ \rho(\grad f_H) $ is positive definite for any $s \in [0,1]$ and $h\in [-\theta^{\pm}-m,\theta^{\pm}+m]^{b_1}$. We also assume that  $\mathbf{Q}_{s}(0)$ does not depends on $s$ so that the desuspension indices are constant. Then, for any $s,s'\in [0,1]$, we have isomorphisms $\tau^{n,\pm}_{m'}(s,s'): \bar{I}^{n,\pm}_{m'}(s)\rightarrow \bar{I}^{n,\pm}_{m'}(s')$ when $m'\leq m$ and $n$ sufficiently large relative to $m'$. Furthermore, these isomorphisms  satisfy the following commutative diagram
\begin{equation*}\label{rescaling diagram}
\xymatrix{
  \bar{I}^{n,+}_{m'-1}(s)\ar[d]_{\tau^{n,+}_{m'-1}(s,s') } \ar[r]^{ }& \bar{I}^{n,+}_{m'}(s)\ar[d]^{\tau^{n,+}_{m'}(s,s') }\\
\bar{I}^{n,+}_{m'-1}(s') \ar[r]_{ }& \bar{I}^{n,+}_{m'}(s'),}
\end{equation*}
where the horizontal maps are attractor maps. A similar diagram holds for the repeller maps. When $\mathfrak{s}$ is spin, similar results hold for the $Pin(2)$-equivariant spectra.
\end{lem}

Next, we describe a situation when Theorem~\ref{spectrum for the linearized flow} implies that the spectrum invariants are just some suspension of sphere spectra.

\begin{thm}\label{Morse-Bott reducible}
Let $Y$ be a 3-manifold with a torsion spin$^c$ structure $\mathfrak{s}$ and a Riemannian metric $g$. Suppose that there exists a real number $\delta$ such that the following conditions hold:
\begin{enumerate}[(i)]
\item The functional $CSD+\frac{\delta}{2}\|\phi\|_{L^{2}}^{2}$ has only reducible critical points;
\item The operator $\slashed{D}_{A}+\delta$ has no kernel for any $A$ with $F_{A^{t}}=0$.
\end{enumerate}
Then we have
\begin{align*}
         \underline{\operatorname{SWF}}^{A}(Y,\mathfrak{s};S^{1})
        &\cong
        (S^{0},0,n(Y,\mathfrak{s},A_{0},g)+m(\slashed{D}_{A_0}, \delta))\in \ob(\mathfrak{S}),\\
        \underline{\operatorname{SWF}}^{R}(Y,\mathfrak{s};S^{1})
       &\cong
       (S^{0},0,n(Y,\mathfrak{s}, A_{0}, g)+m(\slashed{D}_{A_0},\delta))\in \ob(\mathfrak{S}^{*}),
\end{align*}
where $A_{0}$ is any base connection with $F_{A^{t}_{0}}=0$.

Moreover, if $\mathfrak{s}$ is a spin structure, we also have
\begin{align*}
              \underline{\operatorname{SWF}}^{A}(Y,\mathfrak{s};Pin(2))
              &\cong
              \left(  S^{0},0,\frac{n(Y,\mathfrak{s},A_{0},g)+m(\slashed{D}_{A_0},\delta)}{2}  \right)
              \in    \ob(\mathfrak{S}_{Pin(2)}),\\
              \underline{\operatorname{SWF}}^{R}(Y,\mathfrak{s};Pin(2))
              &\cong
              \left(  S^{0},0,\frac{n(Y,\mathfrak{s},A_{0},g)+m(\slashed{D}_{A_0},\delta)}{2}  \right)
              \in  \ob(\mathfrak{S}^{*}_{Pin(2)}).\end{align*}
\end{thm}


\begin{proof} Since $\slashed{D}_A + \delta$ has no kernel for any flat connection, we can find a positive number
$\sigma_0 $ such that
\begin{equation} \label{eq uniform posdef flat}
\langle (\slashed{D}_{A_0 + h} + \delta)^{2}\phi,\phi\rangle_{L^{2}}\geq \sigma_{0}\|\phi\|^{2}_{L^{2}}
\end{equation}
for any $ h \in i\Omega^{1}_{h}(Y)$ and $\phi \in L^2 (\Gamma(S_Y))$. We divide the proof into two steps.

\textbf{Step 1} We would like to show that there exists $\epsilon_{0}>0$
such that, for any $\epsilon\in [0,\epsilon_{0}]$, all critical points and finite type negative gradient flow lines of a
functional
$$
\mathcal{L}_{\delta, \epsilon_{}} :=CSD+\frac{\delta}{2}\|\phi\|_{L^{2}}^{2}+\epsilon f_H({p}_\mathcal{H} (a,\phi))
$$
on $Coul(Y)$ are contained in $i\Omega^{1}_{h}(Y)$ with $f_H \colon i\Omega^{1}_{h}(Y)\cong \mathbb{R}^{b_1}\rightarrow
\mathbb{R}$ given explicitly by $f_H(\theta_{1},\ldots,\theta_{b_1})=-\sum_{j=1}^{b_1} \cos(2\pi \theta_{j})$. This will hold once we can show that $\frac{d^2}{dt^{2}}\|\phi_{}(t)\|^{2}_{L^{2}}>0$ for any irreducible trajectory of finite type $\gamma = (a,\phi) \colon \mathbb{R}\rightarrow Coul(Y) $ of the gradient flow of the functional $\mathcal{L}_{\delta, \epsilon_{}} $.

Suppose that there are a sequence of irreducible trajectories of finite type $\gamma_n = (a_n,\phi_n)$ of the negative gradient flow of  $\mathcal{L}_{\delta, \epsilon_{n}} $ with positive $\epsilon_n \rightarrow 0 $ and a sequence of real numbers $\{t_n\} $ such that
$$
             \left. \frac{d^2}{dt^{2}} \right|_{t = t_n} \|\phi_{n}(t)\|^{2}_{L^{2}}\leq 0.
$$
Since $\gamma_{n}$ is of finite type,
we have $\lim_{t\rightarrow \pm \infty} \gamma_n(t) = \mathfrak{a}^\pm_n$ in $C^{\infty}$, where $\mathfrak{a}_{n}^\pm$ are critical
points of $\mathcal{L}_{\delta,\epsilon_{n}}$. After passing to a subsequence, we can assume that $\mathfrak{a}_{n}^\pm \rightarrow \mathfrak{a}^{\pm}_{\infty}$, where $\mathfrak{a}^{\pm}_{\infty}$ are critical points of $\mathcal{L}_{\delta,0}$. By hypothesis,  $\mathfrak{a}^\pm_{\infty}$ are reducible and $\mathcal{L}_{\delta,0}(\mathfrak{a}^+_{\infty})=\mathcal{L}_{\delta,0}(\mathfrak{a}^-_{\infty})$.
This implies $\mathcal{L}_{\delta,0}(\mathfrak{a}^+_{n})-\mathcal{L}_{\delta,0}(\mathfrak{a}^-_{n})\rightarrow 0$.
In other words, the topological energy of $\gamma_{n}$ approaches zero.

Now we treat a finite trajectory $\gamma_{n}|_{[t_{n}-1,t_{n}+1]}$, denoted by $\hat{\gamma}_{n}$, as a configuration of a 4-manifold $[-1,1]\times Y$. By the standard compactness result of the 4-dimensional
Seiberg-Witten equations, after passing to a subsequence and applying suitable gauge transformations,  $\hat{\gamma}_{n}$
converges to a negative gradient flow line of $\mathcal{L}_{\delta,0}$ in $C^{\infty}$ on any interior domain.
It is not hard to check that the topological energy of $\hat{\gamma}_n $ also approaches 0, so $\hat{\gamma}_n $ in fact converges to a constant trajectory.
In particular, there is $u_{n}\in \mathcal{G}^{h,o}_{Y}$ such that $u_n \cdot \gamma_{n}(t_{n})$  to a critical
point $(h,0)$ of $\mathcal{L}_{\delta,0}$ with $h\in i\Omega^{1}_{h}(Y)$ by the hypothesis.

Since the quantity $\|\phi_{n}(t_{n})\|^{2}_{L^{2}}$ is gauge invariant, we can instead consider a lift of $\gamma_{n}$
to a path $\tilde{\gamma}_{n}=(\tilde{a}_{n},\tilde{\phi}_{n}) \colon \mathbb{R}\rightarrow
\mathcal{C}_{Y}$
which is a negative gradient flow line of the functional $\mathcal{L}_{\delta,\epsilon_{n}}$ on $\mathcal{C}_Y $ with $\tilde{\gamma}_{n}(t_{n})=u_n \cdot \gamma_{n}(t_{n})$.
In particular, we have
\begin{align*}
-\frac{d}{dt}\tilde{\phi}_{n}(t)&=(\slashed{D}_{A_{0}+\tilde{a}_{n}(t)}+\delta)\tilde{\phi}_{n}(t), \\
- \frac{d}{dt}\tilde{a}_{n}(t) &=*d(\tilde{a}_n(t))+\rho^{-1}(\tilde{\phi}(t)\tilde{\phi}(t)^{*})_{0}+ \epsilon_n \grad
f_{H}(\tilde{h}_{n}(t)),
\end{align*}
where $\tilde{h}_{n}(t)$ is the projection of $\tilde{a}_n (t)$ onto $i\Omega^{1}_{h}(Y)$. By a calculation similar to one in the proof Lemma \ref{trajectories
for the linearized flow are reducible}, we obtain
\begin{equation} \label{eq 2deriv upstair}
\begin{split}
   \left.\frac{d^{2}}{dt^{2}} \right|_{t = t_n}   \|\tilde{\phi}_{n} (t)  \|_{L^{2}}^{2}= &
4\langle(\slashed{D}_{A_{0}+\tilde{a}_{n}(t_{n})}+\delta)^{2}\tilde{\phi}_{n}(t_{n}),\tilde{\phi}_{n}(t_{n})\rangle_{L^{2}}+\|\tilde{\phi}_{n}(t_{n})\|_{L^{2}}^{4}\\
&+2\langle\rho(*d\tilde{a}_{n}(t_n)+\epsilon_n \grad
f_{H}(\tilde{h}_{n}(t_n)))\tilde{\phi}_{n}(t_{n}),\tilde{\phi}_{n}(t_{n})\rangle_{L^{2}}.
\end{split}
\end{equation}
Since $\tilde{a}_{n}(t_{n})\rightarrow h \in i\Omega^{1}_{h}(Y)$, we have $*d\tilde{a}_{n}(t_n)+\epsilon_n \grad
f_{H}(\tilde{h}_{n}(t_n)) \rightarrow 0$. Moreover, we can deduce from (\ref{eq uniform posdef flat}) that
$$
        \langle(\slashed{D}_{A_{0}+\tilde{a}_{n}(t_{n})}+\delta)^{2} \tilde{\phi}_{n}(t_{n}),
                     \tilde{\phi}_{n}(t_{n})\rangle_{L^{2}}
             \geq
             \frac{\sigma_0}{2}\|\tilde{\phi}_{n}(t_n)\|_{L^{2}}^{2},
$$
for $n$ sufficiently large. Notice that $\tilde{\phi}_{n}(t_{n})\neq 0$ since $\gamma_{n}$
is irreducible by the unique continuation property \cite[Lemma~10.8.1]{Kronheimer-Mrowka}. Therefore, for a sufficiently large $n$, we can conclude that
$\left. \frac{d^2}{dt^{2}} \right|_{ t=t_n}  \|\phi_{n}(t)\|_{L^{2}}^{2}>0$, which is a contradiction.

\textbf{Step 2} Let $\epsilon_0$ be a constant from Step~1. We pick a positive real number $\epsilon$ satisfying
 $$
 \epsilon< \min\{\epsilon_0, \frac{\sigma_0}{ \sup\limits_{h\in i\Omega^{1}_{h}(Y)}2\|\grad f_{H}(h)\|_{C^{0}}}\} $$ so that critical points and finite type gradient trajectories of $\mathcal{L}_{\delta, \epsilon_{}} $ lie in $i\Omega^{1}_{h}(Y)$ and the family of operators
$$
2(\slashed{D}_{A_0 + h} + \delta)^{2} + \epsilon \rho( \grad f_H(h))
$$
is positive definite. Consequently, we can apply Theorem~\ref{spectrum for the linearized flow} and compute Conley indices of the linearized flow.


Let us focus on the case of $\underline{\text{\text{SWF}}}^{A}(Y,\mathfrak{s};S^{1})$
as the other cases are almost identical. To simplify the calculation, we further deform the family of linear operators
$\slashed{D}_{A_0 + h} + \delta$ through a family $\slashed{D}_{A_0 + (1-s)h} + \delta$ as $0 \le s \le 1 $. We see that, in fact,
$$
        2(\slashed{D}_{A_0 + (1-s)h} + \delta)^{2}+ \epsilon \rho( \grad f_H(h))
$$
is positive definite for all $s \in [0,1]$ and $h \in i\Omega^{1}_{h}(Y) $.
Therefore, Lemma \ref{deformation of linearized flow} allows us to consider
an approximated linearized flow $\bar{\varphi}_{m,1}^{n}$ associated to the constant family $\slashed{D}_{A_0 } + \delta$.
This flow actually splits into the product of the following three flows:
\begin{enumerate}
\item The negative gradient flow of $\epsilon f_H$ on $i\Omega^{1}_{h}(Y)$;
\item The linear flow on $L^{2}_{k}(\im (d^{*}))\cap V^{\mu_{n}}_{\lambda_{n}}$ associated to the operator $-(*d)$;
\item The linear flow on $L^{2}_{k}(\Gamma(S_{Y}))\cap V^{\mu_{n}}_{\lambda_{n}}$ associated to the operator $-(\slashed{D}+\delta)$.
\end{enumerate}
From this, we see that
$$
I_{S^{1}}(\bar{\varphi}_{m,1}^{n},\inv (\bar{J}^{n,+}_{m}))\cong (\bar{V}^{0,-\delta}_{\lambda_{n}})^{+}.
$$
Moreover, the attractor map $I_{S^{1}}(\bar{\varphi}_{m,0}^{n},\inv (\bar{J}^{n,+}_{m}))\rightarrow I_{S^{1}}(\bar{\varphi}_{m,0}^{n},\inv
(\bar{J}^{n,+}_{m+1}))$ is just the identity map. Hence, we can conclude that $\underline{\text{\text{SWF}}}^{A}(Y,\mathfrak{s};S^{1})\cong
(S^{0},0,n(Y,\mathfrak{s},A_{0},g)+m(\delta))$.
\end{proof}

\begin{ex}[$S^2 \times S^1$]
Let $\mathfrak{s}$ be the unique torsion spin$^{c}$ structure on $Y=S^{2}\times S^{1}$ and
$g$ be a Riemannian metric with constant positive scalar curvature. By the Weitzenb$\ddot{\text{o}}$ck formula, the triple
$(Y,\mathfrak{s},g)$ satisfies the conditions in Theorem \ref{Morse-Bott reducible} with $\delta=0$. Therefore, we only need to compute the number $n(Y,\mathfrak{s},A_{0},g)$ to completely describe the spectrum invariants. We choose the base
connection $A_{0}$ such that the induced connection $A_{0}^{t}$ is flat with trivial holonomy (up to gauge transformations,
there are two choices of such connection  and we pick any one of them). There exists an orientation reversing isometry on
$Y$ preserving $(\mathfrak{s},A_{0})$, so $\mathfrak{s}$ and $A_{0}$ correspond in a natural way to a spin$^{c}$
structure and a spin$^{c}$ connection on $-Y$. This implies that $n(Y,\mathfrak{s},A_{0},g)=n(-Y,\mathfrak{s},A_{0},g)$.
On the other hand, by formula (\ref{virtual dimension}), we have $n(Y,\mathfrak{s},A_{0},g)+n(-Y,\mathfrak{s},A_{0},g)=\text{dim}_{\mathbb{C}}(\ker
\slashed{D})=0$ as $\slashed{D}$ has no kernel. Therefore, we get $n(Y,\mathfrak{s},A_{0},g)=0$ and conclude that
 $$\underline{\text{\text{SWF}}}^{A}(S^{2}\times S^{1},\mathfrak{s};S^{1})\cong (S^{0},0,0),\text{ } \underline{\text{\text{SWF}}}^{R}(S^{2}\times
S^{1},\mathfrak{s};S^{1})\cong (S^{0},0,0).$$
Notice that $\mathfrak{s}$ can be lifted to a spin structure in two ways, denoted by $\mathfrak{s}^{j}$ $(j=1,2)$.
Then, we have
  $$\underline{\text{\text{SWF}}}^{A}(S^{2}\times S^{1},\mathfrak{s}^{j};Pin(2))\cong (S^{0},0,0),\text{ } \underline{\text{\text{SWF}}}^{R}(S^{2}\times
S^{1},\mathfrak{s}^{j};Pin(2))\cong (S^{0},0,0).$$
\end{ex}

\section{More examples}\label{section examples}
\subsection{The Seiberg-Witten monopoles on Seifert fibered spaces}\label{Seiberg-Witten on Seifert space}
When $Y$ is a Seifert fibered space, the solutions of the Seiberg-Witten equations on $Y$, which are also called the Seiberg-Witten monopoles, are explicitly described by Mrowka, Ozv$\acute{\text{a}}$th and Yu in \cite{MOY}. In this subsection, we will review their set up and some of their results that will be useful for us.

Let $\Sigma$ be an oriented 2-dimensional orbifold of genus $g$ with $n$ marked points, whose associated multiplicities are $\alpha_{1},...,\alpha_{n}$. Its Euler characteristic is given by
$$\chi(\Sigma)=2-2g-\sum_{j=1}^{n}(\frac{1}{\alpha_{j}}-1).$$
An orbifold Hermitian line bundle $N$ over $\Sigma$ can be specified by a set of integers $(b,\beta_{1},...,\beta_{n})$ with $0\leq \beta_{j}< \alpha_{j}$. The orbifold degree of this bundle is given by
$$
\deg (N)=b+\sum_{j=1}^{n}\frac{\beta_{j}}{\alpha_{j}}.
$$

From now on, suppose that the integer $\beta_{j}$ is coprime to $\alpha_{j}$ for every $j$. Then, the unit circle bundle  $S(N)$ is naturally a smooth 3-manifold, called a Seifert fibered space. We will also denote the manifold $S(N)$ by $Y$ and we have
 $$ H^{1}(Y;\mathbb{Z})=\left\{\begin{array}{l l}
    H^{1}(\Sigma;\mathbb{Z}) & \quad \text{if }\deg (N)\neq 0\\
    H^{1}(\Sigma;\mathbb{Z})\oplus \mathbb{Z} & \quad \text{if }\deg (N)=0
  \end{array} \right.;
$$
$$H^{2}(Y;\mathbb{Z})=(\text{Pic}^{t}(\Sigma)/\mathbb{Z}[N])\oplus \mathbb{Z}^{2g}.$$
Here $\text{Pic}^{t}(\Sigma)$ denotes the topological Picard group of $\Sigma$, i.e. the group of topological isomorphism classes of orbifold line bundles over $\Sigma$, and the subgroup $\text{Pic}^{t}(\Sigma)/\mathbb{Z}[N]\subset H^{2}(S(N);\mathbb{Z})$ is identified by the image of the pull-back of the projection $p\colon Y \rightarrow \Sigma $
$$
\text{Pic}^{t}(\Sigma)\xrightarrow{p^*}[\text{line bundles over }Y]\xrightarrow{c_{1}}H^{2}(Y;\mathbb{Z}).
$$

Let $g_{\Sigma}$ be a constant curvature metric on $\Sigma$ with volume $\pi$. For a positive real $r$, we have a natural metric $g_{r}$ on $Y$ given by $r^{2}\alpha^{2}\oplus p^{*}(g_{\Sigma})$, where $i\alpha\in i\Omega^{1}(Y)$ is a constant curvature connection on $S(N)$. We will only have to pick $r$ to be sufficiently small to  computation in \cite{Nicolescu1}.
Instead of the Levi-Civita connection $\nabla^{LC}$, a connection $\nabla^{\circ}$ on $TY$ which is trivial in the fiber direction and equals to the pull back of the Levi-Civita connection on $\Sigma$ when restricted to $\ker
\alpha$ is used in \cite{MOY}. For any spin$^{c}$ structure $\mathfrak{s}$ with spinor bundle $S_{Y}$, there is a natural one-to-one correspondence between connections on $S_{Y}$ spinorial with respect to $\nabla^{LC}$ and connections on $S_{Y}$ spinorial with respect to $\nabla^{\circ}$ by identifying those inducing the same connection on $\text{det}(S_{Y})$. From \cite[Lemma~5.2.1]{MOY}, we have an identity
\begin{equation}\label{two types of dirac operator}
\tilde{\slashed{D}}_{\tilde{A}}=\slashed{D}_{A}+\delta_{r},
\end{equation}
where $\tilde{\slashed{D}}$ is the Dirac operator induced by $\nabla^{\circ} $ and $\tilde{A}$ is the $\nabla^{\circ} $-spinorial connection corresponding to $A$  and $\delta_{r}=\frac{1 }{2}r\deg (N)$ is a constant. Therefore, under this identification, solutions of the Seiberg-Witten equations described in \cite{MOY} actually correspond to critical points of the functional $CSD+\frac{\delta_{r}}{2}\|\phi\|_{L^{2}}^{2}$.

We also have a canonical spin$^{c}$ structure $\mathfrak{s}_{0}$ on $Y$ with spinor bundle $S_{Y}^{0}\cong p^{*}(K_{\Sigma}^{-1})\oplus
\mathbb{C}$, where $K_{\Sigma}$ is the canonical bundle of $\Sigma$ (recall that $\deg (K_{\Sigma})=-\chi(\Sigma)$).
When $\deg (N)$ is nonzero, the spin$^{c}$ structure $\mathfrak{s}_{0}$ is torsion. Moreover, $\text{Pic}^{t}(\Sigma)/\mathbb{Z}[N]$
is a finite {abelian} group and there is a one-to-one correspondence between $\text{Pic}^{t}(\Sigma)/\mathbb{Z}[N]$ and the set of torsion  spin$^{c}$ structures by identifying $[E_0]$ with $p^{*}(E_{0})\otimes S_{Y}^{0} $.
With this understood, the following proposition is a special case of Theorem 5.9.1 and Corollary 5.8.5 of \cite{MOY}
\begin{pro}\label{Some Seifert space only have reducible critical points}
Let $Y=S(N)$ be a Seifert fibred space corresponding to an orbifold line bundle $N$ over $\Sigma$ with nonzero degree. Let $Y$ be equipped with a torsion spin$^{c}$ structure induced by an orbifold line bundle $E_{0}$ over $\Sigma$. Then we have the following results:
\begin{enumerate}
\item The functional $CSD+\frac{\delta_{r}}{2}\|\phi\|_{L^{2}}^{2}$ has only reducible critical points if there is no orbifold line bundle $E$ over $\Sigma$ such that $[E]\equiv [E_0]  \pmod{\mathbb{Z}[N]}$ and $0\leq \deg (E)<-\frac{\chi(\Sigma)}{2}$.

\item The operator $\slashed{D}_{A}+\delta_{r}$ has no kernel for any flat connection $A$ if there is no orbifold
line bundle $E$ over $\Sigma$ such that $[E]\equiv [E_0]  \pmod{\mathbb{Z}[N]}$ and $\deg(E)=-\frac{\chi(\Sigma)}{2}$.

\end{enumerate}

\end{pro}

Later, we will also consider some examples of $S(N)$ with $\deg (N)=0$, specifically, those manifolds with flat metric ($\deg (N)=\chi(\Sigma)=0$). By the Weitzenb$\ddot{\text{o}}$ck formula, the (unperturbed) Chern-Simons-Dirac functionals on these manifolds also have no irreducible critical points.

\subsection{Large degree circle bundles over surface}\label{subsection large degree circle bundles}
Let $\Sigma$ be a smooth surface of genus $g > 0$ and $N_{d}$ be the complex line bundle with degree $d > 0$. As we explained in last subsection, the torsion spin$^{c}$ structures on $Y=S(N_{d})$ can be parametrized by $\mathbb{Z}/d\mathbb{Z}$ in a natural way. We denote them by $\mathfrak{s}_{0},\mathfrak{s}_{1}...,\mathfrak{s}_{d-1}$ accordingly. In this subsection, we consider the torsion spin$^{c}$ structures $\mathfrak{s}_q $ such that $q \ge g$.
By Proposition~\ref{Some Seifert space only have reducible critical points}, we see that the triple $(Y,\mathfrak{s}_{q} , g_r )$ satisfies the conditions of Theorem~\ref{Morse-Bott reducible} with $\delta= \delta_r = \frac{1}{2}rd$.

To describe the spectrum invariants, we will calculate the quantity $n(Y,\mathfrak{s}_{q},A_{0},g_r)+m(\slashed{D},\delta_r)$ following the approach from \cite{Nicolescu1} and \cite{Nicolescu2}. We introduce a family of connections on $TY$ parametrized by $s \in [0,1]$ given by
$$\nabla^{s}=(1-s)\nabla^{\circ}+s\nabla^{LC},$$
where $\nabla^{\circ}$ and $\nabla^{LC}$ are {described} earlier.
Let $A_{0,s}$ be the connection on $S_{Y}$ which is spinorial with respect to $\nabla^{s}$ and induces the same connection as $A_{0}^{t}$ on $\text{det}(S_{Y})$ and let $\slashed{D}^{s}$ be the Dirac operator corresponding to $A_{0,s}$. From (\ref{two types of dirac operator}), we see that
\begin{equation}\label{dirac operator change with s}\slashed{D}^{s}=(1-s)\tilde{\slashed{D}}+s\slashed{D}^{}=\slashed{D}+(1-s)\delta_r.
\end{equation}

Now consider the cylinder 4-manifold $X=I\times Y$. The  family of connections $\{\nabla^{s}\}$ induces a connection $\widehat{\nabla}$ on $T_{}X$ (with temporal gauge). Similarly, the family of connections $\{A_{0,s}\}$ induces a connection $\hat{A}$ on the spinor bundle $S_{X}$. Let $\hat{\slashed{D}}^{+}$ be the positive Dirac operator coupled with $\hat{A}$. We have the following observations:\begin{enumerate}
\item The operator $\tilde{\slashed{D}}^{}$ has no kernel (by Proposition \ref{Some Seifert space only have reducible critical points});
\item Let $\hat{A}^{t}$ be the induced connection on $\text{det}(S^{+}_{X})$. Then, we see that $F_{\hat{A}^{t}}|_{\{s\}\times Y}=0$ for any $s\in [0,1]$, which implies $F_{\hat{A}^{t}}\wedge F_{\hat{A}^{t}}=0$;
\item As pointed out in \cite[Remark~2.10]{Nicolescu1}, although $\widehat{\nabla}$ is not the Levi-Civita connection for the product metric on $X$, we can still use the Atiyah-{Patodi}-Singer index theorem (cf. \cite{APS}) to calculate the index of $\hat{\slashed{D}}^{+}$.
\end{enumerate}
As a result, the Atiyah-{Patodi}-Singer theorem\footnote{The sign convention of this index formula in \cite{Manolescu1} is different from that in \cite{APS} and \cite{Nicolescu1} because $Y$ is oriented in different ways. For consistency, here we follow the convention in \cite{Manolescu1}.} gives
\begin{equation}\label{index formula}
\text{Ind}_{\mathbb{C}}(\hat{\slashed{D}}^{+})=-\frac{1}{24}\int_{X}p_{1}(\hat{A})+\frac{\eta(\slashed{D})-\text{dim}_{\mathbb{C}}(\ker \slashed{D})}{2}-\frac{\eta(\tilde{\slashed{D}}^{})}{2},
\end{equation}
where $p_{1}(\hat{A})$ is the first Pontryagin form for $\hat{A}$.

On the other hand, the operator $\hat{\slashed{D}}^{+}$ can be written as $\frac{d}{ds}+\slashed{D}^{s}$ under suitable bundle identification. Then, the index equals the spectral flow
\begin{equation}\label{index formula 2}\text{Ind}_{\mathbb{C}}(\hat{\slashed{D}}^{+})=-(\# \text{ eigenvalues of }\slashed{D} \text{ in } [-\delta_r,0]) =-m(\slashed{D} , \delta_{r}),
\end{equation}
where we note that $-\delta_{r}$ is not an eigenvalue of $\slashed{D}$.
Recall that $$n(Y,\mathfrak{s}_{q},A_{0},g_r) = \frac{\eta(\slashed{D})-\text{dim}_{\mathbb{C}}(\ker
\slashed{D})}{2} + \frac{\eta_{\text{sign}}}{8} .$$
Combining this with (\ref{index formula}) and (\ref{index formula 2}), we get
$$
m(\slashed{D} , \delta_{r})+n(Y,\mathfrak{s}_{q},A_{0},g_r)=\frac{\eta(\tilde{\slashed{D}})}{2}+\frac{1}{24}\int_{X}p_{1}(\hat{A})+ \frac{\eta_{\text{sign}}}{8} .
$$
Here we assume that $r$ is sufficiently small and apply Lemma 2.3, formula (2.22)  and Appendix A of \cite{Nicolescu1} to obtain
\begin{align*}
\eta( \tilde{\slashed{D}}) &=\frac{d}{6}+\frac{(g-1-q)(d+g-1-q)}{d},\\
\frac{1}{24}\int_{X}p_{1}(\hat{A})&=\frac{d}{12}(d^{2}r^{4}-(2-2g)r^{2}).
\end{align*}
From \cite[p.108]{Nicolescu1} and \cite{Ouyang}, we also have
$$
\frac{\eta_{\text{sign}}}{8}=\frac{d-3}{24}-\frac{d}{12}(d^{2}r^{4}-(2-2g)r^{2}).
$$
Combining the above formulas, we can conclude
$$n(Y,\mathfrak{s}_{q},A_{0},g_r)+m(\delta_r)=\frac{d-1}{8}+\frac{(g-1-q)(d+g-1-q)}{2d}.$$
In summary, for $0<g\leq q<d$, we have
$$\underline{\text{\text{SWF}}}^{A}(S(N_{d})),\mathfrak{s_{q}};S^{1})\cong (S^{0},0,c(g,d,q)),\; \underline{\text{\text{SWF}}}^{R}(S(N_{d})),\mathfrak{s_{q}};S^{1})\cong (S^{0},0,c(g,d,q)),$$
where $c(g,d,q) $ is the number $\frac{d-1}{8}+\frac{(g-1-q)(d+g-1-q)}{2d} $ appeared above.


\subsection{Circle bundles over torus and other nil manifolds}\label{subsection nil manifolds}
When $Y=S(N)$ where $p\colon N\rightarrow \Sigma$ is a complex line bundle over an orbifold $\Sigma$ with $\chi(\Sigma)=0$ and $\deg (N)=d\neq 0$, the manifold $Y$ supports the so-called ``nil-geometry''. For simplicity, will assume that $d>0$ and $\Sigma$ is orientable. The case $d<0$ is similar (see the remark after formula (\ref{swfr for nil})) and the nonorientable case can be done by passing to
a suitable (orientable) double cover. Our main focus will be the case when $\Sigma $ is smooth, i.e. $Y$ is a circle bundle over the torus.

\subsubsection{Preparation} Let $Y$ be a circle bundle over torus. In this case, the torsion
spin$^{c}$ structures of $Y$ can be parametrized by $\mathbb{Z}/d\mathbb{Z}$ denoted by $\mathfrak{s}_{0}, \ldots,\mathfrak{s}_{d-1}$.
The spectrum invariants for $\mathfrak{s}_{q}$ with $1\leq q\leq d-1$ are already calculated in
Section \ref{subsection large degree circle bundles} and we will only consider the spin$^c$ structure $\mathfrak{s}_0 $ here. As earlier, $Y$ is equipped with a canonical metric $g_r$ and a canonical spinor bundle $S^0_Y $.

Let the torus $\Sigma$ be given by $\left( \mathbb{R}/\sqrt{\pi}\mathbb{Z} \right) \times \left( \mathbb{R}/\sqrt{\pi}\mathbb{Z}\right)$ with coordinate $(x_1 , x_2)$.
To identify $i\Omega^{1}_{h}(Y)$ with $\mathbb{R}^{2}$, we choose harmonic forms $h_{j}$ to be $2i\sqrt{\pi}\, p^{*}(dx_{j})$. We have a canonical trivialization of $K_{\Sigma}$, which also induces a trivialization of $S_{Y}^{0}\cong \underline{\mathbb{C}}\oplus p^{*}(K_{\Sigma}^{-1})$ and a trivialization of $\text{det}(S_{Y}^{0}) \cong \underline{\mathbb{C}}$, where $\underline{\mathbb{C}}$ is the trivial vector bundle on $Y$ with fiber $\mathbb{C}$.  Under this trivialization, the Clifford multiplication is given by
$$\rho(h_{1})=\left(\begin{array} {lr}
 0 & -2i\sqrt{\pi} \\
 2i\sqrt{\pi}& 0
\end{array}\right)\text{, }\rho(h_{2})=\left(\begin{array} {lr}
 0 & -2\sqrt{\pi} \\
 -2\sqrt{\pi} & 0
\end{array}\right).
$$
We set the base connection $A_{0}$ to be the connection which is spinorial with respect to $\nabla^{LC}$ and induces the trivial connection on $\text{det}(S_{Y}^{0})$. Then the corresponding connection $\tilde{A}_{0}$ spinorial with respect to $\nabla^{\circ}$ is just the trivial connection on $S_{Y}^{0}$.

Let $\Gamma_{c}(S_{Y}^{0})$ be the subspace of $\Gamma(S_{Y}^{0})$ consisting of sections which are constant along each fiber of $Y = S(N)$. We see that  $\Gamma_{c}(S_{Y}^{0})$ is the same as a space of function from $\mathbb{T}^2 $ to $\mathbb{C}^2 $, so Fourier series give an $L^2$-orthogonal decomposition
\begin{equation} \label{eq decomp consectionV}
\Gamma_{c}(S_{Y}^{0}) = \bigoplus_{\vec{v}\in \mathbb{Z}^{2}}V_{\vec{v}},
\end{equation}
where $V_{\vec{v}} $ is the two-dimensional vector space spanned by $\phi_{\vec{v}, +}(x) = (e^{2\sqrt{\pi}i\langle x,\vec{v}\rangle},0) $
and $\phi_{\vec{v}, -}(x)  = (0,e^{2\sqrt{\pi}i\langle x,\vec{v}\rangle}) $. We also have an orthogonal decomposition
\begin{equation} \label{eq decomp consectionC0}
 \Gamma_{}(S_{Y}^{0}) = \Gamma_{c}(S_{Y}^{0}) \oplus \Gamma_{0}(S_{Y}^{0}),
\end{equation}
where $\Gamma_{0}(S_{Y}^{0}) $ is the subspace of $\Gamma_{}(S_{Y}^{0}) $ consisting of sections which integrate to 0 along each fiber of $S(N)$. We have the following observations regarding the Dirac operators and these subspaces.

\begin{lem}\label{dirac operator acting on spherical harmonics}
For any $h \in i\Omega^{1}_{h}(Y) $, we consider Dirac operators $\tilde{\slashed{D}}_{\tilde{A}_{0}+h}=\slashed{D}_{A_{0}+h}+\delta_{r}$.
\begin{enumerate}
\item  The operator $\tilde{\slashed{D}}_{\tilde{A}_{0}+h}$ preserves the subspaces $\Gamma_{0}(S_{Y}^{0})$ and ${V}_{\vec{v}}$;
\item \label{item property diracnil2} The operator $\tilde{\slashed{D}}_{\tilde{A}_{0}+h}$ has no kernel when restricted to $L^{2}_{k}(\Gamma_{0}(S_{Y}^{0}))$;
\item  When $h = \theta_1 h_1 + \theta_2 h_2$ and $\vec v = (v_1 , v_2)$, the operator $\tilde{\slashed{D}}_{\tilde{A}_{0}+h}$ restricted to ${V}_{\vec{v}}$ is given by a matrix
$$
\left(\begin{array} {lr}
 0 & -2\sqrt{ \pi}(  \theta_{2} + v_{2}) - 2\sqrt{\pi}(\theta_{1}+v_{1})i\\
 -2\sqrt{ \pi}(\theta_{2}+v_{2})  + 2\sqrt{\pi}(\theta_{1}+v_{1})i  & 0
\end{array}\right).
$$
\end{enumerate}
\end{lem}
\begin{proof} Statement~(\ref{item property diracnil2}) is implied by the proof of Proposition~5.8.4 of \cite{MOY}.
The other statements can be verified by simple calculation.
\end{proof}

As in Section \ref{section linearized flow}, we will consider a functional of the form
$$
\mathcal{L}_{\delta, \epsilon_{}} :=CSD+\frac{\delta}{2}\|\phi\|_{L^{2}}^{2}+\epsilon f_H({p}_\mathcal{H} (a,\phi)),
$$
where $f_H (\theta_1 , \theta_2 ) = -\frac{1}{2\sqrt{2}\pi}(\cos{2\pi\theta_1}+\cos{2\pi\theta_2}) $. From Proposition~\ref{Some Seifert space only have reducible critical points}, the functional $\mathcal{L}_{\delta_{r},0}$ only has reducible critical points (recall that  $\delta_{r}=\frac{ 1}{2}rd$).
Although, the operator $\slashed{D}_{A_{0}+h}+\delta_{r} $ has nontrivial kernel for some $h \in i\Omega^{1}_{h}(Y)$, we  can construct a suitable perturbation on $\mathcal{L}_{\delta_{r},0}$ which allows us to apply Theorem~\ref{spectrum for the linearized flow}. The first step is the following lemma, whose proof is almost identical to the proof of
\cite[Proposition~37.1.1]{Kronheimer-Mrowka}, and we omit it.
\begin{lem}\label{lem delepsilon reducible}
There exists $\delta_{1}>0$ such that, for any $\delta\in (0,\delta_{1})$, there exists $\epsilon_{1}(\delta)>0$ so that the functional
$$\mathcal{L}_{\delta_{r}-\delta,\epsilon}=\mathcal{L}_{\delta_{r},0}-\frac{\delta}{2}\|\phi\|_{L^{2}}^{2}+\epsilon f_H({p}_\mathcal{H} (a,\phi))$$
has only reducible critical points for all $\epsilon \in (0,\epsilon_{1}(\delta)) $.
\end{lem}
We also have the following lemma:
\begin{lem}\label{achieve positive definite condition}
There exists $\delta_{2}>0$ such that, for any $\delta\in (0,\delta_{2})$, there exists $\epsilon_{2}(\delta)>0$ so that
we have
\begin{equation}\label{positive definite condition (2)}
\langle (2(\tilde{\slashed{D}}_{\tilde{A}_{0}+h}-\delta)^{2}+\epsilon \rho( \grad f_H(h)))\phi,\phi\rangle_{L^{2}}\geq C(\delta,\epsilon)\|\phi\|^{2}_{L^{2}},
\end{equation}
for all $\epsilon \in (0,\epsilon_{2}(\delta))$, $h\in i\Omega^{1}_{h}(Y)$ and $\phi\in L^{2}_{k}(\Gamma(S^{0}_{Y}))$, and $C(\delta,\epsilon)$ is a positive constant depending only on $\delta,\epsilon$.
\end{lem}
\begin{proof}
Since we have an $L^{2}$-orthogonal decomposition
$$\Gamma(S_{Y}^{0})= \Gamma_{0}(S_{Y}^{0}) \oplus \bigoplus_{\vec{v}\in \mathbb{Z}^{2}}V_{\vec{v}} ,$$
which is preserved by both $\tilde{\slashed{D}}_{\tilde{A}_{0}+h}-\delta$ and $\rho( \grad f_H(h))$, we just need to prove the statement on each of the summand.

The case $ L^{2}_{k}(\Gamma_{0}(S_{Y}^{0}))$ is easy. By Lemma~\ref{dirac operator acting on spherical harmonics} and the compactness of $i\Omega^{1}_{h}(Y)/\mathcal{G}^{h,o}_{Y}$, we can find a constant $\sigma_{1}>0$ such that, for any $h\in i\Omega^{1}_{h}(Y)$, the restriction of $\tilde{\slashed{D}}_{\tilde{A}_{0}+h}$ to $L^{2}_{k}(\Gamma_{0}(S_{Y}^{0}))$ has no eigenvalue lying in $(-2\sigma_{1},2\sigma_{1})$. Therefore, for any $0 <\delta < {\sigma_1}$ and $0<\epsilon < \sigma_1^2$, we have

$$
\langle (2(\tilde{\slashed{D}}_{\tilde{A}_{0}+h}-\delta)^{2}+ \epsilon \rho( \grad f_H(h)))\phi,\phi\rangle_{L^{2}}\geq \sigma_{1}^{2}\|\phi\|_{L^{2}}^{2}, \; \forall \phi\in L^{2}_{k}(\Gamma_{0}(S_{Y}^{0})).
$$

For the subspace $V_{\vec{v}}$,  we first check the case $\vec{v} = \vec0 $.
For $h=\theta_{1}h_{1}+\theta_{2}h_{2}$, by Lemma~\ref{dirac operator acting on spherical harmonics}, the matrix of $(\tilde{\slashed{D}}_{\tilde{A}_{0}+h}-\delta)|_{V_{\vec{0}}}$ is given by
$$
\left(\begin{array} {lr}
 -\delta & -2\sqrt{\pi}(\theta_{2} + \theta_{1}i)\\
 -2\sqrt{\pi}(\theta_{2}  - \theta_{1}i) & -\delta
\end{array}\right).
$$
The eigenvalues of this matrix is $-\delta\pm 2\sqrt{\pi(\theta_{1}^{2}+\theta_{2}^{2})}$. The kernel of $\tilde{\slashed{D}}_{\tilde{A}_{0}+h}-\delta$ is nontrivial if and only if $h$ is on the sphere $S_{\vec{0}}$ of radius $\frac{\delta}{2\sqrt{\pi}}$  centered at $0$. Let $S(V_{\vec{0}})$ be the unit sphere in $V_{\vec 0}$.
We notice that
\begin{equation}\label{positive definite condition (4)}
\langle 2(\tilde{\slashed{D}}_{\tilde{A}_{0}+h}-\delta)^{2}\phi,\phi\rangle_{L^{2}}\geq 0
\text{ for any }(h,\phi)\in i\Omega^{1}_{h}(Y)\times S(V_{\vec{0}})
\end{equation}
where equality holds if and only if $(h,\phi)$ belongs to the compact set
$$K:=\{(h,\phi)\in S_{\vec{0}}\times S(V_{\vec{0}}) \mid (\tilde{\slashed{D}}_{\tilde{A}_{0}+h}-\delta)\phi=0\}.$$

For any $(h,\phi) \in K$, we consider the negative gradient flow line $\gamma$ of $f_H$ with $\gamma(0) = h$.
Here we choose $\delta<\sqrt{\pi}$ so that $ \left\langle h,  \gamma' (0)\right\rangle < 0 $, i.e. $\gamma$ goes inside $S_{\vec{0}}$ at $h$. For $t>0$, we see that the operator
$(\tilde{\slashed{D}}_{\tilde{A}_{0}+\gamma(t)}-\delta)|_{V_{\vec{0}}}$
has eigenvalues $-\delta\pm 2\sqrt{\pi}\left\Vert\gamma(t)\right\Vert$, which are both negative. Then, the function
$-\delta+ 2\sqrt{\pi}\left\Vert\gamma(t)\right\Vert - \langle (\tilde{\slashed{D}}_{\tilde{A}_{0}+\gamma(t)}-\delta)\phi,\phi\rangle $ is nonnegative for $t\ge0$. Since the value of this function is 0 at $t=0$, its derivative at $t=0$ must be nonnegative as well. After computing the derivative, we have $\langle\rho( \grad f_H(h)))\phi,\phi\rangle_{}>-\left\langle h,  \gamma' (0)\right\rangle $. Therefore, we can conclude that
\begin{equation}\label{positive definite condition (3)}
          \langle\rho( \grad f_H(h)) \phi,  \phi  \rangle_{L^{2}}>0 \text{ for any }(h,\phi)\in K.
\end{equation}
Now we can find a small neighborhood $U\subset  i\Omega^{1}_{h}(Y)\times S(V_{\vec{0}})$ of $K$ and a positive number
$\sigma_{2}$ such that
$$
        \langle\rho( \grad f_H(h))  \phi, \phi\rangle_{L^{2}}  > \sigma_2 \text{ for any }(h,\phi)\in U.
$$
Similar to the case $ L^{2}_{k}(\Gamma_{0}(S_{Y}^{0}))$, there exists $\sigma_3>0$ such that $$
\langle 2(\tilde{\slashed{D}}_{\tilde{A}_{0}+h}-\delta)^{2}\phi,\phi\rangle_{L^{2}}\geq 2\sigma_3^2,$$
for any $\delta \in (0,\sigma_3) $ and $(h,\phi)\in i\Omega^{1}_{h}(Y)\times S(V_{\vec{0}})\setminus U.$
Let $\epsilon > 0$ be a positive number such that
\[
       \epsilon \left| \left< \rho(\grad f_H (h)) \phi, \phi \right>_{L^2} \right| < \sigma_3^2
\]
for any $(h, \phi) \in i\Omega^1_h(Y) \times S(V_{\vec{0}})$. Then
$$
      \langle (  2(\tilde{\slashed{D}}_{\tilde{A}_{0}+h} -  \delta)^{2} + \epsilon \rho( \grad f_H(h)))\phi,  \phi  \rangle_{L^{2}}
                    \geq
            \epsilon \sigma_2
$$
for any $(h,\phi)\in i\Omega^{1}_{h}(Y)\times S(V_{\vec{0}})$. By applying elements in $\mathcal{G}^{h,o}$, we see that similar estimate (with the same constants) holds for general $V_{\vec{v}}$. This finishes the proof of the lemma.
\end{proof}
\begin{rmk}
For any function $\xi:Y\rightarrow \mathbb{R}$, by applying the gauge transformation $u=e^{-i\xi}$ on (\ref{positive definite condition (2)}), we also get $\langle (2(\tilde{\slashed{D}}_{\tilde{A}_{0}+h+id\xi}-\delta)^{2} + \epsilon \rho( \grad f_{H}(h)))\phi,\phi\rangle_{L^{2}}>C(\delta,\epsilon)\|\phi\|^{2}_{L^{2}}$. This observation will be useful soon.
\end{rmk}

We now fix a choice of constants $(\delta_{},\epsilon_{})$ with $0<\delta_{}<\min(\delta_{1},\delta_{2})$ and $0<\epsilon_{}<\min(\epsilon_{1}(\delta_{}),\epsilon_{2}(\delta_{}))$, where $\delta_j$ and $\epsilon_j(\delta)$ are the constants from Lemma~\ref{lem delepsilon reducible} and Lemma~\ref{achieve positive definite condition}.
All critical points $\mathcal{L}_{\delta_{r}-\delta_{},\epsilon_{}}$ are then reducible, so they are of the form $(h_{p,q},0)\in Coul(Y)$  where $h_{p,q}=\tfrac{ph_{1}+qh_{2}}{2}$ for each $p,q \in \mathbb{Z} $. Modulo the action of the whole gauge group, there are four equivalent classes: $[(h_{0,0},0)],[(h_{0,1},0)],[(h_{1,0},0)]$ and $[(h_{1,1},0)]$. The relative gradings between them are given by
\begin{equation}\label{relative grading}
\begin{split}&\grade ([(h_{1,1},0)],[(h_{1,0},0)])=\grade ([(h_{1,1},0)],[(h_{0,1},0)])=1,\\
&\grade ([(h_{1,0},0)],[(h_{0,0},0)])=\grade ([(h_{0,1},0)],[(h_{0,0},0)])=-1.
\end{split}
\end{equation}
Notice that $\grade ([(h_{1,0},0)],[(h_{0,0},0)])$ and $\grade ([(h_{0,1},0)],[(h_{0,0},0)])$ does not coincide with the relative grading for the Morse function $f_H$ (which is $1$). This is because of the appearance of the spectral flow.

There may still be finite type irreducible trajectories between these critical points. Our next aim is to find a suitable perturbation $\bar{f}$ to eliminate these trajectories. To do this, we need to use the results regarding the ``blown-up moduli space''. We quickly describe the situation in our case and refer to \cite{Kronheimer-Mrowka} for the Morse case and \cite{Francesco} for the Morse-Bott case.

The blown-up quotient configuration space $\mathcal{B}^{\sigma}(Y)$ is obtained by blowing up the  quotient space $\mathcal{B}(Y)=\mathcal{C}(Y)/\mathcal{G}_{Y}$ along the reducible locus. The vector field $-\grad  \mathcal{L}_{\delta_{r}-\delta_{},\epsilon_{}}$ can be lifted to a vector field on $\mathcal{B}^{\sigma}(Y)$, which we denote by $-\grad ^{\sigma}\mathcal{L}_{\delta_{r}-\delta_{},\epsilon_{}}$. The set of zero locus of this vector field is given by the union of the critical manifolds $[C_{p,q,\mu}]$
where $\mu$ is an eigenvalue of $\tilde{\slashed{D}}_{\tilde{A}_{0}+h_{p,q}}-\delta_{}$ and $p,q \in \{0,1\}$. The critical manifold $[C_{p,q,\mu}]$, lies in the preimage of $[(h_{p,q},0)]$ under the blow-down map, has (real) dimension $2(\text{multiplicity of } \mu)-2$. This manifold is called boundary stable (resp. boundary unstable) if $\mu>0$ (resp. $\mu<0$). Note that $\tilde{\slashed{D}}_{\tilde{A}_{0}+h_{p,q}}-\delta_{}$ has no kernel.

Denote by $\breve{\mathcal{M}}([C_{p,q,\mu}],[C_{p',q',\mu'}])$ the moduli space of unparametrized trajectories going from $[C_{p,q,\mu}]$ to $[C_{p',q',\mu'}]$.
We will be interested in the expected dimension of the moduli space $\breve{\mathcal{M}}([C_{p,q,\mu}],[C_{p',q',\mu'}])$ in the case $\mu<0<\mu'$. From \cite[Chapter~2, Proposition~3.12]{Francesco} (be careful with the notation there),
we have
\begin{equation} \label{expected dimension}
\begin{split}\edim (\breve{\mathcal{M}}([C_{p,q,\mu}],[C_{p',q',\mu'}]))=&\grade ([(h_{p,q},0)],[(h_{p',q'},0)])-2\\&- 2\#\{\text{eigenvalues of } \slashed{D}_{A_{0}+h_{p,q}}+\delta_{r}-\delta_{} \text{ in } (\mu,0)\}\\&- 2\#\{\text{eigenvalues of } \slashed{D}_{A_{0}+h_{p,q}}+\delta_{r}-\delta_{} \text{ in } (0,\mu')\}\end{split}.
\end{equation}
Observe that the dimension in this case is always negative by (\ref{relative grading}). We have the following result regarding these moduli spaces.

\begin{lem}\label{irreducible trajectories can't be too close to reducible}
There exists $\epsilon'>0$ such that, for any trajectory $\gamma=(a,\phi)\colon\mathbb{R}\rightarrow Coul(Y)$ representing a point of $\breve{\mathcal{M}}([C_{p,q,\mu}],[C_{p',q',\mu'}])$ with $\mu<0<\mu'$, we have $\mathop{\sup}\limits_{t\in\mathbb{R}}\|\phi(t)\|_{L^{2}}>2\epsilon'$.
\end{lem}
\begin{proof}

First, we show that a trajectory representing a point in $\breve{\mathcal{M}}([C_{p,q,\mu}],[C_{p',q',\mu'}])$ must be irreducible. A reducible point in $\breve{\mathcal{M}}([C_{p,q,\mu}],[C_{p',q',\mu'}])$   gives rise to a finite type trajectory $(\beta , h , \phi )\colon \mathbb{R}\rightarrow i \im d^{*} \oplus i\Omega^{1}_{h}(Y)\oplus \Gamma(S_{Y}^0)$
such that
\begin{equation}\label{trajectory for a general linearized flow}
\frac{d}{dt}(\beta(t),h(t),\phi(t))=-(*d\beta(t), \epsilon \grad f_H (h(t)),(\tilde{\slashed{D}}_{\tilde{A}_{0}+h(t)}-\delta_{})\phi(t)),
\end{equation}
where $\| \phi(t) \|_{L^2} = 1$.
From Lemma~\ref{achieve positive definite condition}, we can use Lemma~\ref{trajectories for the
linearized flow are reducible} to conclude that there can be no such trajectory.

Now we can consider only irreducible trajectories. Suppose the contrary that there is a sequence of irreducible trajectories $\gamma_{n}(t)=(a_{n}(t),\phi_{n}(t))$ with $\mathop{\lim}\limits_{n\rightarrow +\infty}( \sup \limits_{t \in \mathbb{R}} \|\phi_{n}(t)\|_{L^{2}}) = 0$. As in the proof of Theorem~\ref{Morse-Bott reducible}, we can lift $\gamma_{n}$ to a negative gradient flow line $\tilde{\gamma}_{n}=(\tilde{a}_{n},\tilde{\phi}_{n})\colon\mathbb{R}\rightarrow \mathcal{C}_{Y}$  of $ \mathcal{L}_{\delta_{r}-\delta_{0},\epsilon_{}}$. Notice that the topological energy of $\tilde{\gamma}_{n}$ is bounded above by $\tfrac{\sqrt{2}\epsilon_{}}{\pi}$. By \cite[Proposition 16.2.1]{Kronheimer-Mrowka}, after passing to a subsequence and applying suitable gauge transformations, the sequence $\tilde{\gamma}_{n}$ converges to a (possibly broken) flow line $\tilde{\gamma}_{\infty}$. From our assumption on the limit of $\sup \limits_{t \in \mathbb{R}} \|\phi_{n}(t)\|_{L^{2}}$, the trajectory $\tilde{\gamma}_{\infty}$ is reducible. Any such reducible flow line is contained in $L^{2}_{k}(\ker d)\times \{0\} $. In particular, if we decompose $\tilde{a}_{n}(t)$ as $h_{n}(t)+\beta_{n}(t) + id\xi_{n}(t) \in i\Omega^{1}_{h}(Y) \oplus \im d^{*} \oplus id \Omega^0(Y) $,
then we have $\mathop{\lim}\limits_{n\rightarrow +\infty}(\sup \limits_{t \in \mathbb{R}}  \|\beta_{n}(t)\|_{L^{2}_{k}})=0$.
This also implies $\mathop{\lim}\limits_{n\rightarrow +\infty}(\sup \limits_{t \in \mathbb{R}} \|d\tilde{a}_{n}(t)\|_{C^{0}})=0$.

From (\ref{eq 2deriv upstair}), we have
\begin{equation}\begin{split}
\frac{d^2}{dt^{2}}\|\phi_{n}(t)\|_{L^{2}}^{2}&=
2\langle\left(2(\tilde{\slashed{D}}_{\tilde{A}_{0}+\tilde{a}_{n}(t)}-\delta)^{2}+\rho( \grad \epsilon f_H(h_{n}(t)))\right)\tilde{\phi}_{n}(t),\tilde{\phi}_{n}(t)\rangle_{L^{2}}\\
& \quad +\|\tilde{\phi}_{n}(t)\|_{L^{2}}^{4}+2\langle\rho(*d\tilde{a}_{n}(t))\tilde{\phi}_{n}(t),\tilde{\phi}_{n}(t)\rangle_{L^{2}}.
\end{split}
\end{equation}
Notice that
\begin{equation*}
\begin{split}
\langle (\tilde{\slashed{D}}_{\tilde{A}_{0}+\tilde{a}_{n}(t)}-\delta_{})^{2}\tilde{\phi}_{n}(t),\tilde{\phi}_{n}(t)\rangle_{L^2}
&=\|(\tilde{\slashed{D}}_{\tilde{A}_{0}+\tilde{a}_{n}(t)}-\delta_{})\tilde{\phi}_{n}(t) \|^2_{L^2} \\
&= \|(\tilde{\slashed{D}}_{\tilde{A}_{0}+h_{n}(t)+id\xi_{n}(t)}-\delta_{})\tilde{\phi}_{n}(t)+\rho(\beta_{n}(t))\tilde{\phi}_{n}(t) \|_{L^2}^{2}\\
&\geq \langle(\tilde{\slashed{D}}_{\tilde{A}_{0}+h_{n}(t)+id\xi_{n}(t)}-\delta_{})^{2}\tilde{\phi}_{n}(t),\tilde{\phi}_{n}(t)\rangle_{L^2}-\|\beta_{n}(t)\|^{2}_{C^{0}}\|\tilde{\phi}_{n}(t)\|_{L^2}^{2}.
\end{split}
\end{equation*}
This gives
\begin{equation}\begin{split}
\frac{d^2}{dt^{2}}\|\phi_{n}(t)\|_{L^{2}}^{2}&\geq
2\langle(2(\tilde{\slashed{D}}_{\tilde{A}_{0}+h_{n}(t)+id\xi_{n}(t)}-\delta_{})^{2}+\rho( \grad \epsilon f_H(h_{n}(t))))\tilde{\phi}_{n}(t),\tilde{\phi}_{n}(t)\rangle_{L^2}\\
& \quad +\|\tilde{\phi}_{n}(t)\|_{L^2}^{4}+2\langle\rho(*d\tilde{a}_{n}(t))\tilde{\phi}_{n}(t),\tilde{\phi}_{n}(t)\rangle_{L^{2}}-4\|\beta_{n}(t)\|^{2}_{C^{0}}\|\tilde{\phi}_{n}(t)\|_{L^{2}}^{2}\\
&\geq \left(2C(\delta,\epsilon_{3})-2\|d\tilde{a}_{n}(t)\|_{C^{0}}-4\|\beta_{n}(t)\|^{2}_{C^{0}}\right)\|\tilde{\phi}_{n}(t)\|_{L^{2}}^{2},
\end{split}
\end{equation}
where we make use of remark after Lemma~\ref{achieve positive definite condition}.
Therefore, when $n$ is sufficiently large, the above inequality implies that  $\frac{d^2}{dt^{2}}\|\phi_{n}(t)\|_{L^{2}}^{2}>0$ for any $t$. This is impossible because $\sup \limits_{  t \in \mathbb{R} }  \|\phi_{n}(t)\|_{L^{2}}^{2}<\infty$.
\end{proof}

To construct final perturbation, we reintroduce the Banach space $\mathcal{P} $ of extended cylinder functions (see Section \ref{sec setup}). Define a subset $O=\{(a,\phi)\in \mathcal{C}_{Y} \mid \|\phi\|_{L^{2}}<\epsilon'\}$ where $\epsilon'$ is the constant from Lemma~\ref{irreducible trajectories can't be too close to reducible} and
a closed subspace $\mathcal{P}_{O}=\{\bar{f}\in \mathcal{P} \mid \bar{f}|_{O}\equiv
0\}$ of $\mathcal{P}$.  By \cite[Proposition~11.6.4]{Kronheimer-Mrowka}, we can find an open neighborhood $\mathcal{U}$ of $0$ in $\mathcal{P}_{O}$
such that for any $\bar{f}$ in this neighborhood, the functional $\mathcal{L}_{\delta_{r}-\delta_{},\epsilon_{}}+\bar{f}$
has no critical points outside $O$. Therefore, the critical points are just $(h_{p,q},0)$ for $p,q\in \mathbb{Z}$ with the corresponding critical manifolds $[C_{p,q,\mu}]$ as in the case of $\mathcal{L}_{\delta_{r}-\delta_{},\epsilon_{}}$. Analogously, we denote by $\breve{\mathcal{M}}_{\bar{f}}([C_{p,q,\mu}],[C_{p',q',\mu'}])$ the moduli space of trajectories of
$\grad ^{\sigma}(\mathcal{L}_{\delta_{r}-\delta_{},\epsilon_{}}+\bar{f})$.

\begin{lem}\label{achieve Smale transversality}
For any pair of critical manifolds $[C_{p,q,\mu}]$ and $[C_{p',q',\mu'}]$ with $\mu<0<\mu'$, there exists a residue subset of $\mathcal{P}_{O}$ such that the moduli space $\breve{\mathcal{M}}_{\bar{f}}([C_{p,q,\mu}],[C_{p',q',\mu'}])$ is empty.
\end{lem}
\begin{proof}
Using the fact that the index of a Fredholm operator does not change under homotopy, it is easy to see that the expected dimension of $\breve{\mathcal{M}}_{\bar{f}}([C_{p,q,\mu}],[C_{p',q',\mu'}])$ does not depend on $\bar{f}$. In particular, it is always negative when $\mu<0<\mu'$ by our discussion earlier. Therefore, we just need prove that $\breve{\mathcal{M}}_{\bar{f}}([C_{p,q,\mu}],[C_{p',q',\mu'}])$ is Smale regular (c.f. \cite[Definition~3.11]{Francesco}). The proof of this fact is very similar to the proof of
\cite[Theorem~15.1.1]{Kronheimer-Mrowka} and \cite[Theorem~3.1.7]{Francesco}, i.e. one introduces a parametrized moduli space and then applies the Sard-Smale theorem.

The main difference is that we require an extended cylinder function to satisfy $\bar{f}(a,\phi)=0$ as long as $\|\phi\|_{L^{2}}<\epsilon'$ instead of asking $\bar{f}$ to vanish in a small neighborhood of the critical manifolds. To see that this new requirement does not affect the result, we recall how the cylinder functions are constructed (c.f. \cite{Kronheimer-Mrowka}): after choosing a set of sections $\Phi_{1},...,\Phi_{m}$ of the bundle $\tilde{S}_{Y}:=(i\Omega^{1}_{h}(Y)\times S_{Y})/\mathcal{G}^{h,o}_{Y}$ over $Y\times (H^{1}(Y; \mathbb{R})/H^{1}(Y; \mathbb{Z}))$ and a set of coexact forms $a_{1},...,a_{n}$, we get a map
$$p_{0} \colon \mathcal{C}_{Y}\rightarrow \mathbb{C}^{m}\times (H^{1}(Y; \mathbb{R})/H^{1}(Y; \mathbb{Z}))\times \mathbb{R}^{n}.$$
 A cylinder function is obtained by composing $p_{0}$ with a compact supported function on $\mathbb{C}^{m}\times (H^{1}(Y; \mathbb{R})/H^{1}(Y; \mathbb{Z}))\times \mathbb{R}^{n}$.

On the other hand, it is straightforward to deduce from Lemma~\ref{irreducible trajectories can't be too close to reducible} that
any trajectory $\gamma = (a,\phi) \colon \mathbb{R}\rightarrow \mathcal{C}_{Y}$ representing a point in
$\breve{\mathcal{M}}_{\bar{f}}([C_{p,q,\mu}],[C_{p',q',\mu'}])$ also satisfies $\sup \limits_{t \in \mathbb{R}}  {\|\phi(t)\|_{L^{2}}}>\epsilon'$ for $\bar{f} \in \mathcal{P}_O $. We now pick a section of the bundle $\tilde{S}_{Y}$ which equals $\frac{\phi(t_{0})}{\|\phi(t_{0})\|_{L^{2}} }$ when restricted to $Y\times \{\pi_\mathcal{H} a(t_{0})\}$ and whose restriction to any $Y\times \{*\}$ has unit $L^{2}$-norm. By specifying $\Phi_{1}$ to be this section, we see that the image $p_0 (O)$ lies in a set
$$
 U:=B(\epsilon')\times \mathbb{C}^{m-1}\times (H^{1}(Y; \mathbb{R})/H^{1}(Y; \mathbb{Z}))\times \mathbb{R}^{n},
 $$
where $B(\epsilon')$ is the ball of radius $\epsilon'$ in $\mathbb{C}$.
Therefore, by composing $p_{0}$ with a function which vanishes on $U$, we get a cylinder function which vanishes on $O$. Moreover, since $p_{0}(\gamma)\nsubseteq U$, this kind of cylinder functions are enough to repeat the argument on Page 269 of \cite{Kronheimer-Mrowka} and prove the transversality result we need.
\end{proof}

The following result is now immediate from the previous lemma.

\begin{pro} \label{eliminate the irreducible trajectories}
There exists an extended cylinder function $\bar{f}\in \mathcal{P}_O$ such that all finite type gradient flow lines of  the functional $\mathcal{L}_{\delta_{r}-\delta_{},\epsilon_{}}+\bar{f}$ are contained in $i\Omega^{1}_{h}(Y)$.
\end{pro}
\begin{proof}
Notice that any finite type irreducible gradient flow line of $\mathcal{L}_{\delta_{0}-\delta,\epsilon_{3}}+\bar{f}$
gives a point in $\breve{\mathcal{M}}_{\bar{f}}([C_{p,q,\mu}],[C_{p',q',\mu'}])$ with $\mu<0<\mu'$. By Lemma~\ref{achieve Smale transversality}, the moduli space is empty, so there are only reducible flow lines, which have to lie in $i\Omega^{1}_{h}(Y)$.
\end{proof}
\begin{rmk}\label{negative delta,epsilon}
By reversing the orientation of $Y$ and repeating the arguments above, we see that one can also choose $\delta_{},\epsilon_{}>0$ and $\bar{f}'\in \mathcal{P}_O$ such that gradient flow lines of $\mathcal{L}_{\delta_{r}+\delta_{},-\epsilon_{}}+\bar{f'}$ lie in $i\Omega^{1}_{h}(Y)$. This observation will be useful when we calculate $\underline{\text{\text{SWF}}}^{R}$.
\end{rmk}

\subsubsection{Computation of the invariants}
We will first consider the case $\underline{\text{\text{SWF}}}^{A}(Y,\mathfrak{s}_{0};S^{1})$. Proposition~\ref{eliminate the irreducible trajectories} and Lemma~\ref{achieve positive definite condition} allow us to apply Theorem~\ref{spectrum for the linearized flow} to work on the linearized flow  given by
 \begin{equation*}
-\frac{d}{dt}\left(\beta(t),h(t),\phi(t)\right)=\left(*d\beta(t),  \grad \epsilon  f_H(h(t)),\mathbf{D}(h(t))\phi(t)\right),
\end{equation*}
where  $\mathbf{D}(h)=\tilde{\slashed{D}}_{\tilde{A}_{0}+h}-\delta_{}$ with $\delta , \epsilon $ and $f_H$ previously chosen.
Recall that we will have to compute Conley indices of finite dimensional approximation of the flow on the bounded region $\bar{J}_{m}^{+}=\widebar{Str}({\epsilon}')\cap p_\mathcal{H}^{-1}([-\frac{1}{4}-m,\frac{1}{4}+m]^{2})$, where $\epsilon' $ is from Lemma~\ref{irreducible trajectories can't be too close to reducible} and we fix $ \theta^+ = \frac{1}{4}$ (see Section \ref{section linearized flow}).

To simplify the calculation of the corresponding Conley index, we consider subspaces$$
W=\bigoplus_{\vec{v}\in \mathbb{Z}^2 \cap [-m,m]^{2}}V_{\vec{v}}, \quad
W'=\bigoplus_{\vec{v}\in \mathbb{Z}^2 \setminus [-m,m]^{2}}V_{\vec{v}}
$$
and note that, from (\ref{eq decomp consectionV}) and (\ref{eq decomp consectionC0}), we have an orthogonal decomposition $\Gamma_{}(S_{Y}^{0}) = W \oplus W' \oplus \Gamma_{0}(S_{Y}^{0}) $. To deform $\mathbf{D}(h) $, we define a family of operators $ \mathbf{D}^s (h) $ parametrized by $[0,1] \times i\Omega^1_h (Y)$ as following
$$
\mathbf{D}^{s}(h)\phi=\left\{\begin{array}{l l}
   (\tilde{\slashed{D}}_{\tilde{A}_{0}+h}-\delta_{})\phi & \quad \text{if }\phi\in W,\\
   (\tilde{\slashed{D}}_{\tilde{A}_{0}+sh}-\delta_{})\phi & \quad \text{if }\phi\in W' \oplus \Gamma_{0}(S_{Y}^{0}) .\\
  \end{array} \right.
$$
Notice that for any $h\in [-\frac{1}{4}-m,\frac{1}{4}+m]^{2}$, the operator $\mathbf{D}^{s}(h)$, when restricted to $W' \oplus \Gamma_{0}(S_{Y}^{0})$, has no eigenvalue in $[-\epsilon_{1},\epsilon_{1}]$, where $\epsilon_{1}$ is a constant independent of $m$. Therefore, by setting $\epsilon$ small and applying Lemma~\ref{deformation of linearized flow}, we can consider $\mathbf{D}^{0}(h)$ instead of $\mathbf{D}(h)$.

Fix an integer $n$ large enough so that $W\subset V^{\mu_{n}}_{\lambda_{n}}$. The approximated linearized flow on $V^{\mu_{n}}_{\lambda_{n}}$ corresponding to $\mathbf{D}^{0}(h)$ can be split into the product of the following flows:
\begin{enumerate}
\item The linear flow on $(W'\oplus \Gamma_{0}(S_{Y}^{0}))\cap V^{\mu_{n}}_{\lambda_{n}}$ given by the operator $-(\tilde{\slashed{D}}_{\tilde{A}_{0}}-\delta_{3});$
    \item The linear flow on $L^{2}_{k}(\im d^{*})\cap V^{\mu_{n}}_{\lambda_{n}}$ given by the operator $-d^{*}$;
\item The flow $\varphi$ on a finite dimensional space $i\Omega^{1}_{h}(Y)\oplus W$ generated by the vector field $(- \epsilon \grad f_H,-(\tilde{\slashed{D}}_{\tilde{A}_{0}+h}-\delta_{}))$.
\end{enumerate}
Since the first two flows are linear, the corresponding Conley indices are just $S^{0}$ after suitable desuspension.
Therefore, the stable Conley index $\bar{I}^{n,+}_{m}$ is determined by the third flow. More precisely, we have
$$
    \bar{I}^{n,+}_{m}
      =
    \Sigma^{-W^{-}}I_{S^{1}}(\varphi,\inv (\bar{J}^{+}_{m} \cap (i\Omega^{1}_{h}(Y)\oplus W))),
$$
where $W^{-}\subset W$ is spanned by the negative eigenvectors of $(\tilde{\slashed{D}}_{\tilde{A}_{0}}-\delta)|_{W}$.

By Lemma~\ref{achieve positive definite condition} and \cite[Lemma~12]{Khandhawit2}, we can deduce that
$[-\frac{1}{4}-m,\frac{1}{4}+m]^{2}\times B(W)$ is an isolating block for $\inv (\bar{J}^{+}_{m}\cap (i\Omega^{1}_{h}(Y)\oplus W))$, where $B(W)$ is the unit ball inside $W$. Moreover, by \cite[Lemma~4]{Khandhawit2}, we have
$I_{S^{1}}(\varphi,\inv (\bar{J}^{+}_{m}\cap (i\Omega^{1}_{h}(Y)\oplus W)))\cong \Sigma n^{-}$, where $n^-$ is the exit set of $[-\frac{1}{4}-m,\frac{1}{4}+m]^{2}\times B(W)$ with respect to the flow $\varphi$ and $\Sigma$ denotes the unreduced suspension. By the definition of the flow $\varphi$, we see that
$$
   n^{-}
     =
     \left\{  \left.   (h,\phi)   \in [-\frac{1}{4}-m,\frac{1}{4}+m]^{2}\times S(W) \right|
        \langle \phi, (\tilde{\slashed{D}}_{\tilde{A}_{0}+h}-\delta_{})\phi\rangle\leq 0
        \right\},
$$
where $S(W)$ is the unit sphere in $W$.

We now start to deform $n^{-}$ to simpler spaces.
Let $W^{-}(h)$ be the space spanned by nonpositive eigenvectors of $(\tilde{\slashed{D}}_{\tilde{A}_{0}+h}-\delta_{3})|_{W}$. We consider the following subset
$$
n^{-}_{1}:=\{(h,\phi)\in n^- \mid \phi \in W^{-}(h)\}.
$$
By Lemma~\ref{dirac operator acting on spherical harmonics}, the operator $(\tilde{\slashed{D}}_{\tilde{A}_{0}+h}-\delta_{})|_{W}$ can be  represented by the matrix$$\mathop{\bigoplus}\limits_{-m\leq v_{1},v_{2}\leq m}\left(\begin{array} {lr}
 -\delta_{} & -2\sqrt{ \pi}(\theta_{2}+v_{2})-2\sqrt{\pi}(\theta_{1}+v_{1})i\\
 -2\sqrt{ \pi}(\theta_{2}+v_{2})+2\sqrt{\pi}(\theta_{1}+v_{1})i& -\delta_{}
\end{array}\right).$$
Then, we see that
$$
S(W^{-}(h))=\left\{\begin{array}{l l}
   S(\mathbb{C}^{(2m+1)^{2}+1})&\quad \text{if } h\in B_{\vec{v}} \text{ for some }\vec{v} \in \mathbb{Z}^2 \cap [-m,m]^{2},\\
   S(\mathbb{C}^{(2m+1)^{2}})& \quad \text{otherwise},\\
  \end{array} \right.
$$
where $B_{\vec{v}}$ denotes the ball in $\mathbb{R}^2$ centered at $\vec{v}$ with radius $\frac{\delta_{}}{2\sqrt{\pi}}$. A careful (but elementary) check shows that $n^{-}_{1}$ is a deformation retract of $n^{-}$. Note that this deformation retraction does not preserve the projection $p_\mathcal{H}$.

To further deform $n^{-}_{1}$, we choose a point $z_0\in[-\frac{1}{4}-m,\frac{1}{4}+m]^{2}$ outside the union of balls $ \bigcup_{\vec{v}\in \mathbb{Z}^2 \cap [-m,m]^{2}} B_{\vec{v}}$ and connect $z_0$ to each of $B_{\vec{v}} $ by disjoint pathes $\gamma_{\vec{v}}$ containing in $[-\frac{1}{4}-m,\frac{1}{4}+m]^{2}$. It is not hard to see that the subset
\begin{equation}\label{deformation retract of the cube}
\Lambda := \bigcup_{\vec{v}\in \mathbb{Z}^2 \cap [-m,m]^{2}}(\gamma_{\vec{v}}\cup B_{\vec{v}})
\end{equation}
is a deformation retract of $[-\frac{1}{4}-m,\frac{1}{4}+m]^{2}$. This induces a deformation retract of $n^{-}_{1}$ to
$n^{-}_{2}:=n^{-}_{1}\cap p_1^{-1}(\Lambda)$, where
$p_1 \colon [-\frac{1}{4}-m,\frac{1}{4}+m]^{2}\times B(W) \rightarrow [-\frac{1}{4}-m,\frac{1}{4}+m]^{2}$ is the projection.
Since $\Lambda$ is homotopy equivalent to the wedge sum of $(2m+1)^{2}$ balls, we see that $n^{-}_{2}$ is homotopy equivalent to
$$n^{-}_{3}:=\bigvee^{(2m+1)^{2}}_{S(\{0\}\times\mathbb{C}^{(2m+1)^{2}})}S(\mathbb{C}^{(2m+1)^{2}+1}).$$
The above notation denotes the space obtained by gluing $(2m+1)^{2}$ copies of $S(\mathbb{C}^{(2m+1)^{2}+1})$ along their subset $S(\{0\}\times\mathbb{C}^{(2m+1)^{2}})$. Later, we will use similar notations again without explaining.
Since $W^{-}\cong \mathbb{C}^{(2m+1)^{2}+1}$, we have
$$
     \bar{I}^{n,+}_{m}
               \cong
              \Sigma^{-\mathbb{C}^{(2m+1)^{2}+1}}
               \left(
               \bigvee^{(2m+1)^{2}}\limits_{(\{0\}\times\mathbb{C}^{(2m+1)^{2}})^{+}}  (\mathbb{C}^{(2m+1)^{2}+1})^{+}
                 \right)
               \cong
                \Sigma^{-\mathbb{C}}  \left(   \bigvee^{(2m+1)^{2}}\limits_{S^{0}}\mathbb{C}^{+}   \right).
$$
To determine the attractor maps, we notice that $[-\frac{1}{4}-m+1,\frac{1}{4}+m-1]^{2}\times B(W) $ is an isolating block for $\inv (J^{+}_{m-1}\cap (i\Omega^{1}_{h}(Y)\oplus W))$ and the corresponding exiting set is just $n^{-}\cap p_1^{-1}([-\frac{1}{4}-m+1,\frac{1}{4}+m-1]^{2})$. By this observation, we see that the morphism induced by the attractor map  $\bar{I}^{n,+}_{m-1}\rightarrow \bar{I}^{n,+}_{m}$ is just given by the desuspension of the natural inclusion $\bigvee^{(2m-1)^{2}}\limits_{S^{0}}\mathbb{C}^{+}\rightarrow \bigvee^{(2m+1)^{2}}\limits_{S^{0}}\mathbb{C}^{+}$. Therefore, we have
$$
     \underline{\text{\text{SWF}}}^{A}(Y,\mathfrak{s}_{0}; S^{1})
       \cong
       \left(  \bigvee^{\infty}\limits_{S^{0}}\mathbb{C}^{+},
       0,   n(Y,\mathfrak{s}_{0},A_{0},g)+m(\slashed{D}_{A_{0}},\delta_{r}-\delta_{})+1  \right),
$$
where  $\bigvee^{\infty}\limits_{S^{0}}\mathbb{C}^{+}\in \ob\mathfrak{S}$ denotes  the following direct system
$$
                    \mathbb{C}^{+}     \rightarrow
                    \bigvee^{2}\limits_{S^{0}}\mathbb{C}^{+}  \rightarrow
                   \bigvee^{3}\limits_{S^{0}}\mathbb{C}^{+}    \rightarrow \cdots,
$$
whose connecting maps are given by the natural inclusions.

Since $\delta_{r},\delta_{}>0$ and $\delta_{}$ is small, we have $m(\slashed{D}_{A_{0}}, \delta_{r}-\delta_{})=m(\slashed{D}_{A_{0}},\delta_{r})$. We can repeat the calculation in Section \ref{subsection large degree circle bundles} to find the value of the number $n(Y,\mathfrak{s}_{0},A_{0},g)+m(\slashed{D}_{A_{0}},\delta_{r})$. The only difference is that $\slashed{D}^{0}=\tilde{\slashed{D}}_{\tilde{A}_{0}}$ now has kernel of dimension 2. Therefore, the equations (\ref{index formula}) and (\ref{index formula 2}) become
\begin{align*}
\text{Ind}_{\mathbb{C}}(\hat{\slashed{D}}^{+}) &=-\frac{1}{24}\int_{X}p_{1}(\hat{A})+\frac{\eta(\slashed{D})-\text{dim}_{\mathbb{C}}(\ker
\slashed{D})}{2}-\frac{\eta(\slashed{D}^{0})-2}{2}, \\
\text{Ind}_{\mathbb{C}}(\hat{\slashed{D}}^{+})&= -m(\delta_{r})-2.
\end{align*}
Consequently, we obtain
  $$n(\mathfrak{s}_{0},A_{0},g)+m(\delta_{r})+1=c(1,d,0)-2=\frac{d-17}{8}.$$
Therefore, we can conclude that
$$
               \underline{\text{\text{SWF}}}^{A}(Y,\mathfrak{s}_{0};S^{1})          \cong
               \left(  \bigvee^{\infty}\limits_{S^{0}} \mathbb{C}^{+}, 0, \frac{d-17}{8}    \right) \text{ for } d>0.
$$

Our next task is to compute $\underline{\text{\text{SWF}}}^{R}(Y,\mathfrak{s}_{0};S^{1})$. To have simpler description, we will replace $(\delta, \epsilon) $ in the $\underline{\text{\text{SWF}}}^{A}$ case by $(-\delta,-\epsilon)$ as discussed in Remark~\ref{negative delta,epsilon}). By setting $\theta^- = \frac{1}{4} $, we also consider
$\bar{J}^{-}_{m}=p_\mathcal{H}^{-1}([-\frac{1}{4}-m,\frac{1}{4}+m]^{2})\cap \widebar{Str}(\epsilon')$.
Notice that this linearized Seiberg-Witten flow goes outside $\bar{J}^{-}_{m}$ along $p_\mathcal{H}^{-1}(\partial[-\frac{1}{4}-m,\frac{1}{4}+m]^{2}).$
Let $W$ be as before and $\tilde{W}^{-}$ be the subspace spanned by negative eigenvectors of $(\tilde{\slashed{D}}_{A_{0}}+\delta_{})|_{W}$.
By similar argument as in the previous case, we obtain
$$
\bar{I}^{n,-}_{m}\cong \Sigma^{-(\tilde{W}^{-}\oplus i\Omega^{1}_{h}(Y))}I_{S^{1}}(\tilde{\varphi},\inv (\bar{J}^{-}_{m}\cap (i\Omega^{1}_{h}(Y)\oplus W))),
$$
where $\tilde{\varphi}$ is the flow on $i\Omega^{1}_{h}(Y)\oplus W$ generated by
$(\epsilon\grad f_H(h),(\tilde{\slashed{D}}_{\tilde{A}_{0}+h}+\delta_{})\phi)$.

Instead of finding the Conley index $I_{S^{1}}(\tilde{\varphi},\inv (\bar{J}^{-}_{m}\cap (i\Omega^{1}_{h}(Y)\oplus W)))
$ directly, we consider the reverse flow $-\tilde{\varphi}$ and use some duality results. Just like $\varphi$, the flow $-\tilde{\varphi}$ goes inside the isolating block $[-\frac{1}{4}-m,\frac{1}{4}+m]^{2}\times B(W)$ along $(\partial[-\frac{1}{4}-m,\frac{1}{4}+m]^{2})\times B(W)$. With the same argument as in the previous case, we have that$$
I_{S^{1}}(-\tilde \varphi,\inv (\bar{J}^{-}_{m}\cap (i\Omega^{1}_{h}(Y)\oplus W)))\cong \bigvee^{(2m+1)^{2}}\limits_{(\{0\}\times\mathbb{C}^{(2m+1)^{2}})^{+}}(\mathbb{C}^{(2m+1)^{2}+1})^{+}.
$$
According to \cite{Cornea} and \cite{Manolescu3} (see also \cite[{Proposition} 3]{Khandhawit2}), this space is the equivariant  $(i\Omega^{1}_{h}(Y)\oplus W)$-dual (see Page~209 of \cite{May} for definition) of $I_{S^{1}}( \tilde{\varphi},\inv (\tilde{J}^{-}_{m}\cap (i\Omega^{1}_{h}(Y)\oplus W)))$. Since we have a decomposition $$i\Omega^{1}_{h}(Y)\oplus W\cong i\Omega^{1}_{h}(Y)\oplus \tilde{W}^{-}\oplus \mathbb{C}^{(2m+1)^{2}+1},$$ we can see that
$$
        \bar{I}^{n,-}_{m}   \cong \Sigma^{\mathbb{C}} \left(   \bigvee^{(2m+1)^{2}}\limits_{S^{0}}\mathbb{C}^{+}  \right)^{*}
$$
where $E^{*}$ denotes the equivariant Spanier-Whitehead dual spectrum of $E$. In order to calculate $\left( \bigvee^{(2m+1)^{2}}\limits_{S^{0}}\mathbb{C}^{+}  \right)^{*}$, we give the following lemma, which is a simple consequence of Theorem~4.1 and Lemma~4.9 of \cite[Chapter 3]{LMS}.
\begin{lem}\label{spanier-whitehead duality}
Let $E$ be a finite $S^{1}$-CW complex embedded into $S(V)$, the unit sphere of an $S^{1}$-representation space $V$. Then, $(V^{+}\setminus E)$ and $\Sigma E$ are equivariant  $V$-dual to each other. Moreover, if $E' \subset E $ is an inclusion of $S^{1}$-CW complex, then the natural inclusion $\Sigma E'\rightarrow \Sigma E$ is dual to the natural inclusion $(V^{+}\setminus E)\rightarrow (V^{+}\setminus E')$. Similar results hold for the $Pin(2)$-equivariant case.
\end{lem}

Notice that
$$\bigvee^{(2m+1)^{2}}\limits_{S^{0}}\mathbb{C}^{+}\cong \Sigma \left(\coprod^{(2m+1)^{2}} S^1\right) $$
and we can embed the disjoint union of circles into $S(\mathbb{C}^2) $.
By Lemma \ref{spanier-whitehead duality}, we have
$$
         \Sigma^{\mathbb{C}^2}  \left(  \bigvee^{(2m+1)^{2}}\limits_{S^{0}}\mathbb{C}^{+} \right)^{*}
          =    (\mathbb{C}^{2})^{+}\setminus  ( \coprod^{(2m+1)^{2}} S^1),
$$
and we can obtain
$$
       \bar{I}^{n_{},-}_{m}
       \cong
       \Sigma^{-\mathbb{C}}  \left(   (\mathbb{C}^{2})^{+}  \setminus ( \coprod^{(2m+1)^2} S^{1} )    \right).
$$
Hence, we can conclude that
$$
      \underline{\text{\text{SWF}}}^{R}(Y,\mathfrak{s}_{0};S^{1})
      \cong
      \left(   (\mathbb{C}^{2})^{+} \setminus (\coprod\limits^{\infty} S^{1}),
                 0 ,  n(Y, \mathfrak{s}_{0},A_{0},g)+m(\slashed{D}_{A_{0}},\delta_{r}+\delta_{})+1   \right),
$$
where $(\mathbb{C}^{2})^{+}\setminus (\coprod\limits^{\infty} S^{1} )$ denotes the inverse system
 $$
(\mathbb{C}^{2})^{+}\setminus (S^1)\leftarrow (\mathbb{C}^{2})^{+}\setminus
(\coprod\limits^{2} S^{1})\leftarrow (\mathbb{C}^{2})^{+}\setminus (\coprod\limits^{3} S^{1})\leftarrow \cdots ,
$$
whose connecting
morphisms are given by the natural inclusions.

Since $\delta_{r},\delta_{}>0$ and $-\delta_{r}$ is an eigenvalue of $\slashed{D}$ with multiplicity $2$, we get
$m(\slashed{D}_{A_{0}}, \delta_{r}+\delta_{})=m(\slashed{D}_{A_{0}}, \delta_{})+2$, which implies
\begin{equation}\label{swfr for nil}
         \underline{\text{\text{SWF}}}^{R}(Y,\mathfrak{s}_{0};S^{1})
         \cong
         \left(  (\mathbb{C}^{2})^{+}\setminus (\coprod\limits^{\infty} S^{1} ) ,0,\frac{d-1}{8}   \right) \text{ for } d>0.
\end{equation}
\begin{rmk} For the case $d<0$, the results are
$$
               \underline{\text{\text{SWF}}}^{A}(Y,\mathfrak{s}_{0};S^{1})          \cong
               \left(  \bigvee^{\infty}\limits_{S^{0}} \mathbb{C}^{+}, 0, \frac{d-15}{8}    \right);
$$
$$
         \underline{\text{\text{SWF}}}^{R}(Y,\mathfrak{s}_{0};S^{1})
         \cong
         \left(  (\mathbb{C}^{2})^{+}\setminus (\coprod\limits^{\infty} S^{1} ) ,0,\frac{d+1}{8}   \right).
$$
\end{rmk}
We have finished the calculation of the $S^{1}$-invariants. Since $c_{1}(\mathfrak{s}_{0})=0$, the spin$^c$ structure $\mathfrak{s}_{0}$ can be lifted to $2^{2}=4$ different spin structures, whose spin connections are given by $A_0, A_0+\frac{h_{1}}{2}, A_0+\frac{h_{2}}{2}$ and $A_0+\frac{h_{1}+h_{2}}{2}$.  We denote the corresponding spin structures by $\mathfrak{s}^{0},\mathfrak{s}^{1},\mathfrak{s}^{2},\mathfrak{s}^{3}$ respectively. Although we have to choose the base connection to be the corresponding spin connection, for consistency, we still identify spin$^{c}$ connections with  $1$-forms by sending $A$ to $A-A_0$ instead.

Most of the argument in the $S^{1}$-equivariant case can be easily adapted to be $Pin(2)$-equivariant case. The only thing to be careful is that the set $\Lambda$ (see (\ref{deformation retract of the cube})) should be invariant under the additional $\jmath$-symmetry.

Now we start the calculation. It turns out that the invariants for $\mathfrak{s}^{1}, \mathfrak{s}^2$ and $\mathfrak{s}^3$ are isomorphic to each other, so we will just focus on $\mathfrak{s}^{1}$. With the same setup when computing $\underline{\text{\text{SWF}}}^{A}(Y,\mathfrak{s}_{0};S^{1})$, the addition $\jmath$-action is given by
 $$
 \jmath \cdot(h,\phi_{\vec{v},+})=(h_{1}-h, \phi_{(1,0)-\vec{v},-}) \text{ and  }\jmath \cdot(h,\phi_{\vec{v},-})=(h_{1}-h, -\phi_{(1,0)-\vec{v},+}).
$$
To preserve this symmetry, we consider
\begin{align*}
\hat{J}_{m}^{+} &=\widebar{Str}({\epsilon}')\cap p_\mathcal{H}^{-1}([-\frac{1}{4}-m+1,\frac{1}{4}+m]^{2}), \\
\hat W &= \bigoplus_{\vec{v}\in \mathbb{Z}^2 \cap [-m+1,m]^{2}}V_{\vec{v}},
\end{align*}
where we note that the basic interval is $[-m+1,m]$ instead of $[-m,m]$.

In the similar manner, we have that, for $n$ sufficiently large,
 $$\bar{I}^{n,-}_{m}(Pin(2))\cong \Sigma^{-\hat{W}^{-}}I_{Pin(2)}(\varphi, \inv (\hat{J}^{+}_{m}\cap (i\Omega^{1}_{h}(Y)\oplus \hat{W})),$$ where $\hat{W}^{-}$ is spanned by the negative eigenvector of $(\slashed{D}_{A_0+\frac{h_{1}}{2}}+\delta_{r}-\delta_{})|_{{W}}$. The set $\Lambda$ can be made $\jmath$-invariant by choosing $z_0 =\frac{h_{1}}{2}$  and requiring that $\jmath \cdot \gamma_{\vec{v}}=\gamma_{(1,0)-\vec{v}}$. Repeating the calculations, we can show that$$
I_{S^{1}}(\varphi, \inv (\hat{J}^{+}_{m}\cap (i\Omega^{1}_{h}(Y)\oplus \hat{W}))\cong \bigvee^{4m^{2}}\limits_{(\{0\}\times\mathbb{C}^{4m^{2}})^{+}}(\mathbb{C}^{4m^{2}+1})^{+}\cong  (\mathbb{C}^{4m^{2}})^{+}\wedge \Sigma (\coprod\limits^{4m^{2}}S^{1}).
$$
Notice that these copies of $S^{1}$ correspond to vertices $\vec{v}\in [-\frac{1}{4}-m+1,\frac{1}{4}+m]^{2}$, which are interchanged by $\jmath$. Therefore, we see that
$$
I_{Pin(2)}(\varphi, \inv (\hat{J}^{+}_{m}\cap (i\Omega^{1}_{h}(Y)\oplus \hat{W}))\cong (\mathbb{H}^{2m^{2}})^{+}\wedge \Sigma (\coprod\limits^{2m^{2}}Pin(2)).
$$
Since $\hat{W}^{-}\cong \mathbb{H}^{2m^{2}}$, we can conclude
$$
        \underline{\text{\text{SWF}}}^{A}(Y,\mathfrak{s}^{1};Pin(2))
       \cong
        \left(   \Sigma (\coprod\limits^{\infty}Pin(2)),0,\frac{d-17}{16} \right),
$$
where $\Sigma(\coprod\limits^{\infty}Pin(2)))\in \ob(\mathfrak{S}(Pin(2)))$ denotes  the direct system
$$
\Sigma Pin(2)\rightarrow \Sigma(Pin(2)\amalg Pin(2)) \rightarrow \Sigma(Pin(2)\amalg Pin(2)\amalg Pin(2))\rightarrow \cdots ,
$$
whose connecting morphisms are given by natural inclusions.

The calculation of $\underline{\text{\text{SWF}}}^{R}(Y,\mathfrak{s}^{1};Pin(2))$ is also very similar to the case  $\underline{\text{\text{SWF}}}^{R}(Y,\mathfrak{s}_{0};S^{1})$. For example, we can show that $\bar{I}^{n,-}_{m}(Pin(2))$ is the Spanier-Whitehead dual of $\Sigma (\coprod\limits^{2m^{2}}Pin(2))$.
We can pick an embedding of $\coprod\limits^{\infty} Pin(2) $ in $S(\mathbb{H}) $ and denote by $\mathbb{H}^{+}\setminus \coprod\limits^{\infty}Pin(2)$  the inverse system
$$
           \mathbb{H}^{+} \setminus Pin(2)                 \leftarrow
           \mathbb{H}^{+} \setminus \coprod^2 Pin(2) \leftarrow
           \mathbb{H}^{+} \setminus \coprod^3 Pin(2) \leftarrow \cdots,
$$
whose connecting morphisms are given by inclusions. Then, by Lemma~\ref{spanier-whitehead duality}, the system $\mathbb{H}^{+}\setminus \coprod\limits^{\infty}Pin(2)$ is equivariant $\mathbb{H}$-dual to $\Sigma(\coprod\limits^{\infty}Pin(2))$. Hence we get
$$
         \underline{\text{\text{SWF}}}^{R}(Y,\mathfrak{s}^{1};Pin(2))
         \cong
         \left(   \mathbb{H}^{+}\setminus \coprod\limits^{\infty}Pin(2) ,0,\frac{d-1}{16} \right).
$$

Now we compute the invariants for $\mathfrak{s}^{0}$. The corresponding spin connection is $A_0$ and the additional $\jmath$-symmetry is given by
$$
\jmath \cdot(h,\phi_{\vec{v},+})=(-h,\phi_{-\vec{v},-}) \text{ and }\jmath \cdot (h,\phi_{\vec{v},-})=(-h,-\phi_{-\vec{v},+}).
$$
We want to compute the Conley index
$$\bar{I}^{n,+}_{m}(Pin(2))\cong\Sigma^{-W^{-}}I_{Pin(2)}(\varphi,\inv (\hat{J}^{+}_{m}\cap (i\Omega^{1}_{h}(Y)\oplus W))).$$
Note that the set $\Lambda$ defined in (\ref{deformation retract of the cube}) can never be made $\jmath$-invariant in this case.
Instead, we can use the union of balls $B_{\vec{v}}$ and paths $\gamma_{\vec{v}}$ connecting $B_{\vec{0}}$ to  $B_{\vec{v}}$  for each $\vec{v}\neq 0$ with  $\jmath \cdot\gamma_{\vec{v}}=\gamma_{-\vec{v}}$. Using a deformation retract to this set, we can describe the Conley index $I_{Pin(2)}(\varphi,\inv (\tilde{J}^{+}_{m}\cap (i\Omega^{1}_{h}(Y)\oplus W)))$ as follows: the ball $B_{\vec{v}}$ contributes to a copy of $(\mathbb{H}^{2m^{2}+2m+1})^{+}\cong\Sigma^{(2m^{2}+2m)\mathbb{H}}\Sigma S(\mathbb{H})$. For $\vec v \neq 0 $, each pair of $B_{\pm \vec{v}}$ together contributes a copy of $\Sigma^{(2m^{2}+2m)\mathbb{H}}\Sigma (\tilde{Z}_2\times S(\mathbb{H}))$, where $\tilde{Z}_2 = \{\pm1 \}$ is the two-point space with nontrivial $Pin(2)$-action. More precisely, we can write
$$
       I_{Pin(2)}(\varphi,\inv (\bar{J}^{+}_{m}\cap (i\Omega^{1}_{h}(Y)\oplus W)))
         \cong
          \Sigma^{(2m^{2}+2m) \mathbb{H}}
          \Sigma
           \left(
            S(\mathbb{H}) \mathop{\vee}\limits_{ Pin(2)}
               \mathop{\bigvee}\limits_{Pin(2)}^{2m^{2}+2m}(\tilde{Z}_2 \times S(\mathbb{H}))
            \right).
$$
Here we think of $Pin(2)$ as the subset $\{(1,e^{i\theta})\}\cup \{(-1,je^{i\theta})\}$ in $\tilde{Z}_2\times S(\mathbb{H})$.
Since $W^{-}\cong \mathbb{H}^{2m^{2}+2m+1}$, we obtain
$$
        \underline{\text{\text{SWF}}}^{A}(Y,\mathfrak{s}^{0};Pin(2))
        \cong
        \left(
        \Sigma  \left(    S(\mathbb{H})\mathop{\vee}\limits_{ Pin(2)}
         \bigvee\limits_{Pin(2)}^{\infty}  (\tilde{Z}_2  \times S(\mathbb{H}))   \right),
          0,  \frac{d-9}{16}
        \right),
$$
where $S(\mathbb{H})\mathop{\vee}\limits_{ Pin(2)}(\bigvee\limits_{Pin(2)}^{\infty}(\tilde{Z}_2\times
S(\mathbb{H})))$ denotes the direct system
\[
     \begin{split}
     & S(\mathbb{H})\mathop{\vee}\limits_{Pin(2)}(\tilde{Z}_2\times S(\mathbb{H}))
        \rightarrow
      S(\mathbb{H})\mathop{\vee}\limits_{ Pin(2)} \left(\bigvee\limits_{Pin(2)}^{2}(\tilde{Z}_2  \times S(\mathbb{H})) \right)
       \rightarrow   \\
     & \qquad
       S(\mathbb{H})\mathop{\vee}\limits_{ Pin(2)} \left( \bigvee\limits_{Pin(2)}^{3}(\tilde{Z}_2\times S(\mathbb{H}))  \right)
       \rightarrow   \cdots.
     \end{split}
\]

We are only left with the calculation of $\underline{\text{\text{SWF}}}^{R}(Y,\mathfrak{s}^{0};Pin(2))$. To do this, we need to find the Spanier-Whitehead dual of $S(\mathbb{H})\mathop{\vee}\limits_{ Pin(2)}(\bigvee\limits_{Pin(2)}^{m}(\tilde{Z}_2\times
S(\mathbb{H})))$. It is not hard to check that this space can be embedded into $S(\mathbb{H}^{2})$ as
$$
D_{m}:=\bigcup\limits_{n=0}^{m}\{(z_{1}+jz_{2},z_{3}+jz_{4})\in S(\mathbb{H}^{2})|z_{3}=-\bar{z}_{4}=nz_{1} \text{ or } z_{3}=\bar{z}_{4}=n\bar{z}_{2} \}.
$$
Repeating the calculation we did for $\underline{\text{\text{SWF}}}^{R}(Y,\mathfrak{s}^{0};S^{1})$, we get
$$
    \underline{\text{\text{SWF}}}^{R}(Y,\mathfrak{s}^{0};Pin(2))
    \cong
    \left( (\mathbb{H}^{2})^{+}\setminus D_{\infty},0,\frac{d+7}{16}  \right),
$$
where $(\mathbb{H}^{2})^{+}\setminus D_{\infty}$ denotes  the inverse system
$$
(\mathbb{H}^{2})^{+}\setminus D_{1}\leftarrow (\mathbb{H}^{2})^{+}\setminus D_{2}\leftarrow (\mathbb{H}^{2})^{+}\setminus D_{3}\leftarrow \cdots .
$$

\subsubsection{Other nil manifolds} Suppose that $\Sigma$ is not smooth. Then the genus of $\Sigma$ is $0$ and $Y$ is a rational homology sphere. Without any further perturbation, the functional $\mathcal{L}_{\delta_{r},0}$ has a unique critical point $(A_{0},0)$. This allows us to apply Theorem~\ref{Morse-Bott reducible} and obtain
\begin{align*}
\underline{\text{\text{SWF}}}^{A}(Y,\mathfrak{s};S^{1}) &\cong(S^{0},0,n(Y,\mathfrak{s},A_{0},g)+m(\slashed{D},\delta_{r} )), \\
\underline{\text{\text{SWF}}}^{R}(Y,\mathfrak{s};S^{1}) &\cong(S^{0},0,n(Y,\mathfrak{s},A_{0},g)+m(\slashed{D},\delta_{r})).
\end{align*}
Moreover, when $\mathfrak{s}$ is spin, we have
\begin{align*}
      \underline{\text{\text{SWF}}}^{A}(Y,\mathfrak{s};Pin(2)) &
      \cong
      \left(  S^{0},0,\frac{n(Y,\mathfrak{s},A_{0},g)+m(\slashed{D},\delta_{r} )}{2}  \right), \\
       \underline{\text{\text{SWF}}}^{R}(Y,\mathfrak{s};Pin(2)) &
       \cong
      \left(   S^{0},0,\frac{n(Y,\mathfrak{s},A_{0},g)+m(\slashed{D},\delta_{r})}{2}  \right).
\end{align*}
 The explicit formula for $n(Y,\mathfrak{s},A_{0},g)+m(\slashed{D},\delta_{r})$ can be obtained in the same fashion as in Section \ref{subsection large degree circle bundles} (cf. \cite{Nicolescu1}, \cite{Nicolescu2}) and we omit it.
\subsection{Flat manifolds except $T^3$}\label{subsection flat manifolds} In this subsection, we calculate the spectrum invariants for manifolds $Y$ supporting a flat metric $g$ other than the 3-torus $T^3$. There are five manifolds belonging to this class: four of them are $T^{2}$-bundles
over $S^{1}$ with monodromy automorphism fixing a point and having orders $2,3,4,6$, and the last of them is the Hantzsche-Wendt manifold. By the Weitzenb$\ddot{\text{o}}$ck formula, for any torsion spin$^{c}$ structure $\mathfrak{s}$ on $Y$, the functional $CSD$ has only reducible critical points.

The Hantzsche-Wendt manifold is a rational homology sphere. Therefore, the functional $\mathcal{L}_{0,0}$ has only one critical point $(A_{0},0)$ and we can just apply Theorem~\ref{Morse-Bott reducible} to conclude that
\begin{align*}
\underline{\text{\text{SWF}}}^{A}(Y,\mathfrak{s};S^{1}) &\cong(S^{0},0,n(Y,\mathfrak{s},A_{0})), \\
\underline{\text{\text{SWF}}}^{R}(Y,\mathfrak{s};S^{1}) &\cong(S^{0},0,n(Y,\mathfrak{s},A_{0}))
\end{align*}
as well as analogous results for the spin case.
The numbers $n(Y,\mathfrak{s},A_{0},g)$ can be calculated using the method of \cite{Nicolescu1} and \cite{Nicolescu2} and we omit the result.

Now we consider the  $T^{2}$-bundles over $S^{1}$ whose monodromies are automorphisms $\tau \colon T^{2}\rightarrow T^{2}$ of order 2 (i.e. the hyperelliptic involution on $T^{2}$). The situations for the cases of order 3,4 or 6 are very similar, so we will focus our attention to this case of order 2.

Let $T^{2}$ be given by $\mathbb{R}^{2}/\mathbb{Z}^{2}$ with $\tau(\theta_{1},\theta_{2})=-(\theta_{1},\theta_{2})$.
Then, the manifold $Y$ can be obtained as the quotient of $(\mathbb{R}\times T^{2})/\mathbb{Z}$, where the $\mathbb{Z}$-action is given by $(\theta_{0},\theta_{1},\theta_{2}) \mapsto (\theta_{0}+1,-\theta_{1},-\theta_{2})$. Since $H_{1}(Y;\mathbb{Z})=(\mathbb{Z}/2\mathbb{Z})\oplus (\mathbb{Z}/2\mathbb{Z})\oplus \mathbb{Z}$, there are four torsion spin$^{c}$ structures on $Y$.
By \cite[Lemma~37.4.1]{Kronheimer-Mrowka}, only one of them admits a spin$^{c}$ connection $A$ with $F_{A^{t}}=0$ and $\ker \slashed{D}_{A}\neq 0$. We denote it by $\mathfrak{s}_{0}$ and the other three by $\mathfrak{s}_{1},\mathfrak{s}_{2},\mathfrak{s}_{3}$.

Let us consider the spin$^c$ structure $\mathfrak{s}_{0}$ first. This spin$^{c}$ structure can be identified with a quotient of the spin$^{c}$ structure $\tilde{\mathfrak{s}}_{0}$ on $\mathbb{R}\times T^{2}$ whose spinor bundle is trivial and the Clifford multiplication is given by
$$
\rho(d\theta_{0})=\left(\begin{array} {lr}
 i & 0 \\
 0& -i
\end{array}\right)\text{, }\rho(d\theta_{1})=\left(\begin{array} {lr}
 0 & -1 \\
 1 & 0
\end{array}\right)\text{, }\rho(d\theta_{2})=\left(\begin{array} {lr}
 0 & i \\
 i & 0
\end{array}\right).
$$
Note that the generator of the $\mathbb{Z}$-action on the spinor bundle is given by the constant matrix $\left(\begin{smallmatrix}
  1 & 0 \\
  0& -1
 \end{smallmatrix}\right)$.

As in Section \ref{subsection nil manifolds}, we let $\Gamma_{c}(S_{Y})$ (resp. $\Gamma_{0}(S_{Y})$) be the space of sections of that is constant (resp. integrate to $0$) along each fiber of the $T^2$-bundle.
For each integer $n$, we define sections $\phi_{n,\pm}\in\Gamma_{c}(S_{Y})$ as
$$
\phi_{n,+}(\theta_{0},\theta_{1},\theta_{2})=(e^{\pi i \cdot 2n\theta_{0}},0), \; \phi_{n,-}(\theta_{0},\theta_{1},\theta_{2})=(0,e^{\pi i\cdot (2n+1)\theta_{0}}).
$$
These give an $L^2$-orthonormal basis of $\Gamma_{c}(S_{Y})$.

 We choose as the base connection $A_{0}$ induced from the trivial connection on $\tilde{\mathfrak{s}}_{0}$ and we pick $ h_1 = 2\pi i\cdot d\theta_{0}$ as a basis for $i \Omega^1_h (Y) $. We have the following observation.
\begin{lem}
For any $\theta\in \mathbb{R}$, the kernel of the operator $\slashed{D}_{A_{0}+ \theta h_{1}}|_{L^{2}_{k}(\Gamma_{0}(S_{Y}))}$ is trivial. Moreover, we have
\begin{align*}
 \slashed{D}_{A_{0}+\theta h_{1}}(\phi_{n,+})&=-\pi(2n+2\theta)\phi_{n,+}, \\
 \slashed{D}_{A_{0}+\theta h_{1}}(\phi_{n,-})&=\pi(2n+1+2\theta)\phi_{n,-}.
\end{align*}
\end{lem}
 \begin{proof}
 The first assertion can be proved by passing to the double cover of $Y$, which is $T^{3}$. The second assertion is easy to verify by direct calculation.
 \end{proof}
Let $f_H \colon i\Omega^{1}_{h}(Y)\cong \mathbb{R}\rightarrow \mathbb{R}$ be given by $f(\theta) =- \cos(4\pi\theta)$. 
By the same argument as in the proof of \cite[Proposition~37.1.1]{Kronheimer-Mrowka} and Lemma~\ref{achieve positive definite condition}, we have
 \begin{lem}
 There exist a constant $\delta_{1}\in (0,\frac{\pi}{2})$ and a function $\epsilon_{1}:(0,\delta_{1})\rightarrow (0,\infty)$ such that for any $\delta\in (0,\delta_{1})$ and $\epsilon\in (0,\epsilon_{1}(\delta_{1}))$, we have\begin{enumerate}
\item   The functional
$$\mathcal{L}_{-\delta,\epsilon}=CSD|_{Coul(Y)}-\frac{\delta}{2}\|\phi\|_{L^{2}}^{2}+{\epsilon}f_H(p_\mathcal{H}(a,\phi))$$
has only reducible critical points.
\item For any $h\in i\Omega^{1}_{h}(Y)$ and $\phi\in L^{2}_{k}(\Gamma(S_{Y}))$, we have
$$
\langle (2(\slashed{D}_{A_{0}+h}-\delta)^{2}+\rho( \grad {\epsilon}f_H(h)))\phi,\phi\rangle_{L^{2}}\geq C(\delta,\epsilon)\|\phi\|^{2}_{L^{2}},$$
where $C(\delta,\epsilon)$ is a positive constant depending only on $\delta,\epsilon$.
\end{enumerate}
\end{lem}
Let us fix a choice of $\delta_{2}\in (0,\delta_{1})$ and $\epsilon_{2}\in (0,\epsilon_{1}(\delta_{2}))$. The the critical points of $\mathcal{L}_{-\delta_{2},\epsilon_{2}}$ are just $(\theta h_{1},0)$ with $4\theta\in \mathbb{Z}$ and there are four gauge equivalent classes $[(0,0)],[(\frac{h_{1}}{4},0)],[(\frac{h_{1}}{2},0)]$ and $[(\frac{3h_{1}}{4},0)]$. Notice that the spectral flow of the operator $\slashed{D}_{A_{0}+\theta h_{1}}-\delta_{2}$ is $0$ when $\theta$ goes from $0$ to $\frac{1}{4}$ or $\theta$ goes from $\frac{1}{4}$ to $\frac{1}{2}$, whereas the spectral flow is $1$ when $\theta$ goes from $\frac{1}{2}$ to $\frac{3}{4}$. Consequently, we have
\begin{align*}
\grade ([(0,0)],[(\frac{h_{1}}{4},0)])&=\grade ([(\frac{h_{1}}{2},0)],[(\frac{h_{1}}{4},0)]) \, \;=-1, \\
\grade ([(0,0)],[(\frac{3h_{1}}{4},0)])&=\grade ([(\frac{h_{1}}{2},0)],[(\frac{3h_{1}}{4},0)])=1.
\end{align*}

As in the proof of Lemma~\ref{achieve Smale transversality}, we can find an extended cylinder function $\bar{f}$ satisfying the following requirements:\begin{enumerate}
\item There exist $\epsilon'>0$ such that $\bar{f}(a,\phi)=0$ whenever $\|\phi\|_{L^{2}}\geq \epsilon'$;
\item The functional $\mathcal{L}_{-\delta_{2},\epsilon_{2}}+\bar{f}$ has only reducible critical points;
\item For any boundary unstable reducible critical manifold $[C]$ and any boundary stable reducible critical manifold $[C']$, the moduli space $\breve{\mathcal{M}}_{\bar{f}}([C],[C'])$ is Smale regular.
\end{enumerate}

We now consider the perturbed functional $\mathcal{L}_{-\delta_{2},\epsilon_{2}}+\bar{f}$.
Let $[C],[C']$ be critical manifolds whose corresponding critical points are $[\mathfrak{a}],[\mathfrak{b}]$ respectively. By a formula analogous to (\ref{expected dimension}), we see that the expected dimension of  $\breve{\mathcal{M}}_{\bar{f}}([C],[C'])$ is nonnegative only if $\grade ([\mathfrak{a}],[\mathfrak{b}])$ is at least 2. This can happen only when $[\mathfrak{a}]=[(\frac{h_{1}}{4},0)]$ and $[\mathfrak{b}]=[(\frac{3h_{1}}{4},0)]$. However, since the functional $\mathcal{L}_{-\delta_{2},\epsilon_{2}}+\bar{f}$ takes the same value at these two points, there cannot be any trajectory connecting them. As a result, we have proved that the moduli space $\breve{\mathcal{M}}_{\bar{f}}([C],[C'])$ is actually empty. This implies that there are no irreducible trajectories for the functional $\mathcal{L}_{-\delta_{2},\epsilon_{2}}+\bar{f}$ and we can use Theorem~\ref{spectrum for the linearized flow} to calculate our invariants.

The computation of the Conley indices of the linearized Seiberg-Witten flow is exactly the same as in the case of nil manifolds and we will skip all the details.
The number $n(Y,\mathfrak{s},A_{0},g)$ can be calculated using the formula in \cite{Nicolescu1}. Since $Y$ admits an orientation reversing diffeomorphism preserving $(\mathfrak{s}_{0},A_{0})$, we see that $\eta(\slashed{D})=\eta_{sign}=0$, which implies $n(Y,g,\mathfrak{s},A_{0})=-\frac{1}{2}\text{dim}_{\mathbb{C}}(\ker \slashed{D})=-\frac{1}{2}$. Note that $m(\slashed{D},-\delta_{2})=0$
as $\delta_{2}$ is small. Therefore, we can conclude that
\begin{align*}
     \underline{\text{\text{SWF}}}^{A}(Y,\mathfrak{s}_{0};S^{1})
   &\cong
     \left(  \bigvee^{\infty}\limits_{S^{0}}\mathbb{C}^{+},  0,  \frac{1}{2}  \right), \\
    \underline{\text{\text{SWF}}}^{R}(Y,\mathfrak{s}_{0};S^{1})
   &\cong
    \left(  (\mathbb{C}^{2})^{+}\setminus (\coprod\limits^{\infty}  S^{1} ) ,0,\frac{3}{2}  \right).
\end{align*}

As for the $Pin(2)$-invariants, notice that $\mathfrak{s}_{0}$ can be lifted to two spin structures. Since the holonomy of $A_{0}^{t}$ along the loop $(0,0,\theta) $ equals $-1$, we see that the spin connections are $A_{0}+\frac{h_{1}}{4}$ and $A_{0}+\frac{h_{3}}{4}$ and denote the corresponding spin structures by $\mathfrak{s}^{0}_0$ and $\mathfrak{s}^{1}_0$ respectively. We have
\begin{align*}
     \underline{\text{\text{SWF}}}^{A}(Y,\mathfrak{s}^{0}_0;Pin(2))
     &\cong
      \left(
        \Sigma
         \left(
          S(\mathbb{H})\mathop{\vee}\limits_{ Pin(2)} \bigvee_{Pin(2)}^{\infty} (  \tilde{Z}_2  \times  S(\mathbb{H})  )
         \right),
         0,  \frac{3}{4}
       \right), \\
       \underline{\text{\text{SWF}}}^{R}(Y,\mathfrak{s}^{0}_0;Pin(2))
        &\cong
       \left(  (\mathbb{H}^{2})^{+}\setminus D_{\infty},0,\frac{5}{4}   \right),\\
        \underline{\text{\text{SWF}}}^{A}(Y,\mathfrak{s}^{1}_0;Pin(2))
        &\cong
        \left(  \Sigma (\coprod^{\infty}Pin(2)),0,\frac{1}{4}  \right), \\
        \underline{\text{\text{SWF}}}^{R}(Y,\mathfrak{s}^{1}_0;Pin(2))
         &\cong
         \left(  \mathbb{H}^{+}\setminus \coprod^{\infty}Pin(2),  0,   \frac{3}{4}  \right).
\end{align*}

 As for the other spin$^{c}$ structures $\mathfrak{s}_{1},\mathfrak{s}_{2}$ and $\mathfrak{s}_{3}$, since the $\ker (\slashed{D}_{A_{0}+h})=0$ for any $h\in i\Omega^{1}_{h}(Y)$, we can just apply Theorem~\ref{Morse-Bott reducible} and get sphere spectrums with suitable suspension. Notice, for $j=1,2,3$, that we have $n(Y,A_{0},\mathfrak{s}_{j},g)=0$ because of existence of an orientation reversing diffeomorphism preserving $(\mathfrak{s}_{j},A_{0})$.

Finally, when $Y$ is one of the other $T^2$-bundles, we can prove that the spectrum invariants of $Y$ are just shifts in the  suspension indices of the above results. The only {difference} comes from the change of the number $n(Y,\mathfrak{s},A_{0},g)$. Again, we refer to \cite{Nicolescu1} and \cite{Nicolescu2} for the calculation of this quantity.

\bibliographystyle{amsplain}
\bibliography{Bbib418}
\end{document}